%% file: main.tex
\documentclass[10pt]{article}
\usepackage[british]{babel}
\usepackage{amsmath}
\usepackage{amssymb}
\usepackage{amsthm}
\usepackage{mathtools}
\usepackage{hyperref}
\usepackage{bookmark}
\usepackage{makecell}
\usepackage{mathpartir}
\usepackage{csquotes}
\usepackage[sortcites=true,date=year,url=false,isbn=false]{biblatex}
\usepackage{tikz}
\usepackage{quiver}
\usepackage{tabularx}
\usepackage{xspace}
\usepackage{multirow}
\usepackage{adjustbox}
\usepackage{listings}
\usepackage[shortcuts]{extdash}
\usetikzlibrary{nfold,positioning,cd,arrows.meta}
\newcolumntype{C}{>{\centering\arraybackslash}X}

\usepackage[format=plain,
            labelfont={it},
            textfont=it]{caption}

\def\cdot{\text{\scriptsize\textbullet}}

\theoremstyle{plain}
\newtheorem{theorem}{Theorem}[section]

\newtheorem{lemma}[theorem]{Lemma}
\newtheorem{proposition}[theorem]{Proposition}

\theoremstyle{definition}
\newtheorem{definition}[theorem]{Definition}

\newtheorem{example}[theorem]{Example}

\theoremstyle{remark}

\makeatletter%
\NewDocumentCommand{\mpr@label}{ m }{%
  \def\@currentlabelname{#1}%
  \phantomsection %
}%
\DeclareDocumentCommand{\inferdef}{m m m}{%
   \inferrule*[vcenter,right=#1]{#2}{#3}%
   \mpr@label{\textsc{#1}}%
}%
\makeatother%
\newcommand{\ruleref}[1]{Rule \nameref{#1}}%

\usepackage{macros}
\usepackage{etoolbox}
\makeatletter
\def\calign@preamble{%
   &\hfil\strut@
    \setboxz@h{\@lign$\m@th\displaystyle{##}$}%
    \ifmeasuring@\savefieldlength@\fi
    \set@field
    \hfil
    \tabskip\alignsep@
}
\let\cmeasure@\measure@
\patchcmd\cmeasure@{\divide\@tempcntb\tw@}{}{}{}
\patchcmd\cmeasure@{\divide\@tempcntb\tw@}{}{}{}
\patchcmd\cmeasure@{\ifodd\maxfields@
  \global\advance\maxfields@\@ne
  \fi}{}{}{}
\newenvironment{calign}
{%
  \let\align@preamble\calign@preamble
  \let\measure@\cmeasure@
  \align
}
{%
  \endalign
}
\makeatother
\makeatletter
\def\cleartheorem#1{%
    \expandafter\let\csname#1\endcsname\relax
    \expandafter\let\csname c@#1\endcsname\relax
}
\makeatother

\def\sm{\mbox{{\footnotesize--}}}
\def\sp{\mbox{{\footnotesize+}}}
\def\sm{\text{--}}
\def\sp{\text{+}}

\lstdefinestyle{cattstyle}{
  basicstyle={\ttfamily},
  keywordstyle={\color{black}\ttfamily\bfseries},
}

\lstdefinelanguage{catt}{
  sensitive=false,
  keywords={catt,coh,let,check,benchmark},
  otherkeywords={@}
}

\renewcommand\prescript[2]{\,\ensuremath{{}^{#1}_{#2}}}
\renewcommand\paragraph[1]{

\vspace{5pt}
\noindent
\emph{#1.}}
\tikzset{Rightarrow/.style={double equal sign distance,>={Implies},->},
triple/.style={-,preaction={draw,Rightarrow}},
quadruple/.style={preaction={draw,Rightarrow,shorten >=0pt},shorten >=1pt,-,double,double
distance=0.2pt}}

\begin{document}

\title{
    \textbf{Naturality for Higher-Dimensional\\ Path Types}
}

\newcommand{\email}[1]{\texttt{\href{mailto:#1}{#1}}}
\author{
  Thibaut Benjamin\footnote{\email{thibaut.benjamin@cl.cam.ac.uk}}\and
  Ioannis Markakis\footnote{\email{ioannis.markakis@cl.cam.ac.uk}}\and
  Wilfred Offord\footnote{\email{wgbo2@cam.ac.uk}}\and
  Chiara Sarti\footnote{\email{cs2197@cam.ac.uk}}\and
  Jamie Vicary\footnote{\email{jamie.vicary@cl.cam.ac.uk}} \and \\
  Department of Computer Science and Technology,\\
  University of Cambridge,
  Cambridge, UK
  }

\maketitle

\thispagestyle{plain}
\pagestyle{plain}

\begin{abstract}

\noindent We define a naturality construction for the operations of weak
$\omega$\=/categories, as a meta-operation in a dependent type theory.
Our construction has a geometrical motivation as a local tensor product with a directed interval, and behaves logically as a globular analogue of Reynolds parametricity. Our construction operates as a ``power tool'' to support construction of terms with geometrical structure,
and we use it to define composition operations for cylinders and
cones in $\omega$\=/categories. The machinery can generate terms of high complexity,
and we have implemented our construction in a proof assistant, which
verifies that the generated terms have the correct type. All our
results can be exported to homotopy type theory, allowing the explicit computation of complex path type inhabitants. \newline

\noindent\emph{Keywords:} dependent type theory, identity types, higher categories, naturality

\end{abstract}

\input{sections/introduction}

\input{sections/acknowledgements}
\input{sections/catt}

\input{sections/naturality}

\input{sections/cylinders}
\input{sections/cones}
\input{sections/implementation}
\printbibliography[]

\clearpage

\appendix

\input{appendix/composites}
\input{appendix/cones}

\end{document}

%% file: sections/introduction.tex

\section{Introduction}\label{sec:introduction}
\tikzset{tikzpic/.style={font=\scriptsize, inner sep=2pt, scale=1.1}}

\paragraph{Overview} Higher categories are algebraic tools that describe composites and coherences
of $n$\=/cells in all dimensions, which today find wide application in physics~\cite{bartlett_modular_2015, schommer_classification_2011, baez_higherdimensional_1995}, geometry~\cite{lurie_derived_2004, lurie_higher_2009} and type theory~\cite{lumsdaine_weak_2009, vandenberg_types_2011,
  altenkirch_syntactical_2012}.
For applications in computer science, these $n$\=/cells may  represent inhabitants of identity types in Martin-L\"of type theory, paths in a space, or rewriting relations~\cite{seely_modelling_1987, mimram_3dimensional_2014}.

In this article we work with the dependent type theory \(\catt\) due to Finster and Mimram~\cite{finster_typetheoretical_2017}, models of which are \textit{weak globular $\omega$\=/categories}.\footnote{These models are known to be equivalent to the $\omega$\=/categories of Grothendieck-Maltsiniotis with an adequate choice of coherator~\cite{benjamin_globular_2024}, and also correspond to the $\omega$\=/categories of Batanin-Leinster~\cite{ara_infty_2010,bourke_iterated_2020} defined via contractible operads.} The theory \catt is an attractive foundation for formal work on $\omega$\=/categories, as it has a relatively simple definition, which is highly amenable to proof by induction. We use a type-theoretic notation to denote an $n$\-cell $a$, writing $a: \obj$ for a 0\=/cell (that is, an object), and for $n>0$ writing \(a : \partial^{-}a \to \partial^{+}a\) to denote the source and target $(n\!-\!1)$\=/cells, which must  have equal type. We also write $\partial_i^\pm a$ for the source or target boundary of dimension $i$.  An $\omega$\-category has a rich compositional structure; in particular, given an $n$\-cell $a$ and $m$\-cell $b$, and a chosen composition direction $0 \leq k < \min(n,m)$ with $\partial_k^+ a = \partial_k ^- b$, we may form their \textit{$k$\=/composite} $a \s_k b$.

Weak higher categories contain  coherence cells, such as \textit{associators} $\alpha_{f,g,h}: (f \s_0 g) \s_0 h \to f \s_0 (g \s_0 h)$ and \textit{unitors} $\lambda_f : \id \s_0 f \to f$, which are not in general equal to the identity. Correct handling of these coherences  can become  intricate as the dimension increases, and construction of  large proofs can therefore become intractable,
underscoring the need for more powerful tools for term construction and manipulation.

The goal of the present work is to describe a new meta-operation on $\omega$\=/categories, called the \textit{naturality construction},  producing cells which can be regarded as witnesses for the naturality of certain $\omega$\-category operations. The constructed cells have a straightforward type, which makes them easy to incorporate into larger proofs. However, the witnesses themselves can be highly complex, and in many cases intractable  to compute  manually.

\tikzset{twoarr/.style={-{Implies}, double distance=1.5pt, shorten <=2pt, shorten >=2pt}}
In the theory \catt, contexts comprise lists of free variables, each considered an $n$\=/cell with an assigned type. A valid judgement $\Gamma \vdash t : A$ then represents an $n$\=/cell $t : A$  in the free $\omega$\=/category generated by the variables of $\Gamma$. We may visualise an example context as follows:
\begin{align*}
\Gamma \quad&= \quad
\begin{aligned}
\begin{tikzpicture}[tikzpic]
\node (x) at (0,0) {$x$};
\node (y) at (1,0) {$y$};
\node (z) at (2,0) {$z$};
\node (u) at (3,0) {$u$};
\node (v) at (4,0) {$v$};
\draw [->] (x) to [bend left=55] node (f) [above] {$f$} (y);
\draw [->] (x) to [bend right=55] node (g) [below] {$g$} (y);
\draw [->] (y) to node [above] {$h$} (z);
\draw [->] (z) to node [above] {$j$} (u);
\draw [->] (u) to node [above] {$k$} (v);
\draw [-{Implies}, double distance=1.5pt, shorten <=2pt, shorten >=2pt] (.5,.3) to node [right, inner sep=2pt] {$a$} (.5,-.3);
\end{tikzpicture}
\end{aligned}
\intertext{Our construction requires a chosen subset $X \subseteq \Var(\Gamma)$
of the free variables of $\Gamma$,  with the property that $X$ is \emph{up-closed}, meaning that if any $x \in X$ appears in the source or target of some  variable $y$, then we also have  $y \in X$. In this case, an example of an up-closed set is $X = \{h, z, j, k\}$. We then construct a new context $\Gamma \uparrow X$ in which the variables of $X$ have been duplicated, and filler variables appropriately inserted, as follows:}
\Gamma \uparrow X
\quad &= \quad
\begin{aligned}
\begin{tikzpicture}[tikzpic]
\node (x) at (0,0) {$x$};
\node (y) at (1,0) {$y$};
\node (z0) at (2,.5) {$z^-$};
\node (z1) at (2,-.5) {$z^+$};
\node (u) at (3,0) {$u$};
\node (v) at (4,0) {$v$};
\draw [->] (x) to [bend left=55] node (f) [above] {$f$} (y);
\draw [->] (x) to [bend right=55] node (g) [below] {$g$} (y);
\draw [->] (y) to node [above, pos=0.4] {$h^-$} (z0);
\draw [->] (z0) to node [above, pos=0.6] {$j^-$} (u);
\draw [->] (y) to node [below, pos=0.4] {$h^+$} (z1);
\draw [->] (z1) to node [below, pos=0.6] {$j^+$} (u);
\draw [->] (u) to [bend left=55] node (k0) [above] {$k^-$} (v);
\draw [->] (u) to [bend right=55] node (k1) [below] {$k^+$} (v);
\draw [twoarr] (.5,.3) to node [right, inner sep=2pt] {$a$} (.5,-.3);
\draw [twoarr] (3.5,.3) to node [right, inner sep=2pt, font=\scriptsize] {$\fun k$} (3.5,-.3);
\draw [->] (z0) to node [right, inner sep=0pt, pos=0.4] {$\fun z$} (z1);
\draw [twoarr] (1.6,.3) to node [right, inner sep=2pt] {$\fun h$} +(-0,-.6);
\draw [twoarr] (2.4,.3) to node [right, inner sep=2pt] {$\fun j$} +(-0,-.6);
\end{tikzpicture}
\end{aligned}
\end{align*}
This sketch makes clear the necessity of the up-closure property: for example, since we are duplicating $z$ to yield $\fun z : z^- \to z^+$, the variable $h:y \to z$ must also be modified, since $z$ no longer exists in $\Gamma \uparrow X$.

There are an obvious pair of substitutions $\inj^-$, $\inj^+$ that map variables of $\Gamma$ to variables of $\Gamma \uparrow X$, mapping to the upper or lower duplicates respectively.\footnote{That is, we have a pair of substitutions $\Gamma \uparrow X \vdash \inj^\pm: \Gamma$, acting on any $x \in \Var(\Gamma)$ as $\inj ^\pm(x) = x ^\pm$ if $x \in X$, and $\inj^\pm(x) = x$ otherwise.} Given some valid term \mbox{$\Gamma \vdash t : A$}, we may apply these substitutions to yield a pair of terms $t[\inj^-]$, $t[\inj^+]$ valid in $\Gamma \uparrow X$. A natural question then arises:
\emph{can we  find a term $t \uparrow X$ valid in $\Gamma \uparrow X$,  which is suitably bounded by the injected terms $t[\inj ^\pm]$?}

This procedure has an intuitive geometrical motivation, where we consider $\Gamma \uparrow X$ as a ``local tensor product'' of $X \subseteq \Var(\Gamma)$ with the directed interval. In the case $X = \Var(\Gamma)$, this
construction should yield a path object construction  for weak  $\omega$\=/categories, a result which has been long-sought as the missing component for  a model structure on weak  $\omega$\=/categories~\cite{lanari_globular_2020,
henry_algebraic_2016}, analogous to the folk model structure on strict $\omega$\=/categories~\cite{lafont_folk_2010}. It would also likely be a major advance towards  the definition of a weak $\omega$\=/category of weak $\omega$\=/categories, for which the tensor product is expected to give a monoidal structure by analogy to the strict case~\cite{ara_folk_2020}.%

Our main result takes a step towards this  construction, showing how to build a term $t \uparrow X$ valid in $\Gamma \uparrow X$, under certain conditions. First, some brief definitions. For a variable $x \in \Var(t)$, the \emph{depth of $x$ in $t$} is given by $\dim(t) - \dim(x)$. Furthermore, for a context $\Gamma$, its \emph{boundary poset} has  elements $\Var(\Gamma)$, and an edge \mbox{$x < y$} when $x$ appears in some boundary of $y$; the \emph{depth of $x$ in $\Gamma$} is then defined as the length of the longest upwards chain starting at $x$, a notion which we illustrate in Fig.~\ref{fig:contextdepth}.
Our conditions can now be given.
\begin{itemize}
\item Every $x \in X \cap \Var(t)$ must have depth at most 1 in $t$.
\item Every $x \in X$ must have depth at most $1$ in $\Gamma$.
\end{itemize}
We will see that in the depth\=/1 case, the term $t \uparrow X$  can be interpreted as a {naturality witness} for the operation $t$. This allows the description of a wide range of interesting geometric and algebraic phenomena, some of which have previously been out of reach. The depth\=/0 case, which our construction subsumes, can be understood as a  functoriality construction, which has already been described~\cite{benjamin_type_2020}. The higher-depth cases, which are beyond our present work, would likely govern the behaviour of weak higher transfors~\cite{crans_localizations_2003}.

\begin{figure}[t]
\[
\begin{aligned}
\begin{tikzpicture}[tikzpic, scale=1.5, font=\normalsize]
\node (x) at (0,0) {3};
\node (y) at (1,0) {3};
\node (z) at (2,0) {1};
\draw [->] (x) to [bend left=65, looseness=1.2] node (f) [above] {2} (y);
\draw [->] (x) to [bend right=65, looseness=1.2] node (g) [below] {2} (y);
\draw [->] (y) to node [above] {0} (z);
\draw [-{Implies}, double distance=1.5pt, shorten <=2pt, shorten >=2pt] (.6,.35) to [bend left] node [right, inner sep=2pt] {1} +(0,-.7);
\draw [-{Implies}, double distance=1.5pt, shorten <=2pt, shorten >=2pt] (.4,.35) to [bend right] node [left, inner sep=2pt] {1} +(0,-.7);
\node (w) at (3,0) {2};
\node (u) at (4,0) {2};
\draw [->] (w) to [bend left=75, looseness=1.2] node (f) [above, pos=0.5] {1} (u);
\draw [->] (w) to [bend right=75, looseness=1.2] node (f) [below, pos=0.5] {1} (u);
\draw [->] (w) to [] node (f) [above, pos=0.25] {1} (u);
\draw [->] (z) to [] node (f) [above] {0} (w);
\draw [triple] (.4,0) to node [above=1pt, pos=0.45] {0} +(.2,0);
\draw [-{Implies}, double distance=1.5pt, shorten <=2pt, shorten >=2pt] (3.5,.4) to [] node [right, inner sep=2pt, pos=0.4] {0} +(0,-.4);
\draw [-{Implies}, double distance=1.5pt, shorten <=2pt, shorten >=2pt] (3.5,-.02) to [] node [right, inner sep=2pt, pos=0.4] {0} +(0,-.4);
\end{tikzpicture}
\end{aligned}
\]
\caption{\label{fig:contextdepth}Depths of variables in a context.}
\end{figure}

We have implemented our naturality construction as a meta-operation in an existing proof assistant for \catt, and all of our examples have been demonstrated in the implementation, which we present in Section~\ref{sec:implementation}. Our implementation verifies that the generated terms have the expected type, which is a significant check on the correctness of our results, since some of the generated terms are of high complexity, while the corresponding types take a much simpler form. Our naturality construction  is conceptually similar to \textit{parametricity translations},
studied in particular by Bernady et al.~\cite{bernardy_proofs_2012} to internalise Reynolds parametricity~\cite{reynolds_types_1983}.

Terms in the type theory \catt can be exported as functions that compute inhabitants of identity types in homotopy type theory (\hott)~\cite{benjamin_generating_2024}. Thus all of our examples also immediately apply to  higher dimensional path types. Defining our construction in \catt, rather than in \hott directly, has a number of advantages. Since we are doing meta-programming, our technical development is considerably simplified, as \hott is a far richer language. We also obtain greater generality, since our setting is directed and fully weak.

\paragraph{Examples} We now give a  range of examples that illustrate the scope of our construction. In each case, we suppose we have some term $t$ valid in $\Gamma$, and we describe $t \uparrow X$ as a term valid in $\Gamma \uparrow X$. For each example we indicate the corresponding source file in the supplementary material; see Sec.~\ref{sec:implementation} for more information about the implementation.

\begin{example}[Functoriality of Composition]
\label{ex:func}
~
\\*[0pt]
{
\footnotesize
\noindent
\verb|./examples/example_1.catt|
}

\vspace{1pt}
\noindent
For our first example we choose $\Gamma$ to be the  context built from  two composable 1\=/morphisms. Choosing $X = \{ f \}$, which is clearly up-closed, our construction produces the following context:
\begin{calign}
\begin{aligned}
\begin{tikzpicture}[tikzpic]
\node (y) at (1,0) {$x$};
\node (z) at (2,0) {$y$};
\node (u) at (3,0) {$z$};
\draw [->] (y) to node [above] {$f$} (z);
\draw [->] (z) to node [above] {$g$} (u);
\end{tikzpicture}
\end{aligned}
&
\begin{aligned}
\begin{tikzpicture}[tikzpic]
\node (x) at (0,0) {$x$};
\node (y) at (1,0) {$y$};
\node (z) at (2,0) {$z$};
\draw [->] (x) to [bend left=55] node (f) [above] {$f^-$} (y);
\draw [->] (x) to [bend right=55] node (g) [below] {$f^+$} (y);
\draw [->] (y) to node [above] {$h$} (z);
\draw [-{Implies}, double distance=1.5pt, shorten <=2pt, shorten >=2pt] (.5,.3) to node [right, inner sep=2pt] {$\fun f$} (.5,-.3);
\end{tikzpicture}
\end{aligned}
\nonumber
\\[-4pt]
\Gamma & \Gamma \uparrow \{ f \}
\nonumber
\end{calign}
We define $t = f \s_0 g$, the composite 1\=/morphism, and it can readily be checked that $X = \{ f \}$ is depth\=/0 in both $\Gamma$ and $t$. Applying our construction, we obtain the following result:
$$(f \s_0 g) \uparrow \{ f \} = \fun f \s_0 g : f^- \s_0 g \to f^+ \s_0 g$$
This term witnesses functoriality of $f \s_0 g$ with respect to the first argument. We could choose $X = \{ g \}$ to yield the functoriality property for the second argument, or $X = \{ f,g \}$ for the joint functoriality property. As a depth\=/0 instance, this construction was already known~\cite{benjamin_type_2020}, but serves well here as an initial example.
\end{example}

\begin{example}[Naturality of Composition]
\label{ex:natcomp}
~
\\[0pt]
{
\footnotesize
\noindent
\verb|./examples/example_2.catt|
}

\vspace{1pt}
\noindent
Here  $\Gamma$ is chosen in the same way as for Example~\ref{ex:func}, but for the subset we choose
the full set of free variables $X = \{ x, f, y, g, z \}$:
\begin{calign}
\begin{aligned}
\begin{tikzpicture}[tikzpic]
\node (y) at (1,0) {$x$};
\node (z) at (2,0) {$y$};
\node (u) at (3,0) {$z$};
\draw [->] (y) to node [above] {$f$} (z);
\draw [->] (z) to node [above] {$g$} (u);
\end{tikzpicture}
\end{aligned}
&
\begin{aligned}
\begin{tikzpicture}[tikzpic]
\node (y0) at (1,0) {$x^-$};
\node (z0) at (2,0) {$y^-$};
\node (u0) at (3,0) {$z^-$};
\node (y1) at (1,-1) {$x^+$};
\node (z1) at (2,-1) {$y^+$};
\node (u1) at (3,-1) {$z^+$};
\draw [->] (y0) to node [above] {$f^-$} (z0);
\draw [->] (z0) to node [above] {$g^-$} (u0);
\draw [->] (y1) to node [below] {$f^+$} (z1);
\draw [->] (z1) to node [below] {$g^+$} (u1);
\draw [->] (y0) to node [left, pos=.4] {$\fun x$} (y1);
\draw [->] (z0) to node [left, pos=.4] {$\fun y$} (z1);
\draw [->] (u0) to node [left, pos=.4] {$\fun z$} (u1);
\draw [twoarr] (1.75,-.25) to node [above left, inner sep=0pt] {$\fun f$} +(-.5,-.5);
\draw [twoarr] (2.75,-.25) to node [above left, inner sep=0pt] {$\fun g$} +(-.5,-.5);
\end{tikzpicture}
\end{aligned}
\nonumber
\\[-3pt]
\Gamma & \Gamma \uparrow \{ x, f, y, g, z\}
\nonumber
\end{calign}
Here $X$ is at most depth 1, so we expect a naturality construction. For our term, again we choose $t = f \s_0 g$. Our construction produces a term of the following type:$$(f \s_0 g) \uparrow \{x,
f, y, g, z\} : (f^- \s_0 g^-) \s_0 \fun{z} \to \fun{x} \s_0 (f^+ \s_0 g^+)$$ This term has the following composite form, where $\alpha_{}, \alpha'_{}$ are suitable associator coherences:
\settowidth{\arrlen}{\scriptsize{(associatord)}}
\[
  \begin{split}
    (f^- \s_0 g^-) \s_0 \fun{z}
    & \xrightarrow{\makebox[\arrlen]{\scriptsize $\alpha_{f^-, g^-, \fun z}$}}
      f^- \s_0 (g^- \s_0 \fun{z})\\
    &\xrightarrow{\makebox[\arrlen]{\scriptsize{\(f^- \s_0 \fun{g}\)}}}
      f^- \s_0 (\fun{y} \s_0 g^+)\\
    & \xrightarrow{\makebox[\arrlen]{\scriptsize $\alpha'_{f^-, \fun y, g^+}$}}
      (f^- \s_0 \fun{y}) \s_0 g^+ \\
    & \xrightarrow{\makebox[\arrlen]{\scriptsize{\(\fun{f} \s_0 g^+\)}}}
      (\fun{x} \s_0 f^+) \s_0 g^+ \\
    & \xrightarrow{\makebox[\arrlen]{\scriptsize $\alpha_{\fun x, f^+, g^+}$}}
      \fun{x} \s_0 (f^+ \s_0 g^+)
  \end{split}
\]
We recognize this as the  2\=/dimensional part of the vertical composition formula for natural transformations of pseudofunctors~\cite[Figure~8.5.4]{benabou_introduction_1967}. In this example, we see the first evidence of the complex structure that our construction can produce.
\end{example}

\begin{example}[Naturality of the Associator]\label{ex:nat-assoc}
~
\\*[0pt]
{
\footnotesize
\noindent
\verb|./examples/example_3.catt|
}

\vspace{1pt}
\noindent
For the next example we work with a context $\Gamma$ containing three composable 1\=/morphisms $f,g,h$, and for the term we choose $t = \alpha_{f,g,h} : (f \s_0 g) \s_0 h \to f \s_0 (g \s_0 h)$, the 2\=/dimensional associator coherence. For the subset we choose $X = \{f \}$, and we illustrate $\Gamma$ and $\Gamma \uparrow \{ f \}$ as follows:
\begin{calign}
\begin{aligned}
\begin{tikzpicture}[tikzpic]
\node (y) at (1,0) {$x$};
\node (z) at (2,0) {$y$};
\node (u) at (3,0) {$z$};
\draw [->] (y) to node [above] {$f$} (z);
\draw [->] (z) to node [above] {$g$} (u);
\node (w) at (4,0) {$w$};
\draw [->] (u) to node [above] {$h$} (w);
\end{tikzpicture}
\end{aligned}
&
\begin{aligned}
\begin{tikzpicture}[tikzpic]
\node (x) at (0,0) {$x$};
\node (y) at (1,0) {$y$};
\node (z) at (2,0) {$z$};
\draw [->] (x) to [bend left=55] node (f) [above] {$f^-$} (y);
\draw [->] (x) to [bend right=55] node (g) [below] {$f^+$} (y);
\draw [->] (y) to node [above] {$g$} (z);
\draw [-{Implies}, double distance=1.5pt, shorten <=2pt, shorten >=2pt] (.5,.3) to node [right, inner sep=2pt] {$\fun f$} (.5,-.3);
\node (w) at (3,0) {$w$};
\draw [->] (z) to node [above] {$h$} (w);
\end{tikzpicture}
\end{aligned}
\nonumber
\\[-3pt]
\Gamma & \Gamma \uparrow \{ f \}
\nonumber
\end{calign}
For the term $t \uparrow \{ f \}$, we obtain a coherence that fills the
following square:
\[
\begin{tikzpicture}[tikzpic, xscale=2.8, yscale=1.2]
\path [use as bounding box] (0,-.25) rectangle (1,1.25);
\node (bl) at (0,0) {$(f^{+} \s_{0} g) \s_{0} h$};
\node (br) at (1,0) {$f^{+} \s_{0} (g \s_{0} h)$};
\node (tl) at (0,1) {$(f^{-} \s_{0} g) \s_{0} h$};
\node (tr) at (1,1) {$ f^{-} \s_{0} (g \s_{0} h)$};
\draw [->] (tl) to node [above] {$\alpha_{f^{-},g,h}$} (tr);
\draw [->] (bl) to node [below] {$\alpha_{f^{+},g,h}$} (br);
\draw [->] (tl) to node [left] {$(\fun{f} \s_{0} g) \s_{0} h$} (bl);
\draw [->] (tr) to node [right] {$\fun{f} \s_{0} (g \s_{0} h)$} (br);
\draw [twoarr] (.7,.75) to node [above left, inner sep=1pt] {$\alpha_{f,g,h} \! \uparrow \!\{f\}$} +(-.2,-.5);
\end{tikzpicture}
\]
We recognize this as a witness for naturality of the associator in the first argument. A different subset $X$ would yield a different naturality property; for example, $X = \{ f,g,h \}$ would yield a witness for naturality in all arguments simultaneously.
\end{example}

\begin{example}[Cylindrical Composites]\label{ex:cylinder-nat}
~
\\*[0pt]
{
\footnotesize
\noindent
\verb|./examples/example_4.catt|
}

\vspace{1pt}
\noindent
Our examples so far have been in dimension 2, but in fact our construction works in arbitrary higher dimension. To illustrate this, we consider the following pair of contexts:

\begin{calign}
  \begin{aligned}
    \begin{tikzpicture}[tikzpic, yscale=2]
      \node (00) at (0,0) {$\cdot$};
      \node (10) at (1,0) {$\cdot$};
      \node (20) at (2,0) {$\cdot$};
      \node (01) at (0,1) {$\cdot$};
      \node (11) at (1,1) {$\cdot$};
      \node (21) at (2,1) {$\cdot$};
      \draw [->] (00) to node [below] {$p$} (10);
      \draw [->] (10) to node [below] {$q$} (20);
      \draw [->] (01) to node [above] {$f$} (11);
      \draw [->] (11) to node [above] {$g$} (21);
      \draw [->] (01) to node [left] {$h$} (00);
      \draw [->] (11) to node [left] {$j$} (10);
      \draw [->] (21) to node [right] {$l$} (20);
      \draw [twoarr] (.75,.75) to node [above left, inner sep=1pt] {$a$} +(-.5,-.5);
      \draw [twoarr] (1.75,.75) to node [above left, inner sep=1pt] {$b$} +(-.5,-.5);
    \end{tikzpicture}
  \end{aligned}
&
\begin{aligned}
\begin{tikzpicture}[tikzpic, yscale=2, xscale=1.1]
      \node (00) at (0,0) {$\cdot$};
      \node (10) at (1,0) {$\cdot$};
      \node (20) at (2,0) {$\cdot$};
      \node (01) at (0,1) {$\cdot$};
      \node (11) at (1,1) {$\cdot$};
      \node (21) at (2,1) {$\cdot$};
      \def\bd{27}
      \draw [->] (00) to [bend left=\bd] node [above=1pt, pos=0.86] {$p^{\sm}$} (10);
      \draw [->] (00) to [bend right=\bd] node [below, pos=0.8] {$p^{\sp}$} (10);
      \draw [->] (10) to [bend left=\bd] node [above=1pt, pos=.8] {$q^{\sm}$} (20);
      \draw [->] (10) to [bend right=\bd] node [below, pos=.8] {$q^{\sp}$} (20);
      \draw [->] (01) to [bend left=\bd] node [above, pos=0.2] {$f^{\sm}$} (11);
      \draw [->] (01) to [bend right=\bd] node [below, pos=0.2] {$f^{\sp}$} (11);
      \draw [->] (11) to [bend left=\bd] node [above, pos=.2] {$g^{\sm}$} (21);
      \draw [->] (11) to [bend right=\bd] node [below, pos=.2] {$g^{\sp}$} (21);
      \draw [->] (01) to node [left] {$h$} (00);
      \draw [->] (11) to node [right] {$j$} (10);
      \draw [->] (21) to node [right] {$l$} (20);
      \draw [twoarr] (.9,.8) to [bend left] node [below right=-1pt, inner sep=1pt, pos=0.4] {$a^{\sp}$} (.2,.2);
      \draw [twoarr] (.8,.8) to [bend right] node [above left=-1.5pt, inner sep=1pt] {$a^{\sm}$} (.1,.2);
      \draw [twoarr] (1.9,0.8) to [bend left] node [below right=-0pt, inner sep=1pt, pos=0.4] {$b^{\sp}$} (1.2,.2);
      \draw [twoarr] (1.8,.8) to [bend right] node [above left=-1.5pt, inner sep=1pt] {$b^{\sm}$} (1.1,.2);
\draw [triple] (.4,.55) to node [above right=-2pt, pos=0.2] {$\fun a$} +(.2,-.1);
\draw [triple] (1.4,.55) to node [above right=-2pt, pos=0.3] {$\fun {\normalfont b}$} +(.2,-.1);
\draw [twoarr] (.5,.15) to node [right, inner sep=2pt] {$\fun p$} +(0,-.3);
\draw [twoarr] (1.5,.15) to node [right, inner sep=2pt] {$\fun q$} +(0,-.3);
\draw [twoarr] (.5,1.15) to node [right, inner sep=2pt] {$\fun f$} +(0,-.3);
\draw [twoarr] (1.5,1.15) to node [right, inner sep=2pt] {$\fun g$} +(0,-.3);
    \end{tikzpicture}
    \end{aligned}
    \nonumber
\\
\Gamma
&
\Gamma \uparrow \{ f, a, p, g, b, q \}
\nonumber
\end{calign}
Here $\Gamma$ contains horizontally adjacent squares. Choosing ${X = \{ f, a, p, g, b, q \}}$ yields  $\Gamma \uparrow X$ containing horizontally adjacent 3\=/dimensional {cylinders}, which we call \emph{3\=/cylinders}, whose volumes are filled with 3\=/cells $\fun a,\fun b$. We choose $t$ in $\Gamma$ to be an appropriate horizontal square composite of  the fillers  $a,b$, either constructed by hand or generated automatically with the method of Example~\ref{ex:natcomp}. Our construction then yields a term $t \uparrow X$ representing the horizontal 3\=/cylinder composite.

In Section~\ref{sec:cylinders} we show how this idea can be iterated to define a  composition and stacking operations on $n$\-cylinders in any dimension. Evaluating these operations in our implementation, we observe that the sizes of the resulting proof artifacts grow rapidly. Here we give output artifact sizes in bytes for the horizontal composition operation of $n$\-cylinders:
\begin{align*}
n=2&&n=3&&n=4 && n=5
\\
\text{818} && \text{10,236} && \text{67,498} && \text{509,702}
\end{align*}
Note  the $n=3$ case is already an order of magnitude more complex than the square composite  presented explicitly in Example~\ref{ex:natcomp}. The case $n=5$ reaches the limits of our ability to generate and type-check on a typical workstation.

These binary cylinder composites have been generated for $n=3$ by Henry and Lanari, where they play a central role in the construction of a semi-model structure on weak 3\=/groupoids~\cite{henry_homotopy_2023, lanari_semimodel_2018}. For $n \geq 4$, we do not believe these composites have been previously constructed in a weak setting. In the strict case these constructions are known for all $n$, where they are used in the construction of the folk model structure on strict $\omega$\=/categories~\cite{lafont_folk_2010}.
\end{example}

\begin{example}[Conical Composites]\label{ex:cone-nat}
~
\\*[0pt]
{
\footnotesize
\noindent
\verb|./examples/example_5.catt|
}

\vspace{1pt}
\noindent
Similar to the cylindrical case, we can form horizontal composites of $n$\=/dimensional \textit{cones}, with the first two such contexts shown here:
\vspace{0pt}
\begin{calign}
\begin{aligned}
\begin{tikzpicture}[tikzpic, yscale=2, xscale=1.1, scale=1]
      \node (00) at (0,0) {$\cdot$};
      \node (10) at (1,0) {$\cdot$};
      \node (20) at (2,0) {$\cdot$};
      \node (11) at (1,1) {$\top$};
      \def\bd{27}
      \draw [->] (00) to node [below=0pt, pos=0.5] {$f$} (10);
      \draw [->] (10) to node [below=0pt, pos=.5] {$g$} (20);
      \draw [<-] (11) to node [right, pos=.7] {$j$} (10);
      \draw [twoarr] (.2,.5) to node [above=2pt, inner sep=1pt, pos=.5] {$a$} +(.75,0);
      \draw [twoarr, {Implies}-] (1.8,.5) to node [above=2pt, inner sep=1pt, pos=.5] {$b$} +(-.75,0);
      \draw [->] (00) to [out=90, in=-165] node [left, pos=.2] {$h$} (11);
      \draw [->] (20) to [out=90, in=-15] node [right, pos=.2] {$l$} (11);
      \draw [->, white] (00) to [bend right=\bd] node [below, pos=0.8] {$f^\sp$} (10);
    \end{tikzpicture}
    \end{aligned}
&
\begin{aligned}
\begin{tikzpicture}[tikzpic, yscale=2, xscale=1.1, scale=1]
      \node (00) at (0,0) {$\cdot$};
      \node (10) at (1,0) {$\cdot$};
      \node (20) at (2,0) {$\cdot$};
      \node (11) at (1,1) {$\top$};
      \def\bd{27}
      \draw [->] (00) to [bend left=\bd] node [above=1pt, pos=0.86] {$f^\sm$} (10);
      \draw [->] (00) to [bend right=\bd] node [below, pos=0.8] {$f^\sp$} (10);
      \draw [->] (10) to [bend left=\bd] node [above=1pt, pos=.8] {$g^\sm$} (20);
      \draw [->] (10) to [bend right=\bd] node [below, pos=.8] {$g^\sp$} (20);
      \draw [twoarr] (.1,.5) to [bend right=10] node [below=1pt, inner sep=1pt, pos=0.6] {$a^\sp$} +(.9,0);
      \draw [twoarr] (.23,.65) to [bend left=10] node [above=2pt, inner sep=1pt, pos=.6] {$a^{\sm}$} +(.76,0);
\draw [triple] (.6,.62) to node [right=0pt, pos=0.4] {$\fun a$} +(.0,-.1);
      \draw [twoarr, {Implies}-] (1.9,.5) to [bend left=10] node [below=1pt, inner sep=1pt, pos=0.5] {$b^\sp$} +(-.91,0);
      \draw [twoarr, {Implies}-] (1.75,.65) to [bend right=10] node [above=2pt, inner sep=1pt, pos=.5] {$b^{\sm}$} +(-.77,0);
\draw [triple] (1.4,.62) to node [right=0pt, pos=0.45] {$\fun b$} +(.0,-.1);
\draw [twoarr] (.5,.15) to node [right, inner sep=2pt] {$\fun f$} +(0,-.3);
\draw [twoarr] (1.5,.15) to node [right, inner sep=2pt] {$\fun g$} +(0,-.3);
      \draw [-, line width=4pt, white] (00) to [out=90, in=-165] (11);
      \draw [->] (00) to [out=90, in=-165] node [left, pos=.2] {$h$} (11);
      \draw [->] (20) to [out=90, in=-15] node [right, pos=.2] {$l$} (11);
      \draw [-, line width=4pt, white] (11) to (10);
      \draw [<-] (11) to node [right, pos=.7] {$j$} (10);
    \end{tikzpicture}
    \end{aligned}
\nonumber
\\[-4pt]
\Gamma & \Gamma \uparrow \{ f,a,g,b \}
\nonumber
\end{calign}
The importance of conical geometries in higher category theory has been emphasized by Buckley and Garner~\cite{buckley_orientals_2016}, and Ara and Maltsiniotis~\cite{ara_joint_2020}. We give a detailed presentation of conical composites in Section~\ref{sec:cones}. We believe this is the first such construction in the weak $\omega$\-category setting.
\end{example}



\paragraph{Outline} In Sec.~\ref{sec:catt} we recall the presentation and
properties of the theory \catt. In Sec.~\ref{sec:naturality} we present
our naturality construction as a meta-operation, along with a proof of
correctness. In Sec.~\ref{sec:cylinders} we define cylinders and their
composites, and in Sec.~\ref{sec:cones}, we define cones and their composites.
We conclude in  Sec.~\ref{sec:implementation} with a discussion of our implementation.

%% file: sections/acknowledgements.tex
\subsection*{Acknowledgements}

We would like to express our gratitude to Eric Finster for many insightful discussions. We also thank Dimitri Ara, Fran\c cois M\'etayer and  Alex Rice for helpful comments.

%% file: sections/catt.tex

\section{Globular Higher Categories}\label{sec:catt}

\subsection{Introducing the Type Theory}

\noindent
We give a complete but concise presentation of the dependent type theory
\(\catt\) introduced by Finster and Mimram~\cite{finster_typetheoretical_2017}
to model weak \(\omega\)\=/categories. Contexts in this type theory are
generating data for \(\omega\)\-categories, and they correspond precisely to
finite computads~\cite{benjamin_catt_2024} in the sense of Batanin~\cite{batanin_computads_2002}
and Dean et al.~\cite{dean_computads_2024}.

We first choose a countably infinite set
\(\V\) of variables, and define the \emph{raw syntax} for the theory \(\catt\)
by the grammar presented in Fig.~\ref{fig:grammar-catt}.
\begin{figure}
  \centering
  \begin{tabular}{lrclll}
    Variable & \(x\) & \(\Coloneqq \) & \(x \in \V\)  \\
    Context & \(\Gamma,\Delta\) & \(\Coloneqq \) & \(\emptycontext\)
            & \(\vert\) & \((\Gamma,x:A)\) \\
    Type & \(A,B\) & \(\Coloneqq\) & \(\obj\)
            & \(\vert\) & \(\arr[A]{t}{u}\) \\
    Term & \(t,u\) & \(\Coloneqq\) & \(x\)
            & \(\vert\) & \(\coh_{\Gamma,A}[\gamma]\)\\
    Substitution & \(\gamma,\delta\) & \(\Coloneqq\) & \(\sub{}\)
            & \(\vert\) & \(\sub{\gamma, x \mapsto t}\)
  \end{tabular}
  \caption{Grammar for the raw syntax of \(\catt\).}
  \label{fig:grammar-catt}
\end{figure}
The action of a raw substitution \(\gamma\) on raw types, raw terms and raw
substitutions is defined by induction. If \(x\) is a variable, then we view $\gamma$ as an association list, and define \(x[\gamma]\) as the term most recently associated to $x$ in $\gamma$, or \(x\) if that does not exist. The other cases are defined as follows:
\begin{align*}
  \obj[\gamma] &= \obj
  & (\arr[A]{u}{v})[\gamma] &= \arr[{A[\gamma]}]{u[\gamma]}{v[\gamma]} \\
               &
  & \coh_{\Gamma,A}[\delta][\gamma] &= \coh_{\Gamma,A}[\delta\circ \gamma]\\
  \sub{}\circ\gamma &= \sub{}
  & \sub{\delta,x\mapsto t}\circ\gamma &= \sub{\delta\circ\gamma,x\mapsto t[\gamma]}
\end{align*}
We also define the notion of \emph{dimension}:
\begin{align*}
  \dim(\obj) &= -1 & \dim(\arr[A]{u}{v}) &= \dim(A)+1 \\
  \dim(\Gamma) &= \max_{(x:A) \in \Gamma}
                 \dim A + 1
\end{align*}

The raw syntax is subject to the following judgements that characterise valid
elements:
\begin{align*}
  \text{Contexts} &~~ \Gamma\vdash
  & \text{Substitutions} &~~ \Delta\vdash \gamma\colon \Gamma\\
  \text{Types} &~~ \Gamma\vdash A
  & \text{Terms} &~~ \Gamma\vdash u\colon A
\end{align*}
A well typed term \(\Gamma\vdash u\colon A\) is equipped with a dimension
given by \(\dim_{\Gamma}(u) = \dim A + 1\). When \(A = \arr[B]{v}{w}\), we write
\(\partial_{\Gamma}^{-}u = v\) and \(\partial_{\Gamma}^{+}u = w\), and we omit
the context \(\Gamma\) when there is no ambiguity.

For contexts, substitutions, and terms consisting of variables, the rules for those judgements
are the standard ones of dependent type theory (see
Fig.~\ref{fig:rules-catt}.) The arrow type does not represent a function type, but rather
a \textit{directed path type}, which can be considered a directed variant of the
identity type of Martin-L\"of type theory. Where convenient we will drop the type annotation on the path type $\to_A$, writing simply $\to$ where it will not cause confusion. The term formation \ruleref{rule:coh} admits two alternative side conditions, with the \textsc{(comp)} condition yielding composites, and the \textsc{(inv)} condition yielding coherences.

\begin{figure}[t]
  \centering
  {\small
    \begin{mathpar}
      \inferdef{(ec)}{\null}{\emptycontext \vdash} \and
      \inferdef{(cc)}{\Gamma \vdash \\ \Gamma\vdash A \\
        x\notin\Var(\Gamma)}{\Gamma,x:a\vdash} \\
      \inferdef{(obj)}{\Gamma\vdash}{\Gamma\vdash\obj}\label{rule:obj} \and
      \inferdef{(arr)}{\Gamma\vdash t\colon A \\ \Gamma\vdash u\colon A}
      {\Gamma\vdash \arr[A]{t}{u}}\label{rule:arr} \\
      \inferdef{(es)}{\Gamma\vdash}{\Gamma\vdash\sub{}:\emptycontext} \and
      \inferdef{(var)}{\Gamma\vdash \\ (x:A)\in\Gamma}{\Gamma\vdash x:A}
        \\
      \inferdef{(sc)}{\Delta\vdash \gamma:\Gamma \\ \Gamma,x:A\vdash \\
        \Delta\vdash t:A[\gamma]}{\Delta\vdash\sub{\gamma,x\mapsto
          t}:(\Gamma,x:A)}
      \\
      \inferdef{(coh)}{\Gamma\vdashps \\
        \Gamma\vdash u:A \\ \Gamma\vdash v:A \\
        \Delta\vdash \gamma:\Gamma}{\Delta\vdash
        \coh_{\Gamma,\arr{u}{v}}[\gamma]
        : \arr[{A[\gamma]}]{u[\gamma]}{v[\gamma]}}\label{rule:coh}
    \end{mathpar}}
    \\
\begin{minipage}[l]{4cm}\raggedright\textsc{(coh)} alternative \\ side conditions\end{minipage}
\begin{minipage}[l]{3cm}{\small \begin{align*}
      \textsc{(comp)}
      &\begin{cases}
        \Var(\partial^{-}\Gamma) = \Var(u)\cup\Var(A) \\
        \Var(\partial^{+}\Gamma) = \Var(v)\cup\Var(A)
      \end{cases}
      \\
      \textsc{(inv)}
      &\begin{cases}
        \Var(\Gamma) = \Var(u)\cup\Var(A) \\
        \Var(\Gamma) = \Var(v)\cup\Var(A)
      \end{cases}
    \end{align*}}
    \end{minipage}%
        \caption{Derivation rules for \catt}
  \label{fig:rules-catt}
\end{figure}

\begin{figure}
  \centering
  \begin{mathpar}
    \inferrule{\null}{(x:\obj)\vdashps x:\obj} \and
    \inferrule{\Gamma\vdashps x:A}{\Gamma,y:A,f:\arr[A]xy\vdashps f:\arr[A]xy} \\
    \inferrule{\Gamma\vdashps f:\arr[A]xy}{\Gamma\vdashps y:A}
    \and
    \inferrule{\Gamma\vdashps x:\obj}{\Gamma\vdashps}
  \end{mathpar}
  \caption{Derivation rules for pasting contexts}
  \label{fig:rules-ps}
\end{figure}

The \ruleref{rule:coh} makes use of an
auxiliary judgement \(\Gamma\vdashps\), characterising an important class of contexts called \emph{pasting contexts}. The rules for this judgement are given in Fig.~\ref{fig:rules-ps}. The importance of those structures for \(\omega\)\-categories was first pointed out by Batanin~\cite{batanin_monoidal_1998}, who described them in terms of
rooted planar trees; they can also be characterised
as certain colimits of discs~\cite{ara_infty_2010}. We illustrate all these perspectives in Fig.~\ref{fig:pasting-scheme}. The figure demonstrates that for a given rooted planar tree, we obtain a variable in each separate region around a node, which we call a \textit{sector}.

Each pasting context \(\Gamma\) comes equipped with source and target pasting
contexts \(\partial_i^{\pm}\Gamma\) for \(i \in \N\) defined recursively by as follows. For the base case, we define $\partial_i^\pm(x : \obj) = (x : \obj)$. For the recursive case, writing  \(\Gamma' = (\Gamma,y:A,f:\arr[A]xy)\), and writing \(\operatorname{drop}\) for the operator that removes the tail element:
\begin{align*}
    \partial_{i}^{-}\Gamma'
    & =
      \begin{cases}
        \partial_i^{-}\Gamma & \dim A \ge i-1 \\
        (\partial^{-}_{i}\Gamma,y\!:\!A,f\!:\!\arr[A]xy) & \text{otherwise}
      \end{cases} \\
    \partial_{i}^{+}\Gamma'
    & =
      \begin{cases}
        \partial_i^{+}\Gamma & \text{if } \dim A > i-1 \\
        (\operatorname{drop}(\partial_i^{+}\Gamma),y:A) & \text{if } \dim A = i-1 \\
        (\partial^{+}_{i}\Gamma,y:A,f:\arr[A]xy) & \text{otherwise}
      \end{cases}
\end{align*}
In the case that \(i = \dim \Gamma - 1\), we simply write  \(\partial^{\pm}\Gamma\) for the codimension-1 boundaries. The source and target contexts are always \(\alpha\)\=/equivalent, and it is sometimes useful to write \(\partial \Gamma\) for either of them. In that case, we recover the source and target via the weakening
substitutions, which we denote by \(\Gamma\vdash \delta^{\pm}_{\Gamma} : \partial\Gamma\).
 Considered as a planar tree of height $n$, the boundaries $\partial^-(\Gamma)$ resp. $\partial^+(\Gamma)$ are obtained by removing  leaves at height $n$, and at height $n\!-\!1$ retaining the variables in the left- or right-most sectors respectively around each node.
 \begin{figure}
  \begin{calign}
  \Gamma\hspace{-10pt} &
    \adjustbox{valign=c}{
    \begin{tikzpicture}[every node/.style={scale=1}, xscale=.9, yscale=.8, inner sep=2pt, font=\small]
      \node [black,label={left:$ x$},label={above:$ y$},label={right:$ z$}](x) {$\bullet$};
      \node [black,label={above:$ k$}] (x0) at (.5,1) {$\bullet$};
      \node [black,label = {left:$ f$}, label={above:$g$},label={right:$ h$}] (x1) at (-.5,1)  {$\bullet$};
      \node[label={above:$a$}] (x10) at (-1,2)  {$\bullet$};
      \node[label={above:$b$}] (x11) at (0,2)  {$\bullet$};
      \draw[black] (x.center) to (x0.center);
      \draw[black] (x.center) to (x1.center);
      \draw (x1.center) to (x10.center);
      \draw (x1.center) to (x11.center);
    \end{tikzpicture}
    }
    &\hspace{-0.5em}\small
    \begin{tikzcd}[ampersand replacement=\&]
      x
      \ar[r, bend left = 70, "f"{name=A}]
      \ar[r, bend right = 70, "h"{below,name=C}]
      \ar[r, "g"{description,name=B}]
      \& y
      \ar[r, "k"]
      \& z
      \ar[from=A,to=B,"a", Rightarrow]
      \ar[from=B,to=C,"b"{near start}, Rightarrow]
    \end{tikzcd}
    &\hspace{-0.5em}\small
    \left(
      \begin{array}{@{}l@{}}
        x:\obj, y:\obj,\\
        f:\arr[]xy,\\
        g:\arr[]xy\\
        a:\arr[]fg,\\
        h:\arr[]xy\\
        b:\arr[]gh,\\
        z:\obj\\
        k:\arr[]yz
      \end{array}
    \right)
    \nonumber
    \\
    \partial^-(\Gamma)\hspace{-10pt} &
    \adjustbox{valign=c}{
    \begin{tikzpicture}[every node/.style={scale=1}, xscale=.9, yscale=.8, inner sep=2pt, font=\small]
      \node [black,label={left:$ x$},label={above:$ y$},label={right:$ z$}](x) {$\bullet$};
      \node [black,label={above:$ k$}] (x0) at (.5,1) {$\bullet$};
      \node [black, label={above:$f$}] (x1) at (-.5,1)  {$\bullet$};
      \draw[black] (x.center) to (x0.center);
      \draw[black] (x.center) to (x1.center);
    \end{tikzpicture}
    }
    &\hspace{-0.5em}\small
    \begin{tikzcd}[ampersand replacement=\&]
      x \ar[r, bend left = 70, "f"{name=A}] \& y \ar[r, "k"] \& z
    \end{tikzcd}
    &\hspace{-0.5em}\small
    \left(
      \begin{array}{@{}l@{}}
        x:\obj, y:\obj\\
        f : \arr[]xy,\\
        z:\obj\\
        k:\arr[]yz
      \end{array}
    \right)
    \nonumber
    \\
    \partial^+(\Gamma)\hspace{-10pt} &
    \adjustbox{valign=c}{
    \begin{tikzpicture}[every node/.style={scale=1}, xscale=.9, yscale=.8, inner sep=2pt, font=\small]
      \node [black,label={left:$ x$},label={above:$ y$},label={right:$ z$}](x) {$\bullet$};
      \node [black,label={above:$ k$}] (x0) at (.5,1) {$\bullet$};
      \node [black, label={above:$h$}] (x1) at (-.5,1)  {$\bullet$};
      \draw[black] (x.center) to (x0.center);
      \draw[black] (x.center) to (x1.center);
    \end{tikzpicture}
    }
    &\hspace{-0.5em}\small
    \begin{tikzcd}[ampersand replacement=\&]
      x \ar[r, bend right = 70, "h"{below}] \& y \ar[r, "k"] \& z
    \end{tikzcd}
    &\hspace{-0.5em}\small
    \left(
      \begin{array}{@{}l@{}}
        x:\obj, y:\obj\\
        h : \arr[]xy,\\
        z:\obj \\
        k:\arr[]yz
      \end{array}
    \right)
    \nonumber
  \end{calign}
  \caption{A pasting context and its boundaries displayed as trees, pastings of discs, and contexts.}
  \label{fig:pasting-scheme}
\end{figure}

The side-conditions \textsc{(comp)} and \textsc{(inv)} make reference to the variables of
a term \(t\) (resp.~type \(A\), context \(\Gamma\), substitution \(\gamma\)),
denoted \(\Var(t)\) (resp.~\(\Var(A)\), \(\Var(\Gamma)\), \(\Var(\gamma)\).)
These are defined by induction on the syntax as follows:
\begin{align*}
  \Var(\emptycontext) &= \emptyset & \Var(\Gamma,x:A) &= \Var(\Gamma)\!\cup\!\{x\}\\
  \Var(\obj) &= \emptyset & \Var(\arr[A]{u}{v}) &=
                                                      \Var(A)\!\cup\!\Var(u)\!\cup\!\Var(v)
  \\
  \Var(x) &= \{x\} & \Var(\coh_{\Gamma,A}[\gamma]) &= \Var(\gamma)\\
  \Var(\langle \rangle) &= \emptyset & \!\!\!\!\Var(\sub{\gamma,x\mapsto t})
  &= \Var(\gamma)\!\cup\!\Var(t)
\end{align*}
Given a pasting context \(\Gamma\vdashps\) and a type \(\Gamma\vdash \arr[A]{u}{v}\), at most one of the side conditions will be satisfied; if one is satisfied, we say that type is \emph{full} for the side condition. If $u \to_A v$ is full for  \textsc{(comp)}, then \(u\) and \(v\) give composition operations for the source and target of \(\Gamma\) respectively, and \(\coh_{\Gamma,\arr u v}[\id]\) is interpreted as a \textit{composition operation} of all the cells of $\Gamma$. If $u \to_A v$ is full for \textsc{(inv)}, then \(\coh_{\Gamma,\arr u v}[\id]\) is interpreted as a \textit{coherence}, witnessing equivalence of the
operations $u,v$. This coherence cell will be invertible in an appropriate $\omega$\-categorical sense, which explains the name of the side-condition. Taken together, these conditions are directly analogous to the admissibility conditions of Maltsiniotis~\cite{maltsiniotis_grothendieck_2010}.

\subsection{Constructions in \texorpdfstring{\catt}{CATT}}

\noindent
This section aims to present various useful constructions in \catt that are commonly used. We call a context a \emph{disc}, if it is obtained from a
linear tree. For instance, the following gives a representation of the disc
context of dimension \(3\):
\[
  \begin{tikzcd}[ampersand replacement=\&, column sep = large]
    x \ar[r, bend left = 35, shift left, "f",""{below,name=A}]
    \ar[r,bend right=35, shift right, "g"',""{name=B}]
    \ar[r, phantom, "\underset{m}{\Rrightarrow}"]
    \& y
    \ar[from=A, to=B, "b", Rightarrow, bend left, shift left = 2]
    \ar[from=A, to=B, "a"', Rightarrow, bend right, shift right = 2]
  \end{tikzcd}
\]
Pasting contexts give rise to  canonical composition opertions called \emph{unbiased composites}, which we define as follows.
\begin{definition}
  For a pasting context $\Gamma \vdashps$, we define the \emph{unbiased
    composite} $\comp_{\Gamma}$ as follows. If \(\Gamma\) is a disc with
  top-dimensional variable \(x\), we define \(\comp_{\Gamma}= x\). Otherwise we
  define $\comp_\Gamma$ by induction on $\dim \Gamma$:
  \begin{align*}
    \comp_{\Gamma} &= \coh_{\Gamma, \arr[]{\comp_{\partial^{-}\Gamma}}{\comp_{\partial^{+}\Gamma}}}[\id_\Gamma]
  \end{align*}
  This construction uses the \textsc{(comp)} side condition. One can check that the following are derivable:
  \begin{align*}
    &\Gamma\vdash\comp_{\Gamma}:\obj & \quad\text{if \(\dim\Gamma=0\)}\\
    &\Gamma\vdash\comp_{\Gamma}:\arr[]{\comp_{\partial^{-}\Gamma}}{\comp_{\partial^{+}\Gamma}}
                                     & \quad\text{otherwise}
  \end{align*}

\end{definition}

When $\Gamma$ consists of discs glued along successive $k$\=/dimensional faces, we use the notation: \(t_1 *_k \dots *_k t_n\) for \(\comp_{\Gamma}[\sigma]\), where  the $t_i$ are the images of the top-dimensional variables of the discs under $\sigma$. For instance, if $\Gamma$ is the following context, and $f[\sigma] = t_1$, $a[\sigma] = t_2$, $k[\sigma] = t_3$, $l[\sigma] = t_4$, we write $t_1 *_0 t_2 *_0 t_3 *_0 t_4$ for the term $\comp_{\Gamma}[\sigma]$:
\[\begin{tikzcd}[ampersand replacement=\&]
        v \& w \& x \& y \& z
        \arrow["f", from=1-1, to=1-2]
        \arrow[""{name=0, anchor=center, inner sep=0}, "g", curve={height=-12pt}, from=1-2, to=1-3]
        \arrow[""{name=1, anchor=center, inner sep=0}, "h"', curve={height=12pt}, from=1-2, to=1-3]
        \arrow["k", from=1-3, to=1-4]
        \arrow["l", from=1-4, to=1-5]
        \arrow["a", shorten <=3pt, shorten >=3pt, Rightarrow, from=0, to=1]
\end{tikzcd}\]
The other case of the side condition of the introduction rule for coherences permits the
definition of cells that witness relations between other cells, like for
instance the associator of Example~\ref{ex:nat-assoc} which can be defined in
the context \(\Gamma\) illustrated in the same Example as:
\[
  \alpha_{f,g,h} = \coh_{\Gamma, \arr[]{(f\s_{0} g) \s_{0} h}{f\s_{0}(g\s_{0} h)}}[\id_{\Gamma}]
\]

\subsection{Suspension and Opposites.}
\label{sec:suspopp}

\noindent
Several \emph{meta-operations} have been defined for \catt, as operations that
act on the syntax~\cite{benjamin_type_2020, benjamin_hom_2024}: they take a term
in \catt as input, and output a new term. We present a brief overview of those
that we use use in this article.

The first meta-operation is \emph{suspension}, which takes a term \(\Gamma\vdash t:A\), and returns a
new term \(\Sigma\Gamma\vdash\Sigma t:\Sigma A\) obtained by freely adjoining
two new objects (say \(N\) and \(S\)) in the context, and replacing the type
\(\obj\) of objects in the original term by \(\arr N S\). Formally, it is defined
on the raw syntax as follows, and illustrated for contexts in Fig.~\ref{fig:meta-operations}:
\begin{align*}
  \Sigma\emptycontext &= (N:\obj,S:\obj)
  & \Sigma(\Gamma,x:A) &= (\Sigma\Gamma,x:\Sigma A)\\
  \Sigma\obj &= \arr[\obj]NS
  & \Sigma(\arr[A]uv) &= \arr[\Sigma A]{\Sigma u}{\Sigma v} \\
  \Sigma x &= x
  & \Sigma(\coh_{\Gamma,A}[\gamma]) &= \coh_{\Sigma\Gamma,\Sigma
                                      A}[\Sigma\gamma] \\
  \Sigma\sub{} &= \sub{N\mapsto N,S\mapsto S}
  &\Sigma\sub{\gamma,x\mapsto t} &= \sub{\Sigma\gamma,x\mapsto\Sigma t}
\end{align*}

The second meta-operation is the formation of \emph{opposites}.
Given a term \(\Gamma\vdash t:A\) and a subset \(M\subseteq \N_{>0}\), it
returns a new term \(\op_{M}\Gamma\vdash \op_{M}(t):\op_{M}(A)\) by formally
reversing the direction of all variables in \(t\) whose dimension is in
\(M\). Formally, this operation is defined by:
\begin{align*}
  \op_{M}\emptycontext &= \emptycontext
  & \op_{M}(\Gamma,x:A) &= (\op_{M}\Gamma,x:\op_{M}A)\\
  \op_{M}\obj &= \obj\\
  \op_{M}x &= x
  & \op_{M}(\coh_{\Gamma,A}[\gamma]) &= \coh_{\Gamma',\op_{M} A}[\op_{M}^{\Gamma}\circ\op_{M}\gamma] \\
  \op_{M}\sub{} &= \sub{}
  &\op_{M}\sub{\gamma,x\mapsto t} &= \sub{\op_{S}\gamma,x\mapsto\op_{M} t} \\
  \omit\rlap{\(
    \op_{M}(\arr[A]uv)
    \!=\!
    \begin{cases}
      \arr[{\op_{M}A}]{\op_{M}u}{\op_{M}v} \text{ if }\dim A+2 \in M\\
      \arr[{\op_{M}A}]{\op_{M}v}{\op_{M}u} \text{ if }\dim A+2 \notin M
    \end{cases}\)
  }
\end{align*}
where \(\Gamma'\) is the unique pasting context isomorphic to \(\op_{M}\Gamma\),
up to \(\alpha\)\=/equivalence, and \(\op_{M}\Gamma\vdash\op_{M}^{\Gamma}:\Gamma'\) is the isomorphism.
Benjamin and Markakis have established correctness of those meta-operations~\cite{benjamin_type_2020,benjamin_hom_2024},
meaning that the following rules are admissible:
\begin{mathpar}
  \inferrule{\Gamma\vdash}{\Sigma\Gamma\vdash}
  \and \inferrule{\Gamma\vdash A}{\Sigma\Gamma\vdash \Sigma A}
  \and \inferrule{\Gamma\vdash t:A}{\Sigma\Gamma\vdash \Sigma t:\Sigma A} \\
  \and \inferrule{\Delta\vdash\gamma:\Gamma}{\Sigma\Delta\vdash\Sigma\gamma:\Sigma\Gamma}
  \and \inferrule{\Gamma\vdash}{\op_{M}\Gamma\vdash}
  \and \inferrule{\Gamma\vdash A}{\op_{M}\Gamma\vdash \op_{M} A} \\
  \and \inferrule{\Gamma\vdash t:A}{\op_{M}\Gamma\vdash \op_{M} t:\op_{M} A}
  \and \inferrule{\Delta\vdash\gamma:\Gamma}{\op_{M}\Delta\vdash\op_{M}\gamma:\op_{M}\Gamma}
\end{mathpar}

\begin{figure}
  \begin{align*}
    \begin{aligned}
      \Gamma:\ &
      \begin{tikzcd}[ampersand replacement=\&]
        x
        \ar[r, bend left = 35, "f", ""{below, name = A}]
        \ar[r, bend right = 35, "g"', ""{name = B}]
        \ar[from = A, to = B, Rightarrow, "a"]
        \& y\ar[r, "h"]
        \& z
      \end{tikzcd} \\
      \op_1\Gamma:\ &
      \begin{tikzcd}[ampersand replacement=\&]
        z \ar[r, "h"]
        \& y
        \ar[r, bend left = 35, "f", ""{below, name = A}]
        \ar[r, bend right = 35, "g"', ""{name = B}]
        \ar[from = A, to = B, Rightarrow, "a"]
        \& x
      \end{tikzcd}
    \end{aligned} &&
    \begin{aligned}
      \Sigma \Gamma:\ &
      \begin{tikzcd}[column sep = 1cm, row sep = 1.2cm]
        N
        \arrow[dd,""{name=0, anchor=center, inner sep=0}, "{x}"',
        bend right = 90]
        \arrow[dd,""{name=1, anchor=center, inner sep=0}, "{z}", bend
        left = 90, pos=0.49]
        \arrow[dd,""{name=2, anchor=center, inner sep=0},
        "{y}"{description}]
        \\
        \\
        S
        \arrow[""{name=3, anchor=center, inner sep=0}, "{f}", curve={height = -20pt}, shorten <=8pt, shorten >=8pt, Rightarrow, from=0, to=2]
        \arrow[""{name=4, anchor=center, inner sep=0}, "{g}"', curve={height = 20pt}, shorten <=8pt, shorten >=8pt, Rightarrow, from=0, to=2, pos = 0.48]
        \arrow["{h}"', shorten <=6pt, shorten >=6pt, Rightarrow, from=2, to=1]
        \arrow["{a}", shorten <=3pt, shorten >=3pt, Rightarrow, scaling nfold=3, from=3, to=4]
      \end{tikzcd}
    \end{aligned}
  \end{align*}
  \caption{Suspension and opposite of a context.}
  \label{fig:meta-operations}
\end{figure}

%% file: sections/naturality.tex
\section{Naturality}\label{sec:naturality}

\subsection{Overview}

\noindent
Here we formally introduce our naturality construction. The notion of \emph{depth} plays an important role, which we define for contexts $\Gamma$,  types \({\Gamma \vdash A}\), terms \(\Gamma \vdash t : A\) and substitutions $\Gamma \vdash \sigma : \Delta$ with respect to a set of variables \(X \subseteq \Var \Gamma\) as follows, where by convention we let \(\max \emptyset = -1\):
\begin{align*}
  \depth_X \!t &= \max \{\dim t - \dim x : x\!\in\! (\Var(t)\!\cup\!\Var(A))\!\cap \!X\} \\
  \depth_X \hspace{-3pt} A &= \max \{\dim A - \dim x : x\!\in\! \Var(A)\!\cap\! X\}
\\
  \depth_X \!\sigma &= \max\{ \depth_X x[\sigma] : x\!\in\! \Var\Delta\}
\\
  \depth_X \!\Gamma &= \depth_X(\id_\Gamma).
\end{align*}
We define naturality with respect to a set \(X\) of variables of a context
\(\Gamma\) of depth at most \(1\) that is \emph{up-closed},
meaning that if \(x \in X\) appears in the variables of the type of
some other variable \(y \in \Var(\Gamma)\), then \(y\in X\). We denote by
\(\U(\Gamma)\) the set of up-closed sets of variables of \(\Gamma\).
We also define the \emph{preimage} $\gamma^{-1}X$ of a set of variables $X \subseteq \Var(\Gamma)$ under a substitution
\(\Gamma \vdash \gamma : \Delta\) to be:
\[
  \gamma^{-1}X = \{x\in\Var(\Delta)\,|\, \Var(x[\gamma])\cap X\neq \emptyset\}.
\]

We now present the naturality construction, which operates on well-formed syntactic objects.
We will first define its outputs as raw syntactic objects, before proving they are well-formed.
Our construction proceeds by mutual recursion on the derivation trees of contexts, types, terms and substitutions; this is sound since derivations in \catt are unique~\cite[Lemma~7]{finster_typetheoretical_2017}.
We produce the following objects for all \(d\in \{-1,0,1\}\) and \(k \in \{0,1\}\):

\paragraph{{\normalfont (1)} Naturality of Contexts}
Given a context \(\Gamma\vdash\) and a set \(X\in\U(\Gamma)\)
  of \(\depth_{X}\Gamma \leq d\), we define two contexts \(\Gamma \pmctx X\) and
  \(\Gamma\uparrow X\), together with two substitutions \(\inj^\pm_{\Gamma,X}\) such that:
  \begin{align*}
    \Gamma \pmctx X & \vdash &
    \Gamma \pmctx X & \vdash \inj^{\pm}_{\Gamma,X} : \Gamma \\
    \Gamma\uparrow X & \vdash &
    \Gamma\uparrow X & \vdash \inj^{\pm}_{\Gamma,X} : \Gamma
  \end{align*}

\paragraph{{\normalfont (2)} Naturality of Types for Fresh Variables}
Given a type \mbox{\(\Gamma\vdash A\)}, a
  variable \(x\) that is fresh in \(\Gamma\), and \(X\in \U(\Gamma)\)
  such that \(\depth_X(\Gamma)\le d\) and \(\depth_X(A)\le k-1\), we define a
  type $A \uparrow^x X$ such that:
  \[
    (\Gamma,x:A)\pmctx (X\cup\{x\})\vdash (A\uparrow^{x} X)
  \]

\paragraph{{\normalfont (3)} Naturality of Types for Arbitrary Terms}
Given a type \mbox{\(\Gamma\vdash A\)} and a set
  \(X\in\U(\Gamma)\) such that \(\depth_{X}\Gamma \leq d\), then for all terms
  \(\Gamma\vdash t:A\) such that \(0\leq \depth_{X}t \leq k\), we define a type $A \uparrow ^t X$ such that:
  \[
    (\Gamma\uparrow X) \vdash (A\uparrow^{t}X)
  \]

\paragraph{{\normalfont (4)} Naturality of Terms}
Given a term \(\Gamma\!\vdash\! t\!:\!A\) and a set
  \mbox{\(X\in\U(\Gamma)\)}
  with \(\depth_{X}(\Gamma)\leq d\) and \({0\leq\depth_{X}(t)\leq k}\), we define a term $t \uparrow X$ such that:
  \[
    \Gamma\uparrow X\vdash t\uparrow X : A\uparrow^{t}X.
  \]

\paragraph{{\normalfont (5)} Naturality of Substitutions}
Given a substitution \({\Delta\vdash \gamma:\Gamma}\) and
  \(X\in\U(\Delta)\) such that \(\depth_{X}\Delta \le d\) and \(\depth_{X}\gamma\leq k\), we define a
  substitution $\gamma \uparrow X$ such that:
  \[
    \Delta\uparrow X \vdash (\gamma \uparrow X) : (\Gamma\uparrow
    \gamma^{-1}X)
  \]

\paragraph{{\normalfont (6)} Naturality of Term Constructors}
Given a pasting context \(\Gamma\vdashps\), a type
  \(\Gamma\vdash A\) full in \(\Gamma\), and \(X\in\U(\Gamma)\) such that
  \(0\leq \depth_{X}(\coh_{\Gamma,A}[\id_{\Gamma}])\leq d\), we define a term
  $\coh_{\Gamma,A}\uparrow X$ such that:
  \[
    \Gamma\uparrow X \vdash (\coh_{\Gamma,A}\uparrow X) :
    (A\uparrow^{\coh_{\Gamma,A}[\id]} X)
  \]
For the sake of this definition, we assume that we can
use an oracle picking fresh variables, so that for  any variable \(x\in X\), we may pick
fresh variables \(x^{-},x^{+}\) and \(\fun{x}\). We may freely use the
assumptions that these names do not collide with any others. This can be
formalised by explicitly carrying a choice function everywhere, or with de
Bruijn levels~\cite{debruijn_lambda_1972}. For the sake of this presentation, we
ignore these issues.

We may now state our main theorem, which asserts termination and correctness of the construction. Its proof will be presented after the construction is fully described.
\begin{theorem}\label{thm:correctness}
The recursive  construction of naturality is well\=/founded, and its outputs satisfy judgements (1)-(6).
\end{theorem}

\subsection{Construction of Naturality}\label{sec:nat-constr}

\noindent
Here we give the full description of the naturality construction.
We observe that in the following description, we make no use of the parameters $d$ and $k$.
Nonetheless, they will be crucial in the proof of Theorem \ref{thm:correctness}.

\paragraph{{\normalfont (1)} Naturality of Contexts} For the empty context
\(\emptycontext\vdash\), we make the following definitions:
\begin{align*}
  \emptycontext \pmctx \emptyset &= \emptycontext \uparrow\emptyset =
                                \emptycontext &
  \inj^{\pm}_{\emptycontext,\emptyset} &= \inj_{\emptycontext,\emptyset} = \sub{}
\end{align*}
For the context \((\Gamma,x:A)\) with \(x\notin X\), we define:
\begin{align*}
(\Gamma,x:A)\pmctx X &= (\Gamma,x:A)\uparrow X = (\Gamma\uparrow X,x:A)
\\
  \inj^{\pm}_{(\Gamma,x:A),X} &= \sub{\inj^{\pm}_{\Gamma,X},x\mapsto x}
\end{align*}
For the context \((\Gamma,x:A)\) with \(x\in X\), the set
\(X' = X\setminus\{x\}\) is up-closed in \(\Gamma\), allowing us to define:
\begin{align*}
  (\Gamma,x:A)\pmctx X &= (\Gamma\uparrow X',x^{-}:A[\inj^{-}_{\Gamma,x'}]
                      ,x^{+}:A[\inj^{+}_{\Gamma,x'}]) \\
  (\Gamma,x:A)\uparrow X &= ((\Gamma,x:A)\pmctx X,\fun x : A \uparrow^{x}X) \\
  \inj^{\pm}_{(\Gamma,x:A),X} &= \sub{\inj^{\pm}_{\Gamma,X'},x\mapsto x^{\pm}}
\end{align*}

\paragraph{{\normalfont (2)} Naturality of Types for Fresh Variables}
For \({\Gamma\vdash\obj}\) and a variable \(x \notin \Var(\Gamma)\), define:
\[
  \obj\uparrow^{x} X = \arr[\obj]{x^{-}}{x^{+}}
\]
For \(\Gamma\vdash \arr[A]uv\), we denote \(n=\dim A + 1\) and define:
\begin{calign}
\nonumber
(\arr[A]uv) \uparrow^{x} X = \arr[{\arr[A]{u[\inj^{-}_{\Gamma,X}]}{v[\inj^{+}_{\Gamma,x}]}}]{a}{b}
\\
\nonumber
  a =
      \begin{cases}
       x^{-} & \text{if \(X\cap\Var_{\Gamma} v = \emptyset\)}\\
       x^{-} \s_{n} (v\uparrow X) & \text{otherwise}
     \end{cases}\\
\nonumber
  b =
      \begin{cases}
        x^{+} & \text{if \(X\cap\Var_{\Gamma} u = \emptyset\)}\\
        (u\uparrow X)\s_{n} x^{+} & \text{otherwise}
      \end{cases}
\end{calign}

\paragraph{{\normalfont(3)} Naturality of Types for Arbitrary Terms}
For the type \(\Gamma\vdash \obj\) and the term \(\Gamma\vdash t:\obj\), define:
\[
  \obj\uparrow^{t} X =
  \arr[\obj]{t[\inj^{-}_{\Gamma,X}]}{t[\inj^{+}_{\Gamma,X}]}
\]
For the type \(\Gamma\vdash \arr[A]uv\) and the term
\(\Gamma\vdash t:\arr[A]uv\), we denote \(n = \dim A + 1\) and define:
\begin{calign}
  (\arr[A]uv) \uparrow^{t} X = \arr[{\arr[A]{u[\inj^{-}_{\Gamma,X}]}{v[\inj^{+}_{\Gamma,x}]}}]{a}{b}
  \nonumber
\\
\nonumber
  a =
      \begin{cases}
       t[\inj^{-}_{\Gamma,X}] & \text{if \(X\cap\Var_{\Gamma} v = \emptyset\)}\\
       t[\inj^{-}_{\Gamma,X}] \s_{n} (v\uparrow X) & \text{otherwise}
     \end{cases}\\
\nonumber
  b =
      \begin{cases}
        t[\inj^{+}_{\Gamma,X}] & \text{if \(X\cap\Var_{\Gamma} u = \emptyset\)}\\
        (u\uparrow X)\s_{n} t[\inj^{+}_{\Gamma,X}] & \text{otherwise.}
      \end{cases}
\end{calign}

\paragraph{{\normalfont (4)} Naturality of Terms}
For a variable \(\Gamma\vdash x:A\), necessarily \(x\in X\), and we define:
\[
  x\uparrow X = \fun x
\]
For a term \(\Delta\vdash \coh_{\Gamma,A}[\gamma]:A[\gamma]\) we define:
\[
  (\coh_{\Gamma,A}[\gamma])\uparrow X = (\coh_{\Gamma,A}\uparrow
  X)[\gamma\uparrow X]
\]

\paragraph{{\normalfont(5)} Naturality of Substitutions}
For the empty substitution \({\Gamma\vdash\sub{}:\emptycontext}\), we define:
\[
  \sub{}\uparrow X = \sub{}
\]
For the substitution \(\Delta\vdash\sub{\gamma,x\mapsto t}:(\Gamma,x:A)\), in the
case where \(\Var(t) \cap X = \emptyset\), we define:
\[
  \sub{\gamma,x\mapsto t}\uparrow X = \sub{\gamma\uparrow X,x\mapsto t}
\]
while in the case where \(\Var(t) \cap X \neq \emptyset\), we define:
\begin{align*}
  \sub{&\gamma,x\mapsto t}\uparrow X = \\
       & \sub{\gamma\uparrow X,x^{-}\mapsto
         t[\inj^{-}_{\Delta,x}], x^{+}\mapsto t[\inj^{+}_{\Delta,x}], \fun{x}
         \mapsto t\uparrow X}
\end{align*}

\paragraph{{\normalfont(6)} Naturality of Term Constructors}
Given a pasting context \(\Gamma\vdashps\) and a full type
\(\Gamma\vdash A\), we define the term
\(\coh_{\Gamma,A}\uparrow X\) in two cases. If \(\depth_X(\Gamma) = 0\),
we let:
\[
  \coh_{\Gamma,A}\uparrow X = \coh_{\Gamma\uparrow X,
    A\uparrow^{\coh_{\Gamma,A}[\id]}X}[\id_{\Gamma\uparrow
    X}]
\]
using that \(\Gamma\uparrow X\) is again a pasting context~\cite[Lemma~87]{benjamin_type_2020}.
For the case \(\depth_X(\Gamma) = 1\), which forces
\({\depth_X(\coh_{\Gamma,A}[\id_\Gamma]) = 1}\) as well, before giving the general
formula for this case, we will build the special cases of
\emph{linear} and \emph{reduced composites}.



\paragraph{{\normalfont(6)(i)} Naturality of Linear Composites}
We call a \emph{linear composite} the unbiased composite of a pasting context
of the form \(\Psi^{n}_{k}\) representing the \(k\)\=/ary composite of
\((n\!+\!1)\)\=/cells along their \(n\)\-boundary, illustrated in Fig.~\ref{fig:linear-ps}.
We write \(\comp^{n}_{k}\) for
the unbiased composite of \(\Psi^{n}_{k}\), and \(A^{n}_{k}\) for its type:
\[
  \Psi^{n}_{k} \vdash \comp^{n}_{k} : A^{n}_{k}
\]
We proceed by induction on \(n\).
We define \(\comp^{0}_{k}\uparrow X\) to be
an \(\abs{X}\)\=/ary composite, where every \({x \in X}\) of maximal
dimension gives rise to a whiskering of \(\fun{x}\), and every other \(x\in X\)
gives rise to an associator. The naturality term
\(\comp^{0}_{2} \uparrow \Var(\Psi^{0}_{2})\) is fully described in
Example~\ref{ex:natcomp}.
Appendix~\ref{app:nat-coh} gives a precise account of this construction. For
the inductive step, we use that \(\comp^{n+1}_{k} = \Sigma \comp^{n}_{k}\) and
that suspension commutes with the depth\=/0 naturality, proven
in Lemma~\ref{lemma:funcsusp}, to define:
\[
  \comp^{n+1}_{k}\uparrow X = (\Sigma\comp^{n}_{k}) \uparrow X =
  \Sigma(\comp^{n}_{k} \uparrow X)
\]

\begin{figure}
  \centering
  \[
    \begin{array}{cc}
      \Psi^{0}_{2}
      :
        \begin{tikzcd}[ampersand replacement=\&]
          x \ar[r,"f"]
          \& y \ar[r,"g"]
          \& z
        \end{tikzcd}
      & \multirow{2}*{\(
                \Psi^{1}_{2} :
      \begin{tikzcd}[ampersand replacement=\&]
        x \& y \& {}
        \arrow[""{name=0, anchor=center, inner sep=0}, "f",
        curve={height=-10pt}, shift left, from=1-1, to=1-2]
        \arrow[""{name=1, anchor=center, inner sep=0}, "h"',
        curve={height=10pt}, shift right, from=1-1, to=1-2]
        \arrow[""{name=2, anchor=center, inner sep=0}, "g"{description}, from=1-1, to=1-2]
        \arrow["a", shorten <=2pt, shorten >=2pt, Rightarrow, from=0, to=2]
        \arrow["b", shorten <=2pt, shorten >=2pt, Rightarrow, from=2, to=1]
      \end{tikzcd}
        \)} \\
      \\
      \hspace{10pt}\Psi^{0}_{3}:
        \begin{tikzcd}[ampersand replacement=\&]
          x \ar[r,"f"]
          \& y \ar[r,"g"]
          \& z \ar[r,"h"]
          \& w
        \end{tikzcd}
    \end{array}
  \]
  \caption{Linear pasting contexts.}
  \label{fig:linear-ps}
\end{figure}

\paragraph{{\normalfont (6)(ii)} Naturality of Reduced Composites} We say a pasting context of dimension \(n>0\) is \emph{reduced}~\cite[Section~4.4]{benjamin_invertible_2024} if it does not contain \(n\)\=/dimensional variables glued along an \((n\!-\!1)\)\=/dimensional boundary. We call composites over reduced pasting contexts
\emph{reduced composites}. For a reduced composite
\(\coh_{\Gamma,A}[\id]\), define \(X^{lm} \subseteq X\) as the depth\=/0 variables in $\Gamma$. We then define the following raw terms:
\begin{gather*}
  c_{\Gamma,A,X} = (\coh_{\Gamma,A}\uparrow X^{lm})[\theta_{\Gamma,X}] \\
  \coh_{\Gamma,A}\uparrow X = j^{-}_{\Gamma,X,A} \s_{n} c_{\Gamma,X,A}\s_{n} j^{+}_{\Gamma,X,A}
\end{gather*}
Here, \((\Gamma\uparrow X) \vdash \theta_{\Gamma,X} : (\Gamma\uparrow X^{lm}) \)
is the unique substitution sending \(\fun x\) to \(\fun x\) for \(x\in X\), and
every variable \(y\notin X\) of depth~\(0\) to \(y\).
The correction terms \(j^{\pm}_{\Gamma,X,A}\) are canonical coherences,
interchanging the composites appearing in the type of \(c_{\Gamma,X,A}\) to yield those in the desired type.
A more detailed version of this construction can be found in Appendix~\ref{app:nat-coh}.
We illustrate this construction for the case of the whiskering:
\begin{align*}
  \Gamma
  &=
    \begin{tikzcd}[ampersand replacement=\&]
      x
      \ar[r, bend left=40, ""{below,name=A}, "f"]
      \ar[r, bend right=40, ""{above,name=B}, "g"']
      \ar[from=A, to=B, Rightarrow,"a"]
      \& y\ar[r,"h"]
      \& z
    \end{tikzcd}
  &
    \begin{aligned}
      A  &= \arr[]{f \s_{0} h}{g \s_{0} h} \\
      X  &= \{f,g,a,h\} \\
          X^{lm} &= \{a,h\}
    \end{aligned}
\end{align*}
The contexts \(\Gamma\uparrow X^{lm}\) and
\(\Gamma\uparrow X\) are visualised in Fig.~\ref{fig:reduced-comp}. The
substitution \(\theta_{\Gamma,X}\) sends
\(a^{-}[\theta_{\Gamma,X}] = a^{-}\s_{1}\fun{g}\) and
\({a^{+}[\theta_{\Gamma,X}] = \fun{f}\s_{1}a^{+}}\),
and acts as the identity on the remaining variables.
The term \(\coh_{\Gamma,A}\uparrow X\) is then given as follows:
\settowidth{\arrleng}{\scriptsize{\(c_{\Gamma,X,A}\)}}
\[
  \begin{split}
    (a^- \sj 0 h) \sj 1 (\fun{g} \sj 0 h)
        & \xrightarrow{\makebox[\arrleng]{\scriptsize $j^-_{\Gamma,X,A}$}}
    ((a^- \sj 1 \fun{g}) \sj 0 h^{-}) \sj {1} (g^{+} \sj {0}\fun h) \\
    &\xrightarrow{\makebox[\arrleng]{\scriptsize{\(c_{\Gamma,X,A}\)}}}
    (f^{-} \sj {0}\!\fun h) \sj {1} ((\fun{f} \sj 1 a^+) \sj 0 h^{+})
    \\[-3pt]
    & \xrightarrow{\makebox[\arrleng]{\scriptsize $j^+_{\Gamma,X,A}$}}
    (\fun{f} \sj 0 h) \sj 1 (a^+ \sj 0 h)
  \end{split}
\]

\begin{figure}
  \centering
  \begin{tabularx}{\linewidth}{CC}
    \(\begin{tikzcd}[ampersand replacement=\&, scale=0.85]
      x
      \ar[rr, bend left=40, ""{below,name=A}, "f", shift left = 4, shorten
        <= -8pt, shorten >= -8 pt]
      \ar[rr, bend right=40, ""{above,name=B}, "g"', shift right = 4, shorten
        <= -8pt, shorten >= -8 pt]
      \ar[rr, phantom, "\overset{\small\fun{a}}{\Rrightarrow}"]
      \ar[from=A, to=B, Rightarrow,"a^{+}", bend left, shift left]
      \ar[from=A, to=B, Rightarrow,"a^{-}"', bend right, shift right]
      \&\& y
      \ar[r,"h^{-}", shift left, bend left, ""{below, name=H1}]
      \ar[r,"h^{+}"', shift right, bend right, ""{name=H2}]
      \ar[from = H1, to= H2, "\fun h", Rightarrow]
      \& z
  \end{tikzcd}\)
    &
      \(\begin{tikzcd}[ampersand replacement=\&, scale=0.85]
        x
        \ar[rr, bend left=40, ""{below,name=A}, "f^{-}", shift left = 4, shorten
        <= -8pt, shorten >= -8 pt]
        \ar[rr, bend right=40, ""{above,name=B}, "g^{+}"', shift right = 4, shorten
        <= -8pt, shorten >= -8 pt]
        \ar[rr, phantom, ""{description,name=M1, pos=0.3},
        ""{description,name=M2, pos=0.7}, "\overset{\small\fun{a}}{\Rrightarrow}"]
        \ar[from=A, to=M2, Rightarrow, bend left=12, shift left, shorten >=
        -2.5pt, "\fun{f}"{pos=0.95}]
        \ar[from=A, to=M1, Rightarrow, bend right=12, shift right, shorten >=
        -2.5pt, "a^{-}"'{very near end}]
        \ar[from=M2, to=B, Rightarrow, bend left=12, shift left, shorten <=
        -2.5pt, "a^{+}"{very near start}]
        \ar[from=M1, to=B, Rightarrow, bend right=12, shift right, shorten <=
        -2.5pt, "\fun{g}"'{very near start}]
        \&\& y
        \ar[r,"h^{-}", shift left, bend left, ""{below, name=H1}]
        \ar[r,"h^{+}"', shift right, bend right, ""{name=H2}]
        \ar[from = H1, to= H2, "\fun h", Rightarrow]
        \& z
      \end{tikzcd}\)
  \end{tabularx}

  \caption{The contexts \(\Gamma\uparrow X^{lm}\) and \(\Gamma\uparrow X\).}
  \label{fig:reduced-comp}
\end{figure}

\paragraph{{\normalfont(6)(iii)} Naturality of General Composites} For the
general case, we will use that every pasting context \(\Gamma\) gives rise to a reduced
pasting context \(\Gamma^{r}\) with the same source and target, along with a reduction
substitution \(\Gamma\vdash\rho_{\Gamma}:\Gamma^{r}\). This substitution  acts as the  identity on the boundary, and sends maximal-dimensional
variables of \(\Gamma^{r}\) to linear composites of maximal-dimensional
variables of \(\Gamma\)~\cite[Section~4.4]{benjamin_invertible_2024}. This is
illustrated in Fig.~\ref{fig:ps-reduction}. For the example shown in this figure, the reduction
substitution is determined by
\(c[\rho_\Gamma] = a \s_0 b\) and \(f[\rho_\Gamma] = f\). The equality of the boundaries allows us to view \(A\) as a full type over \(\Gamma^r\), and hence to define the following:
\begin{gather*}
        e_{\Gamma,X,A} = \coh_{\Gamma,\arr[]{\coh_{\Gamma,A}[\id_{\Gamma}]}{\coh_{\Gamma^{r},A}[\rho_{\Gamma}]}}[\id] \\
        e^{-}_{\Gamma,X,A} =
        \begin{cases}
          e_{\Gamma,X,A}[\inj^{-}_{\Gamma,X}] & \text{if \(\Var(v)\cap X = \emptyset\)}\\
          e_{\Gamma,X,A}[\inj^{-}_{\Gamma,X}] \s_{n-1} (v\uparrow X) & \text{otherwise}
        \end{cases} \\
        e^{+}_{\Gamma,X,A} =
        \begin{cases}
          e_{\Gamma,X,A}[\inj^{+}_{\Gamma,X}] & \text{if \(\Var(u)\cap X = \emptyset\)}\\
          (u\uparrow X) \s_{n-1} e_{\Gamma,X,A}[\inj^{+}_{\Gamma,X}] & \text{otherwise}
          \end{cases} \\
        c'_{\Gamma,X,A} = (\coh_{\Gamma^{r},A}\uparrow\rho_{\Gamma}^{-1}(X))[\rho_{\Gamma}\uparrow X]\\
        \coh_{\Gamma,A}\uparrow X = e^{-}_{\Gamma,X,A} \s_n c'_{\Gamma,X,A} \s_n e ^{+}_{\Gamma,X,A}
\end{gather*}
This completes the description of the naturality construction.

\begin{figure}
  \begin{calign}
    \begin{aligned}
      \begin{tikzcd}[ampersand replacement = \&]
        \cdot
        \ar[r, bend left = 60, shift left, ""{below, name = A}]
        \ar[r, bend right = 60, shift right, ""{above,name = D}]
        \ar[r,  ""{above,name = B}, ""{below, name = C}]
        \ar[from = A, to = B, Rightarrow, "a"]
        \ar[from = C, to = D, Rightarrow, "b"]
        \& \cdot
        \ar[r, "f"]
        \& \cdot
      \end{tikzcd}
    \end{aligned} &
    \begin{aligned}
      \begin{tikzcd}[ampersand replacement = \&]
        \cdot
        \ar[r, bend left, shift left, ""{below, name = E}]
        \ar[r, bend right, shift right, ""{above,name = F}]
        \ar[from = E, to = F, Rightarrow, "c"]
        \& \cdot
        \ar[r, "f"]
        \& \cdot
      \end{tikzcd}
    \end{aligned}
    \nonumber \\
    \Gamma & \Gamma^r \nonumber
  \end{calign}

  \caption{The reduction of a pasting context.}
  \label{fig:ps-reduction}
\end{figure}

\subsection{Correctness of Naturality}

\noindent
We now proceed to show correctness of the naturality construction, by providing a proof of Theorem~\ref{thm:correctness}. First we note the following lemmas, which can be
established by structural induction, taken simultaneously with   induction on the depth \(d\). Full proofs
of those lemmas can be found in Appendix~\ref{app:correctness}.

\begin{lemma}\label{lem:empty-intersection}
  Let \(\Gamma\vdash\) and \(X\in\U(\Gamma)\) of depth at most \(1\). Then for
  a term \(\Gamma\vdash t : A\) (resp. type \(\Gamma\vdash A\), substitution
  \(\Gamma\vdash \sigma : \Delta\)) whose variables do not intersect with \(X\):
  \begin{align*}
  t[\inj^\pm_{\Gamma,X}] &= t &
  \sigma\circ \inj^{\pm}_{\Gamma,X} &= \sigma &
  A[\inj^{\pm}_{\Gamma,X}] &= A
  \end{align*}
\end{lemma}

\begin{lemma}\label{lem:pullback-stability}
  Let \(\Gamma\vdash\sigma:\Delta\) and \(X\in\U(\Gamma)\) of depth at most \(1\)
  with respect to \(\Gamma\) and \(\sigma\), we have:
  \begin{align*}
    \sigma\circ\inj^{\pm}_{\Gamma,X} &=\inj^{\pm}_{\Delta,\sigma^{-1}X}\circ(\sigma \uparrow X) \\
    \intertext{Moreover, if \(X\) is of depth at most \(1\) with respect to some
      term \(\Delta\vdash t:A\) or substitution \({\Delta\vdash \tau : \Theta}\):}
    t[\sigma] \uparrow X &= (t\uparrow \sigma^{-1}X)[\sigma\uparrow X] \\
    (\tau\circ\sigma) \uparrow X &= (\tau\uparrow \sigma^{-1}X)\circ
      (\sigma\uparrow X) \\
    A[\sigma]\uparrow^{t[\sigma]}X &= (A\uparrow^{t}(\sigma^{-1}
      X))[\sigma\uparrow X]
  \end{align*}
\end{lemma}

\begin{lemma}\label{lemma:nat-variable}
  Suppose that a type \(\Gamma \vdash A\) is given, and we choose \(x\not\in\Var(\Gamma)\).
  For every choice \(X\in \U(\Gamma)\) such that
  \(\depth_{X\cup\{x\}}(\Gamma,x:A)\le 1\), then the naturality of the type of the term \(t = x\) coincides with the naturality of the type of the fresh variable \(x\):
  \[A\uparrow^{t}(X\cup\{x\}) = A\uparrow^{x}X\]
\end{lemma}

\noindent
We now proceed with the main proof.

\begin{proof}[Proof of Theorem~\ref{thm:correctness}]
We prove this result by case analysis on $d \in \{-1,0,1\}$ and $k \in \{0,1\}$.
Each of these cases will be proved by induction on the derivations, establishing both existence (i.e. termination), and correctness.

\paragraph{Case $d=-1$} For this first case, we handle   $k=0$ and $k=1$ together. Since $d=-1$ we must have
  \(X = \emptyset\), and hence the judgements (3), (4) and (6) hold vacuously since the assumptions on $\depth_X(t)$ require $\Var(t) \cap X \neq \emptyset$. Next, the following can be established by induction on the derivations:
  \begin{align*}
    \Gamma \pmctx X &= \Gamma \uparrow X = \Gamma &
    \inc_{\Gamma,X}^\pm &= \id_\Gamma \\ A\uparrow^{x} X &= \arr[A]{x^-}{x^+} &
    \sigma \uparrow X &= \sigma
  \end{align*}
This is sufficient to prove well-foundedness and the judgements of (1), (2) and (5), for both values of $k$.

\paragraph{Case $d=k=0$} This is handled by Benjamin~\cite[Section~3.4]{benjamin_type_2020}, but for completeness we give a full proof here. To begin, we prove the judgements of (1) and (2) by mutual induction.

\paragraph{-- Judgements of (1)} It suffices to consider contexts of the form
  \((\Gamma, x : A)\) such that \(X \not=\emptyset\). The set
  \(X' = X \cap \Var(\Gamma)\) is up-closed and of depth at most \(d\). By the
  inductive hypothesis for \(\Gamma\) and \(X'\), the substitutions
  \(\inc^\pm_{\Gamma,X}\) exist and are valid, so:
  \[\Gamma\uparrow X'\vdash A[\inc^\pm_{\Gamma,X'}]\]
  It follows that \((\Gamma,x:A)\pmctx X\) and \(\inc^\pm_{(\Gamma,x:A),X}\) exist and satisfy
  the claimed judgements. If \(x\notin X\) then by
  up-closure and Lemma~\ref{lem:empty-intersection}, we get that
  \[A[\inl_{\Gamma,X'}] = A[\inr_{\Gamma,X'}] = A\]
  which shows that \((\Gamma,x:A)\uparrow X\) is a valid context. Otherwise, by the inductive hypothesis for $\Gamma$ on the judgements of (2),
  \[
    (\Gamma,x:A) \pmctx X \vdash A \uparrow^x X'
  \]
  establishing again the validity of the context \((\Gamma,x:A)\uparrow X\).

\paragraph{-- Judgement of (2)} We fix a context \(\Gamma\), a type \(\Gamma\vdash A\), and a fresh variable $x$. By the inductive hypothesis for $\Gamma$, we may assume the judgements of (1) on $\Gamma$. Since \(k = 0\) then
  \({\Var A \cap X = \emptyset}\), hence:
  \begin{gather*}
    \Gamma\uparrow X \vdash x^\pm : A[\inc^\pm_{\Gamma,X}] \\
    A[\inl_{\Gamma',X}] = A[\inr_{\Gamma',X}] = A
  \end{gather*}
This shows $A \uparrow^x X$ exists and satisfies the claimed judgement.
This closes the mutual induction.

Next, we prove the judgement of (3). We fix a context \(\Gamma\) and a term \(\Gamma\vdash t : A\). Since \(k = 0\) then
  \(\Var A \cap X = \emptyset\), hence:
  \begin{gather*}
    \Gamma\uparrow X \vdash t[\inc^\pm_{\Gamma,X}] : A[\inc^\pm_{\Gamma,X}] \\
    A[\inl_{\Gamma',X}] = A[\inr_{\Gamma',X}] = A
  \end{gather*}
  This shows $A \uparrow^t X$ exists and satisfies the claimed judgement.

Next, we prove the judgement of (6). Let \(\Gamma\)
  be a pasting context, \(\Gamma \vdash A\) be a full type, and \(t = \coh_{\Gamma,A}[\id]\).  Since $d=0$, it follows that \(\depth_X(\Gamma) = 0\), and hence \(\Gamma \uparrow X\) is a pasting
  context~\cite[Lemma~87]{benjamin_type_2020}. Since $k=0$, it follows that $\depth_X(t) = 0$, hence \(A\uparrow^t X\) is full~\cite[Lemma~90]{benjamin_type_2020}. Hence $\coh_{\Gamma,A} \uparrow X$ exists, and the judgement follows.

  Next, we prove the judgements of (4) and (5) by mutual  induction. First we consider the judgement of (4): when ${t=x}$ is a variable, necessarily \(x\in X\), existence is clear, and the judgement holds
  by Lemma~\ref{lemma:nat-variable}; when \(t = \coh_{\Gamma,A}[\sigma]\), by inductive hypothesis $\sigma\uparrow X$ exists and satisfies the judgement of (5), hence the claimed judgement for \(t\) then holds by
  Lemma~\ref{lem:pullback-stability}. Similarly the judgement in (5) holds
  definitionally for \(\sub{}\), while for \(\sub{\sigma, x\mapsto t}\),
  it holds by the inductive hypothesis on \(\sigma\) and on \(t\).

\newcommand\kcase[2]{\ensuremath{(#1_{k=#2})}}
\paragraph{Case $d=1$} Here it will be convenient to consider certain subcases: for example, we write \kcase 2 0 to denote  the judgement of (2) in the case that $k=0$.

We begin with existence and correctness proofs of the following subcases by mutual induction: (1), \kcase 2 0, \kcase 2 1, \kcase 3 0, \kcase 3 1, \kcase 4 0, \kcase 5 0. For (1) when $\Gamma = \emptyset$, this follows from the case $d=-1$. We now fix a context $\Gamma$, and assume~(1) holds for $\Gamma$. The subcases \kcase 3 0 and \kcase 2 0 both follow from  (1), with an identical proof to the case $d=k=0$, since in this case $\Var(A) \cap X = \emptyset$. This establishes that the type of $t \uparrow X$ exists and is valid. Next, we establish subcases \kcase 5 0 and \kcase 4 0 by a separate mutual induction, again by an identical proof to the case $d=k=0$. To prove \kcase 3 1, we fix a term $\Gamma \vdash t:A$, and suppose that $\depth_X(t) = k = 1$. Then \(A = \arr[A']{u}{v}\),
  and we must consider the possible variable intersections of \(X\) with \(\Var(u)\) and \(\Var(v)\). If they do not intersect, the same argument as \kcase 3 0 applies. If \(\Var(u)\cap X \not=\emptyset\),
  then \(\depth_X(u) = 0\), so by \kcase 4 0 we have:
  \[
    \Gamma\uparrow X \vdash u\uparrow X : \arr{u[\inl_{\Gamma,X}]}{u[\inr_{\Gamma,X}]}
  \]
  If \(\Var(v)\cap X \not=\emptyset\), a similar argument applies. This establishes
  that the type \(A\uparrow^t X\) is well-defined and valid, and hence the subcase \kcase 3 1 holds.
  A similar argument with $x^\pm$ in place of $t[\inc^\pm_{\Gamma,X}]$ shows subcase  $\kcase 2 1$.
  This subcase establishes that (1) holds for $(\Gamma, x:A)$ for any type $\Gamma \vdash A$, and the mutual induction for the subcases is concluded.

  Next, we prove the judgement of (6). For that, let \(\Gamma\)
  be a pasting context, \(\Gamma \vdash A\) be a full type, and denote
  \(n = \dim \Gamma\) and \(t = \coh_{\Gamma,A}[\id]\). Note that by fullness,
  \(0\le \depth_X(\Gamma) \le \depth_X(t) \le 1\).
  We proceed by a case split analogous to the definition of the naturality
  construction. Suppose first that
  \(\depth_X(\Gamma) = 0\), and $\depth_X(t)=1$. Then \(\Gamma \uparrow X\) is a pasting
  context~\cite[Lemma~87]{benjamin_type_2020}, and furthermore we show in Lemma~\ref{lem:fullterm-fullfunct} that $A \uparrow^t X$ is full. Suppose now that \(\depth_X(t) = 1\) and \(\depth_X \Gamma = 1\).
  We consider a case analysis on the structure of $t$.
  \paragraph{-- Linear composites} If \(t = \comp_k^0\), then existence and correctness amounts to
  composability of the associators and whiskerings that \(\coh_{\Gamma,A}\uparrow X\)
  is comprised of. If \(t = \comp_k^{n+1}\), correctness follows from the previous
  case and commutativity of depth\=/0 naturality with suspension. Those facts are established in Lemma~\ref{lemma:lincompnatsusp}.

  \paragraph{-- Reduced composites} When \(\Gamma\)
  is a reduced pasting context, the middle part of the definition of \(\coh_{\Gamma,A}\uparrow X\)
  exists and is valid by the previous cases. Hence, it suffices to show existence of the
  two interchangers. Those facts are established in Lemma \ref{lem:red-comp-correct}.

  \paragraph{-- General composites}  The terms
  \(e^\pm_{\Gamma,X,A}\) are valid by the case $d=0$ and correctness of $e_{\Gamma,X,A}$,
  and \(\coh_{\Gamma^r,A}\uparrow \rho_\Gamma^{-1}(X)\) is valid by the case of reduced composites. Moreover,
  since \(\rho_\Gamma\) sends variables either to variables or to linear
  composites, correctness of the naturality \(\rho_\Gamma\uparrow X\) may be
  established using the same proof of as for arbitrary substitutions. Combining these facts, we conclude that \(\coh_{\Gamma,A}\uparrow X\)
  exists and satisfies the judgement in (6) in this case as well.

  \vspace{5pt}
  \noindent
  This concludes the proof of the judgement of (6).
  It remains to prove existence and correctness for  subcases \kcase 4 1 and \kcase 5 1. These are proven identically to the subcases \kcase 4 0 and \kcase 5 0, making use of the judgements in (6).

\paragraph{Case $d=0,k=1$}
This case is subsumed by the case $d=1$.

\vspace{5pt}
\noindent
This covers all the cases, and so the proof is complete.
\end{proof}

%% file: sections/cylinders.tex
\section{Cylinders, Compositions and Stacking}
\label{sec:cylinders}

\noindent
This section is dedicated to our construction of cylinders and their composites, extending our earlier brief discussion in Example~\ref{ex:cylinder-nat}. We  introduce  cylinder contexts,  then show that  cylinders admit various composites,
as well as an additional operation  we call stacking. As discussed in the introduction, these structures are expected to play an important role in model structure and path object construction for weak $\omega$\-categories. The geometrical structure of cylinder contexts is well-known, and these composites have already been described in the strict case~\cite{lafont_folk_2010}; the new contributions are the explicit composition and stacking  formulas in a weak $\omega$\-category setting.

\paragraph{Cylinder Contexts}
By induction on \(n > 0\), we define contexts \(\partial\cylctx^{n}\) and
\(\cylctx^{n}\) that capture the geometry of cylinders, along with variables
\(\bot_{n}, \top_{n},\filler_{n}\):
\begin{align*}
  \partial\cylctx^{1} &= (\top_{1}:\obj,\bot_{1}:\obj) &
  \filler_{n+1} &= \fun{\filler}_{n} \\
  \cylctx^{1} &= (\partial\cylctx^{1}, \filler_{1}:\arr{\top_{1}}{\bot_{1}}) &
  \top_{n+1} &= \fun{\top}_{n} \\
  \partial\cylctx^{n+1} &= \op_{n+1}(\cylctx^{n}\pmctx\{\top_{n},\bot_{n},\filler_{n}\}) &
  \bot_{n+1} &= \fun{\bot}_{n} \\
  \cylctx^{n+1} &= \op_{n+1}(\func{\cylctx^{n}}{\top_{n},\bot_{n},\filler_{n}})
\end{align*}
Fig.~\ref{fig:cylinder-ctx} illustrates the contexts \(\cylctx^{2}\) and
\(\cylctx^{3}\).

\begin{figure}[b]
  \centering
  \[
    \begin{array}{c@{\qquad}c}
      \cylctx^{2} :
      \begin{tikzcd}[ampersand replacement=\&, column sep=small]
        \top_1^- \ar[d, "\filler_1^-"']\ar[r,"\top_2"]
        \&  \top_1^+ \ar[d,"\filler_{1}^{+}"]\ar[dl, Rightarrow,"\filler_{2}"{description}] \\
        \bot_1^- \ar[r,"\bot_2"']
        \& \bot_1^+
      \end{tikzcd}
      &
        \hspace{10pt}\cylctx^{3} :
      \begin{tikzcd}[ampersand replacement=\&, column sep=small]
        \top_1^- \ar[dd, "\filler_1^-"']
        \ar[rr, bend left=15]
        \ar[rr, bend right=15]
        \ar[rr,phantom,"\rotatebox{30}{\(\Downarrow\)}"]
        \& \& \top_1^+ \ar[dd,"\filler_{1}^{+}"]
        \ar[ddll, Rightarrow, bend left=15, ""{name=A}]
        \ar[ddll, Rightarrow, bend right=15,""'{name=B}]
        \ar[from=A, to=B, phantom,"\rotatebox{-30}{\(\Lleftarrow\)}"]
        \\ \\
        \bot_1^-
        \ar[rr, bend left=15]
        \ar[rr, bend right=15]
        \ar[rr,phantom,"\rotatebox{30}{\(\Downarrow\)}"]
        \& \& \bot_1^+
      \end{tikzcd}
    \end{array}
  \]
  \caption{The cylinder contexts in dimension \(2\) and \(3\).}
  \label{fig:cylinder-ctx}
\end{figure}

By construction, there exists a type \(\partial\cylctx^n \vdash \Cyl^n\) such
that \(\cylctx^n \vdash \filler_n : \Cyl^n\). An
\emph{\(n\)\=/dimensional cylinder} in a context \(\Gamma\) is a term
\(\Gamma\vdash t : \Cyl^n[\gamma]\) for some substitution
\(\Gamma\vdash \gamma:\partial\cylctx^{n}\). Equivalently, it is a substitution
\(\Gamma\vdash \sub{\gamma,\filler_{n}\mapsto t}:\cylctx^{n}\).

For \(n>1\), a substitution \(\Gamma\vdash\gamma:\partial\cylctx^{n}\) is
completely determined by the image of the four variables \(\top_{n}\),
\(\bot_{n}\) \(\filler_{n-1}^{-}\) and \(\filler_{n-1}^{+}\). Denoting those images respectively as
\(a,b,c,d\), we write \(\Cyl^{n}(a,b,c,d)\) a
shorthand for \(\Cyl^{n}[\gamma]\). Given an \(n\)\=/cylinder
\(\Gamma\vdash t : \Cyl^{n}(a,b,c,d)\), we will use the following notation, which we call respectively the \textit{top}, \textit{bottom}, \textit{back} and \textit{front} of the cylinder:
\begin{align*}
  \topface_{\Gamma}(t) &= a & \botface_{\Gamma}(t) &= b & \back_{\Gamma}(t) &= c & \front_{\Gamma}(t) &= d
\end{align*}
Similarly, a substitution \(\Gamma\vdash\gamma:\partial\cylctx^{1}\) is
determined by the image of \(\bot_{1}\) and \(\top_{1}\), and denoting those images as  \(a,b\)
respectively, we write \(\Cyl^{1}(a,b)\) instead of
\(\Cyl^{1}[\gamma]\), and for a 1\=/cylinder
\(\Gamma\vdash t:\Cyl^{1}(a,b)\), we denote as before:
\begin{align*}
  \topface_{\Gamma}(t) &= \partial_{\Gamma}^{-}t = a
  & \botface_{\Gamma}(t) &= \partial_{\Gamma}^{+}t = b
\end{align*}

\begin{proposition} \label{prop:cyl-type}
  Given a context \(\Gamma\), the type \(\Cyl^{1}(a,b)\) is well-defined if and only if these judgements are derivable:
    \begin{align*}
      \Gamma\vdash a:\obj & & \Gamma\vdash b:\obj
    \end{align*}
The type \(\Cyl^{n}(a,b,c,d)\) is well-defined if and only if \(a, b\)
    are \((n\!-\!1)\)\-cells, and \(c, d\) are \((n\!-\!1)\)\-cylinders satisfying the
    following equations, when they are defined:
    \begin{align*}
      \partial_{\Gamma}^{-}a &= \topface_{\Gamma}(c)
      & \partial_{\Gamma}^{-}b &= \botface_{\Gamma}(c) \\
      \partial_{\Gamma}^{+}a &= \topface_{\Gamma}(d)
      & \partial_{\Gamma}^{+}b &= \botface_{\Gamma}(d) \\
      \back_{\Gamma}(c) &= \back_{\Gamma}(d)
      & \front_{\Gamma}(c) &= \front_{\Gamma}(d)
    \end{align*}
\end{proposition}
We write \(\back_{\Gamma}^{k}(t)\) (resp.~\(\front_{\Gamma}^{k}(t)\)) for the
corresponding operation iterated \(n\!-\!k\) times, to obtain a \(k\)\=/dimensional
term. By convention, if \(k\geq n\), we also denote
\(\back_{\Gamma}^{k}(t) = \front_{\Gamma}^{k}(t) = t\). Using this notation,
we give the general cylinder type formula in Fig.~\ref{fig:cyl-type}.

\paragraph{Cylinder Composites}
The operations \(\back\) and \(\front\) give a notion of  source and target
for cylinders. With respect to this notion, the following theorem shows that cylinders can be composed along
their \(k\)\-boundaries, which we illustrate in the first two parts of Figure~\ref{fig:cyl-comp}.

\begin{figure*}
  \[\begin{split}
    \Cyl^{n+1}(a,b,c,d) = &((((a \s_{0} \front^{1}d) \s_{1}\front^{2}d)
    \s_{2} \ldots) \s_{n-2} \front^{n-1}d) \s_{n-1} d \\
      &\to c \s_{n-1} (\back^{n-1}{c}\s_{n-2}(\ldots \s_{2} (\back^{2}c
    \s_{1}(\back^{1}c \s_{0} b))))
  \end{split}\]
  \caption{Formula for the cylinder type for \(n\geq 1\).}
  \label{fig:cyl-type}
\end{figure*}

\begin{theorem}\label{thm:cyl-comp}
  Consider a context \(\Gamma\vdash\), with an \(m\)\=/cylinder \(a\) and an
  \(n\)\=/cylinder \(b\), together with \(k\leq d = \max(m,n)\) such that
  \(\front_{\Gamma}^{k}(a) = \back_{\Gamma}^{k}(b)\). Then there exists a
  \(d\)\=/cylinder \(a\prescript{m}{}{\cyls}_{k}^{n}b\) in \(\Gamma\),
  satisfying the following, where the superscripts have been omitted when there
  is no ambiguity:
  \begin{gather*}
    \begin{aligned}
      \topface_{\Gamma}(a{\prescript{m}{}\cyls}_{k}^{n}b)
      & = \topface_{\Gamma}(a) \s_{k-1} \topface_{\Gamma}(b)\\
      \botface_{\Gamma}(a{\prescript{m}{}\cyls}_{k}^{n}b)
      & = \botface_{\Gamma}(a) \s_{k-1}\botface_{\Gamma}(b)\\
      \back_{\Gamma}(a\prescript{m}{}\cyls_{k}^{n}b)
      &=
        \begin{cases}
          \back_{\Gamma}(a) & \text{if \(m=n=k+1\)}\\
          \back_{\Gamma}^{d-1}(a) \cyls_{k} \back_{\Gamma}^{d-1}(b) & \text{otherwise}
        \end{cases} \\
      \front_{\Gamma}(a\prescript{m}{}\cyls_{k}^{n}b)
      &=
        \begin{cases}
          \front_{\Gamma}(a) & \text{if \(m=n=k+1\)}\\
          \front_{\Gamma}^{d-1}(a) \cyls_{k}  \front_{\Gamma}^{d-1}(b) & \text{otherwise}
        \end{cases}
    \end{aligned}
  \end{gather*}
\end{theorem}
\begin{proof}

The operation \({}^2_{}{\cyls}{}_{1}^{2}\) is the naturality of the
composition \(f \s_0 g\) with respect to the set \(\{x,f,y,g,z\}\),
described in Example~\ref{ex:natcomp}.
We then define inductively on \(k\) the family of operations
\(\prescript{k+1}{}{\cyls}_{k}^{k+1}\), via the following equality of types, where $\Sigma$ is the suspension meta-operation of Section~\ref{sec:suspopp}:
\[
  \Cyl^{k+1}(a,b,c,d) =
  (\Sigma\Cyl^{k})(a\s_{0}\front_{1}(c),\back_{1}(d)\s_{0}b,c,d)
\]
This lets us give the following definition:
\[
  m\prescript{k+1}{}{\cyls}{}_{k}^{k+1} \,n = j^{-}_{\square,k} \s_{k}
  (\Sigma\prescript{k}{}{\cyls}{}_{k-1}^{k})(m,n) \s_{k} j^{+}_{\square,k}
\]
Here, \(j^{\pm}_{\square,k}\) stands in for the canonical interchanger
that corrects the boundary to yield the required type. The term presented on the left-hand side of
Fig.~\ref{fig:cyl-comp} is the composite of 3\=/cylinders obtained this way.

We then define by induction on \(m\geq k+1\) the operation
\(\prescript{m}{}{\cyls}_{k}^{m}\). We just defined the base case. For the
inductive case, suppose that the composite \(\prescript{m}{}{\cyls}{}_{k}^{m}\)
is defined, and define
\begin{align*}
  \prescript{m+1}{}{\cyls}_{k}^{m+1}
  &=
    \op_{m+1}((a\prescript{m}{}{\cyls}_{k}^{m} \, b)
    \uparrow X)
\end{align*}
where $X=\{\topface(a),\botface(a),a,\topface(b),\botface(b),b\}$.
The term representing the cylinder composite on the  left-hand side of Fig.~\ref{fig:cyl-comp} is obtained
this way.
We then define inductively the operation
\(\prescript{m}{}{\cyls}_{k}^{n}\) for \(m\geq n\). Again, we have already
defined the base case, and the inductive case is given as follows:
\[
  \prescript{m+1}{}{\cyls}_{k}^{n} =
  \op_{m+1}((a\prescript{m}{}{\cyls}_{k}^{n} \,b)
  \uparrow\{\topface(a),\botface(a),a\})
\]
Symmetrically, the operation \(\prescript{m}{}{\cyls}_{k}^{n}\) is defined for \(m\leq n\) by
induction, where the base case has already been defined, and the inductive case
is given as follows:
\[
  \prescript{m}{}{\cyls}_{k}^{n+1} =
  \op_{n+1}\big((a\prescript{m}{}{\cyls}_{k}^{n} \,b)
  \uparrow\{\topface(b),\botface(b),b\} \big) \qedhere
\]

\end{proof}

\begin{figure}
  \centering
    \begin{calign}
    \begin{tikzcd}[ampersand replacement=\&, column sep = 30pt, row sep = 70pt]
      \cdot\ar[d]
      \ar[r, bend left]
      \ar[r, bend right]
      \ar[r,phantom,"\rotatebox{30}{\(\Downarrow\)}"]
      \& \cdot \ar[d]
      \ar[r, bend left]
      \ar[r, bend right]
      \ar[r,phantom,"\rotatebox{30}{\(\Downarrow\)}"]
      \ar[dl,Rightarrow,bend left=15, ""{name=A}]
      \ar[dl,Rightarrow,bend right=15, ""'{name=B}]
      \ar[from=A, to=B, phantom, "\rotatebox{-30}{\(\Lleftarrow\)}"]
      \ar[from=A, to=B, phantom, "{a}", shift right=2.5]
      \& \cdot \ar[d]
      \ar[dl,Rightarrow,bend left=15, ""{name=C}]
      \ar[dl,Rightarrow,bend right=15, ""'{name=D}]
      \ar[from=C, to=D, phantom, "\rotatebox{-30}{\(\Lleftarrow\)}"]
      \ar[from=C, to=D, phantom, "b", shift right=2.5]
      \\
      \cdot
      \ar[r, bend left]
      \ar[r, bend right]
      \ar[r,phantom,"\rotatebox{30}{\(\Downarrow\)}"]
      \& \cdot
      \ar[r, bend left]
      \ar[r, bend right]
      \ar[r,phantom,"\rotatebox{30}{\(\Downarrow\)}"]
      \& \cdot
    \end{tikzcd}
        &
    \begin{tikzcd}[ampersand replacement=\&, column sep = 40pt, row sep = 70pt]
      \cdot \ar[d,""{name=A}]
      \ar[r, bend left=30]
      \ar[r, bend right=30]
      \ar[r]
      \ar[r,phantom,"\rotatebox{30}{\(\Downarrow\)}", bend left = 15]
      \ar[r,phantom,"\rotatebox{30}{\(\Downarrow\)}", bend right = 15]
      \& \cdot \ar[d,""'{name=B}]
      \\
      \cdot
      \ar[r, bend left=30]
      \ar[r, bend right=30]
      \ar[r]
      \ar[r,phantom,"\rotatebox{30}{\(\Downarrow\)}", bend left = 15]
      \ar[r,phantom,"\rotatebox{30}{\(\Downarrow\)}", bend right = 15]
      \& \cdot
      \ar[from=1-2, to=2-1, Rightarrow, bend left=30, ""{name=A}]
      \ar[from=1-2, to=2-1, Rightarrow, bend right=30, shorten <= 20, ""'{name=C}]
      \ar[from=1-2, to=2-1, Rightarrow, shorten <= 5,""{name=B},""'{name=B'}]
      \ar[from=A, to=B', phantom, "\rotatebox{-30}{\(\Lleftarrow\)}"]
      \ar[from=A, to=B', phantom, "b", shift right=2.5]
      \ar[from=B, to=C, phantom, "\rotatebox{-30}{\(\Lleftarrow\)}"]
      \ar[from=B, to=C, phantom, "a", shift right=2.5]
    \end{tikzcd}
        &
    \begin{tikzcd}[ampersand replacement=\&, column sep = 30pt, row sep = 30pt]
       \cdot
       \ar[r, bend left]
       \ar[r, bend right]
       \ar[r,phantom,"\rotatebox{30}{\(\Downarrow\)}"]
       \ar[d]
       \& \cdot
       \ar[d]\\
       \cdot
       \ar[r, bend left]
       \ar[r, bend right]
       \ar[r,phantom,"\rotatebox{30}{\(\Downarrow\)}"]
       \ar[d]
       \& \cdot\ar[d]
       \\
       \cdot
       \ar[r, bend left]
       \ar[r, bend right]
       \ar[r,phantom,"\rotatebox{30}{\(\Downarrow\)}"]
       \& \cdot
       \ar[from=1-2, to = 2-1 ,Rightarrow,bend left=15,
       ""'{name=A}, shorten <= 2, shorten >= 2, crossing over]
       \ar[from=1-2, to = 2-1, Rightarrow,bend right=15,
       ""{name=B}, shorten <= 15, shorten >= 2]
       \ar[from=A, to=B, phantom,"{\rotatebox{-30}{\(\Lleftarrow\)}}"]
       \ar[from=A, to=B, phantom, ,"{\scriptstyle{a}}", shift left=2.5]
       \ar[from=2-2, to = 3-1 ,Rightarrow,bend left=15,
       ""'{name=C}, shorten <= 2, shorten >= 2, crossing over]
       \ar[from=2-2, to = 3-1, Rightarrow,bend right=15,
       ""{name=D}, shorten <= 15, shorten >= 2]
       \ar[from=C, to=D, phantom,"{\rotatebox{-30}{\(\Lleftarrow\)}}"]
       \ar[from=C, to=D, phantom,"{\scriptstyle{b}}", shift left=2.5]
     \end{tikzcd}
         \nonumber
         \\
         a{\cyls}_{1}b & a{\cyls}_{2} b & a \stack b
         \nonumber
\end{calign}
  \caption{Composites and stacking of 3\=/cylinders.}
  \label{fig:cyl-comp}
\end{figure}

\paragraph{Cylinder Stacking}
We also show that cylinders can be composed vertically, which we illustrate in the third part of Fig.~\ref{fig:cyl-comp}.

\begin{theorem}
  Given \(n\)\=/cylinders \(a,b\) in a context \(\Gamma\) such that
  \(\botface(a) = \topface(b)\), there exists an \(n\)\-cylinder
  \(a\stack b\), called the \emph{stacking} and illustrated in
  Fig.~\ref{fig:cyl-comp}, satisfying the equations $\topface(a\stack b) = \topface(a)$ and $\botface(a\stack b) = \botface(b)$.
  Moreover, if $n\ge 2$
  \begin{align*}
      \back(a\stack b) &= \back(a) \stack \back(b)
    & \front(a\stack b) &= \front(a) \stack \front(b)
\end{align*}
\end{theorem}
\begin{proof}
  We proceed by induction on the dimension. We define the stacking of
  \(1\)\=/cylinders to be their composites as \(1\)\=/cells. We then define
  the stacking of \((n\!+\!1)\)\=/cylinders \(\stack_{n}\) by the naturality
  construction:
  \[
    \stack_{n} =
    \op_{n}\big( (a\stack_{n-1}b)\uparrow\{\topface(a),\botface(a),a,\botface(b),b\} \big)
\qedhere
  \]
\end{proof}


%% file: sections/cones.tex

\section{Cones and Conical Compositions}\label{sec:cones}

\noindent
Here we describe the geometry of cones, and give a new construction of cone composites, extending the introductory sketch of Example~\ref{ex:cone-nat}. We first introduce the cone contexts and cone types,  and then show that  cones admit composites. We believe this is the first description of cone composites in the weak $\omega$\-category setting.


\paragraph{Cone Contexts}
We define by induction on \(n\) the contexts \(\conectx^{n}\) and
\(\partial\conectx^{n}\) that capture the geometry of cones. We define these
 together with a base variable \(\base_{n}\), an apex variable \(\apex\),
and a filler variable \(\filler_{n}\), using our naturality
procedure:
\begin{align*}
  \partial\conectx^{1} &\!=\! (\apex : \obj, \base_{1} : \obj)
  &
  \conectx^{1} &\!=\! (\partial\conectx^{1}, \filler_{1} \!:\!
                   \base_{1} \!\!\to\!\! \apex)
                   \\
  \partial\conectx^{n+1} &\!=\! \conectx^{n}\!\pmctx\!\{\base_{n},\filler_{n}\}
  &
  \conectx^{n\!+\!1} &\!=\! {\conectx^{n}\!} \uparrow\! \{ {\base_{n},\filler_{n}} \}
\\[-3pt]
  \filler_{n+1} &\!=\! \fun{\filler_{n}}
  & \hspace{-1cm}\base_{n+1} &\!=\! \fun{\base_{n}}
\end{align*}
Fig.~\ref{fig:cone-ctx} gives a visual representation of
 \(\conectx^{2}\) and \(\conectx^{3}\).

\begin{figure}
  \centering
  \[
    \begin{array}{c@{\quad}c}
      \conectx^{2} :
      \begin{tikzcd}[ampersand replacement=\&, column sep=tiny]
        \& \top \& \\
        \base_1^- \ar[ru, "\filler_1^-"]\ar[rr,"\base_2"',""{above, name=A}]
        \& \& \base_1^+ \ar[lu,"\filler_{1}^{+}"']
        \ar[from=1-2, to=A, phantom, "\overset{\filler_{2}}{\Rightarrow}"]
      \end{tikzcd} &
      \conectx^{3} :
      \begin{tikzcd}[ampersand replacement=\&, column sep=tiny]
        \& \top \& \\ \\
        \base_1^- \ar[ruu, "\filler_1^-",""'{name=A}]
        \ar[rr, bend left=15]
        \ar[rr, bend right=15]
        \ar[rr,phantom,"\Downarrow"]
        \& \& \base_1^+ \ar[luu,"\filler_{1}^{+}"',""{name=B}]
        \ar[from=A, to=B,bend left = 15, Rightarrow, shift left=2, shorten >=
        2pt, shorten <= 3pt]
        \ar[from=A, to=B,bend right = 15, Rightarrow, shift right=2, shorten >=
        2pt, shorten <= 3pt]
        \ar[from=A, to=B,phantom,"\rotatebox{-90}{\(\Rrightarrow\)}"]
      \end{tikzcd}
    \end{array}
  \]
  \caption{The first cone contexts.}
  \label{fig:cone-ctx}
\end{figure}

By construction, there exists a type \(\partial\conectx^n \vdash \Cone^n\) such
that \(\conectx^n \vdash \filler_n : \Cone^n\). An
\emph{\(n\)\=/dimensional cone} in a context \(\Gamma\) is a term
\(\Gamma\vdash t : \Cone^n[\gamma]\) for some substitution
\(\Gamma\vdash \gamma:\partial\conectx^{n}\). Equivalently, it is a substitution
\(\Gamma\vdash \sub{\gamma,\filler_{n}\mapsto t}:\conectx^{n}\).

A substitution \(\Gamma\vdash\gamma:\partial\conectx^{n}\) for \(n\geq 2\) is
completely determined by the image of the  variables \(\base_{n}\),
\(\filler_{n-1}^{-}\) and \(\filler_{n-1}^{+}\). Denoting those images as \(a,b,c\) respectively, we write \(\Cone^{n}(a,b,c)\) as a shorthand for
\(\Cone^{n}[\gamma]\). Moreover, given such a cone
\(\Gamma\vdash t:\Cone^{n}(a,b,c)\), we call the following the \textit{base}, \textit{back} and \textit{front} of the cone:
\begin{align*}
  \basecone_{\Gamma}(t) & = a
  & \backcone_{\Gamma}(t) & = b
  & \frontcone_{\Gamma}(t) &= c
\end{align*}
Similarly, a
substitution \(\Gamma\vdash \gamma:\partial\conectx^{1}\) is determined by the
images of \(\base_{1}\) and \(\top\), and if \(a\) and \(b\) are these images,
we write \(\Cone^{1}(a,b)\) instead of \(\Cone^{1}[\gamma]\), and for
\(\Gamma\vdash t : \Cone^{1}(a,b)\), we write $\basecone_{\Gamma}(t) = a$ and $\backcone_{\Gamma}(t) = \frontcone_{\Gamma}(t) = b$.

\begin{figure*}
  \[
    \Cone^{n+1}(a,b,c) =
    \begin{cases}
      b \to (\backcone^{n-1}(b)\s_{n-2}((\ldots \s_3(\backcone^{2}(b)\s_{1} \\
      (a\s_{0}\frontcone^{1}(c)))\s_2\ldots) \s_{n-3}\frontcone^{n-2}(c)))\s_{n-1}c
      & \text{if \(n\) is odd}\\
      b\s_{n-1}((\frontcone^{n-2}(b)\s_{n-3}(\ldots \s_3(\backcone^{2}(b)\s_{1}\\
      (a\s_{0}\frontcone^{1}(c)))\s_2\ldots))\s_{n-2}\backcone^{n-1}(c)) \to c
      & \text{if \(n\) is even}
    \end{cases}
  \]
  \caption{Formula for the cone type for \(n \geq 1\).}
  \label{fig:cone-type}
\end{figure*}

\begin{proposition}~\label{prop:cone-type}
  Given a context  \(\Gamma\), the type \(\Cone^{1}(a,b)\) is well-formed iff $\Gamma\vdash a,b:\obj$. Furthermore, for \(n\geq2\), the type \(Cone^{n}(a,b,c)\) is well-defined iff  \(a\) is an
    \((n\!-\!1)\)\=/cell, and \(b\) and \(c\) are \((n\!-\!1)\)\=/cones, satisfying the cone
    equations:
    \begin{align*}
      \partial^{-}a &= \basecone(b) & \partial^{+}a &= \basecone(c)\\
      \backcone(b) &= \backcone(c) & \frontcone(b) &= \frontcone(c)
    \end{align*}
\end{proposition}

\paragraph{Cone Composites}
We denote by \(\frontcone^{k}\) and \(\backcone^{k}\) the respective operations
iterated \(({n\!-\!k})\) times producing a cone of dimension \(k\).
Fig.~\ref{fig:cone-type} gives an explicit formula for the type
\(\Cone^{n}(a,b,c)\), recovering that of Buckley and Garner~\cite{buckley_orientals_2016}. With this, we give a definition of cone composites as follows, as shown in Fig.~\ref{fig:cone-comp}.

\begin{theorem}\label{thm:cone-comp}
  Given a context \(\Gamma\) with an \(m\)\=/cone \(a\) and an \(n\)\=/cone \(b\)
  such that \(\frontcone_{\Gamma}^{k}a = \backcone_{\Gamma}^{k} b\) for some
  \(k < d=\max(m,n)\), there exists a \(d\)\=/cone
  \(a\prescript{m}{}\cs_{k}^{n}b\) such that:
  \begin{gather*}
    \begin{aligned}
      \basecone(a\prescript{m}{}\cs_{k}^{n}b)
      &= \basecone(a) \s_{k-1} \basecone(b) \\
      \backcone(a\prescript{m}{}\cs_{k}^{n}b)
      &=
        \begin{cases}
          \backcone(a) & \text{if \(m=n=k+1\)}\\
          \backcone^{d-1}(a) \cs_{k} \backcone^{d-1}(b) & \text{otherwise}
        \end{cases} \\
      \frontcone(a\prescript{m}{}\cs_{k}^{n}b)
      &=
        \begin{cases}
          \frontcone(b) & \text{if \(m=n=k+1\)}\\
          \frontcone^{d-1}(a) \cs_{k} \frontcone^{d-1}(b) & \text{otherwise}
        \end{cases}
    \end{aligned}
  \end{gather*}
\end{theorem}

\begin{proof}
  We work over the context \(\Gamma\) of
  Example~\ref{ex:cone-nat}, and define the operation \(a \,{}^{2}{} \hspace{-1pt} {\cs}{}_{1}^{2} \, b\) directly as
  a term of type \({\Cone^{2}(f\s_{0}g,h,l)}\), where $\alpha'$ is the appropriate associator:
\[
  a \,{}^{2}{} \hspace{-1pt} {\cs}{}_{1}^{2} \, b =
  a\s_{1}(f\s_{0}b)\s_{1}\alpha'_{f,g,l}
\]
We then define inductively on \(k\) the family of terms
\({}^{k+1}_{\vphantom k} \hspace{-1pt} {\cs} {}_{k}^{k+1}\). To approach this, we notice the following
equality on types, where \(\bar{k} = \{1,\ldots,k\}\):
\[
  \Cone^{k+1}(a,b,c) =
  (\Sigma(\op_{\bar{k}}\Cone^{k}))((a\s_{0}\frontcone^{1}(c)),c,b)
\]
 We  define inductively the
terms \({}^{k+1}{} \hspace{-1pt} {\cs}{}_{k}^{k+1}\) as follows, where \(j^\pm_{\triangle,k}\) are canonical interchangers giving the desired type:
\[
  m \, {}^{k+1} \hspace{-1pt} {\cs}{}_{k}^{k+1} \,n =
  \begin{cases}
    (\Sigma(\op_{\bar{k}}{}^{k} \hspace{-1pt} {\cs}{}_{k-1}^{k}))(m,n)\s_{k} j^+_{\triangle,k}
    & \text{\(k\) is odd} \\
    j^-_{\triangle,k}\s_{k}(\Sigma(\op_{\bar{k}}{}^{k} \hspace{-1pt}{\cs}{}_{k-1}^{k}))(m,n)
    & \text{\(k\) is even}
  \end{cases}
\]
The operations \({}^{m}{} \hspace{-1pt} {\cs} {}_{k}^{m}\) are defined by induction, for
\({m\geq k+1}\). We have already defined the base case, and the inductive case is
given as follows:
\[
  {}^{m+1}{}{\cs}_{k}^{m+1} = (a\,{}^{m}{}\hspace{-1pt}{\cs}_{k}^{m}\,b)
  \uparrow\{\basecone(a),a,\basecone(b),b\}
\]
We then define \({}^{m}_{\vphantom k} \hspace{-1pt}{\cs}{}_{k}^{n}\) for \(m > n\) by the inductive
formula:
\[
  \prescript{m+1}{}{\cs}_{k}^{n} = (a \, {}^{m}_{}\hspace{-1pt}{\cs}_{k}^{n} \, b)
  \uparrow\{\basecone(a),a\}
\]
Likewise we define
\({}^{m} \hspace{-1pt} {\cs}{}_{k}^{n}\) for \(n > m\) by the inductive
formula
\[
  {}^{m}{} \hspace{-1pt} {\cs} {}_{k}^{n+1} = (a \, {}^{m}{} \hspace{-1pt}{\cs}_{k}^{n} \, b)
  \uparrow\{\basecone(b),b\} \qedhere
\]
\end{proof}
\begin{figure}
  \centering
 \begin{calign}
\begin{aligned}
\begin{tikzpicture}[tikzpic, yscale=2, xscale=1.1, scale=1]
      \node (00) at (0,0) {$\cdot$};
      \node (10) at (1,0) {$\cdot$};
      \node (20) at (2,0) {$\cdot$};
      \node (11) at (1,1) {$\top$};
      \def\bd{27}
      \draw [->] (00) to [bend left=\bd] (10);
      \draw [->] (00) to [bend right=\bd] (10);
      \draw [->] (10) to [bend left=\bd] (20);
      \draw [->] (10) to [bend right=\bd](20);
      \draw [twoarr] (.1,.5) to [bend right=10] +(.9,0);
      \draw [twoarr] (.23,.65) to [bend left=10] +(.76,0);
      \draw [triple] (.6,.62) to node [right=0pt, pos=0.4] {$a$} +(.0,-.1);
      \draw [twoarr, {Implies}-] (1.9,.5) to [bend left=10] +(-.91,0);
      \draw [twoarr, {Implies}-] (1.75,.65) to [bend right=10] +(-.77,0);
      \draw [triple] (1.4,.62) to node [right=0pt, pos=0.45] {$b$} +(.0,-.1);
      \draw [twoarr] (.5,.15) to +(0,-.3);
      \draw [twoarr] (1.5,.15) to +(0,-.3);
      \draw [-, line width=4pt, white] (00) to [out=90, in=-165] (11);
      \draw [->] (00) to [out=90, in=-165] (11);
      \draw [->] (20) to [out=90, in=-15] (11);
      \draw [-, line width=4pt, white] (11) to (10);
      \draw [<-] (11) to (10);
    \end{tikzpicture}
\end{aligned}
&
\begin{aligned}
\begin{tikzpicture}[tikzpic, yscale=2, xscale=1.1, scale=1]
      \node (00) at (0,0) {$\cdot$};
      \node (20) at (2,0) {$\cdot$};
      \node (11) at (1,1) {$\top$};
      \def\bd{17}
      \draw [->] (00) to [bend left=\bd] (20);
      \draw [->] (00) to [bend right=\bd] (20);
      \draw [->] (00) to (20);
      \draw [twoarr] (.50,.4) to [bend right=\bd] +(1,0);
      \draw [twoarr] (.6,.5) to +(0.8,0);
      \draw [twoarr] (.70,.65) to [bend left=\bd] +(.6,0);
      \draw [triple] (1,.65) to node [right=0pt, pos=0.4] {$a$} +(.0,-.13);
      \draw [triple] (1,.48) to node [right=0pt, pos=0.4] {$b$} +(.0,-.13);
      \draw [twoarr] (1,.19) to +(0,-.18);
      \draw [twoarr] (1,-.01) to +(0,-.18);
      \draw [->] (00) to (11);
      \draw [->] (20) to (11);
    \end{tikzpicture}
    \end{aligned}
\nonumber
\\[0pt]
a \, {\cs}{}_{1} \, b & a \, {\cs}_{2} \, b
\nonumber
\end{calign}
  \caption{The two composites of $3$-cones}
  \label{fig:cone-comp}
\end{figure}


%% file: sections/implementation.tex

\section{Implementation}\label{sec:implementation}

\noindent
We have implemented our naturality construction, and in the supplementary material we include the following resources.
\begin{itemize}
\item The proof-assistant \catt, implemented in OCaml, patched with an implementation of naturality as a meta-operation, as well
as built-ins for computing cylinder composites, cylinder stacking, and cone
composites.
\item A readme file with full details on building the proof assistant. This works on both Linux and Mac, and on Windows in the WSL subsystem.
\item \catt files encoding  the examples from the introduction.
\item Artifacts storing a  range of pre-computed cylinder composites, cylinder stackings, and cone composites.
\end{itemize}

\paragraph{Syntax for Naturality}
For our implementation of naturality, we use a simple syntax which decorates the notation for substitution. We illustrate this with the naturality construction of Example~\ref{ex:nat-assoc}. Assuming the associator has already been defined, the following command constructs naturality of the associator with respect to the first argument~$f$:

{\small\begin{lstlisting}[language=catt]
let assoc_nat (x y z w : *) (f_minus : x -> y)
  (f_plus : x -> y)
  (f_arrow : f_minus -> f_plus)
  (g : y -> z) (h : z -> w)
  = assoc [f_arrow] g h
\end{lstlisting}}

\vspace{-5pt}
\noindent
Here \verb|let| is a keyword declaring a new variable, and \verb|assoc_nat| is our chosen variable name. We then write out the functorialised context, which agrees with $\Gamma \uparrow \{ f \}$ of Example~\ref{ex:nat-assoc}. The naturality term $\alpha_{f,g,h} \uparrow \{ f \}$ is then computed with the command \verb|assoc [f_arrow] g h|, with the notation \verb|[|...\verb|]| indicating the variable chosen for the subset $X$. Some more advanced syntax features are explained in the file \verb|./examples/syntax.catt|.

\paragraph{Examples}
Once the proof assistant has been compiled (see the readme), all of the examples from the introduction can be directly executed, with a command such as:

{\small\begin{lstlisting}
dune exec catt ./examples/example_1.catt
\end{lstlisting}}

\noindent
Each such file directly encodes the relevant example, using the naturality syntax described above. The proof assistant will execute the instructions, automatically type-checking any terms which are constructed. If this completes without error, a success message will be printed to the console.

\paragraph{Built-Ins} The proof assistant has built-in syntactic wrappers to evaluate our cylinder and cone constructions, with the syntax
\lstinline[language=catt]|cylcomp(_,_,_)|, \lstinline[language=catt]|cylstack(_)| and \lstinline[language=catt]|conecomp(_,_,_)|. For example, \verb|cylcone(3,1,3)| computes the cylinder composite \(a \,{}^{3}_{\vphantom{k}}{\cyls}{}_{1}^{3} \,b\). These can be used applied to arguments like any other terms, or used with the keyword \verb|check| which typesets and print them. For example, the following code
prints the term corresponding to the cone composite ${}^{2}_{} \hspace{-1pt}{\cs} {}^2_1$
defined in Section~\ref{sec:cones}: \verb|check conecomp(2,1,2)|.

\paragraph{Pre-Computed Artifacts}
Beyond the first few cases,  the built-ins can yield large output artifacts, and memory required for the type-checker can easily exceed the capacity of a typical workstation. Pre-computed artifacts are available under the directories \verb|cylinder-comp|, \verb|cylinder-stack| and \verb|cone-comp| within the \verb|results| directory, which the reader may like to inspect.

In Example~\ref{ex:cylinder-nat} we give the artifact sizes for the cylinder horizontal composites \({}^{n}_{\vphantom{k}}{\cyls}{}_{1}^{n}\), and we have been able to compute some even larger examples; for example, \({}^{6}_{\vphantom{k}}{\cyls}{}_{2}^{6}\) has size 1,015,770. For the cylinder vertical stackings \(\stack_{n}\), the artifact sizes in bytes likewise grow rapidly:
\begin{align*}
n=2&&n=3&&n=4 && n=5
\\
\text{830} && \text{9,866} && \text{66,226} && \text{509,902}
\end{align*}
We observe a similar pattern for the cone  composites
${}^{n}_{} \hspace{-1pt}{\cs} \hspace{-1pt}{}^n_1$:
\begin{align*}
n=2&&n=3&&n=4 && n=5 && n=6
\\
\text{501} && \text{3,211} && \text{16,530} && \text{87,673} && \text{745,617}
\end{align*}
Further optimisation of the code base may allow larger instances to be computed. Such artifacts can enable further computational investigation of these rich mathematical structures, both in \catt and in \hott.



%% file: appendix/composites.tex

\section{Naturality of Coherences}\label{app:nat-coh}

In this appendix, we provide the
full definition of the naturality of linear compositions of dimension
\(1\), and reduced compositions. We first treat the case of linear
compositions, defining formally the contexts over which they live, and
then the \emph{phases} of the naturality of linear compositions, which are given by
appropriate associators and whiskerings.

\begin{definition}
    For every $n,k \in \N$, we define the \emph{linear context}
    $\Psi^n_k$ by recursion via the following
    \begin{align*}
        \Psi^0_0 &= (x_0 : \obj) \\
                \Psi^0_{k+1} &= (\Psi^0_k,x_{k+1}:\obj, f_k :x_k \to x_{k+1}) \\
        \Psi^{n+1}_k &= \Sigma \Psi^n_k
    \end{align*}
\end{definition}

\noindent Those contexts for small values of \(k\) and \(n\) are illustrated in
Fig.~\ref{fig:linear-ps}. We now introduce the necessary associators.

\begin{definition}~\label{def:assoc-lincomp}
    For $k > 0$ and $0\le j \le k+1$, we define the following \emph{biased
    compositions} over the context \(\Psi^0_{k+1}\):
    \begin{align*}
      \bcomp_{k,0} &= f_0 \s_0 (f_1 \s_0 \cdots \s_0 f_k) \\
      \bcomp_{k,j} &= f_0 \s_0 \cdots \s_0 (f_{j-1} \s_0 f_{j}) \s_0 \cdots \s_0 f_k \\
      \bcomp_{k,k+1} &= (f_0 \s_0 \cdots \s_0 f_{k-1}) \s_0 f_k \\
      \intertext{and for \(j \leq k\),  the associator:}
      \assoc_{k,j} &= \coh_{\Psi^0_{k+1},\arr{\bcomp_{k,j+1}}{\bcomp_{k,j}}}[\id]
    \end{align*}
\end{definition}

We then define the substitutions
$\Psi^0_k \uparrow X \vdash \psi_{k,j}^X : \Psi^0_{k+1}$
whenever $x_j \in X$, which will be used to apply the associators to the right
arguments. An instance of such substitution is described in
Fig.~\ref{fig:scan-sub}.

\begin{definition}\label{def:scan-subst}
  Let $k \in \N$, $0 \le i,j \le k+1$ and $X \in \U(\Psi^0_k)$ such that $x_j \in X$,
  writing \({\inc^{\pm} = \inc^{\pm}_{\Psi^{0}_{k},X}}\),
  we define the raw substitution $\psi^X_{k,j,i}$ by recursion
\begin{align*}
        \psi_{k,j,0}^X &= \sub{x_0 \mapsto x_0[\inl]} \\
    \psi_{k,j,i+1}^X &= \sub{\psi_{k,j,i}^X, x_{i+1} \mapsto x_{i+1}[\inl], f_{i} \mapsto f_{i}[\inl]} & \text{if } i < j \\
        \psi_{k,j,i+1}^X &= \sub{\psi_{k,j,i}^X, x_{i+1} \mapsto x_{i}[\inr], f_{i} \mapsto \fun{x_j}} & \text{if } i = j \\
    \psi_{k,j,i+1}^X &= \sub{\psi_{k,j,i}^X, x_{i+1} \mapsto x_{i}[\inr], f_{i} \mapsto f_{i-1}[\inr]} & \text{if } i > j
\end{align*}
        We then define the substitution $\psi^X_{k,j}$ to be $\psi^X_{k,j,k+1}$.
\end{definition}

We then define the whiskering phases, which are parameterised by 1\=/cells
of~\(X\).

\begin{definition}\label{def:whisk-lincomp}
        Let $0 \le j < k$ and $X \in \U(\Psi^0_k)$ such that $f_j \in X$.
    We define the \emph{whiskering phase} \(w_{k,j}^X\) to be the unbiased composition
    \begin{align*}
                w_{k,0}^X &= \fun{f_0} \s_0 f_{1}[\inr_{\Psi^0_k,X}] \s_0 \cdots \s_0 f_{k-1}[\inr_{\Psi^0_k,X}] \\
                w_{k,j}^X &= f_0[\inl_{\Psi^0_k,X}] \s_0 \cdots \s_0 \fun{f_j} \s_0 \cdots \s_0 f_{k-1}[\inr_{\Psi^0_k,X}] \\
                w_{k,k}^X &= f_0[\inl_{\Psi^0_k,X}] \s_0 \cdots \s_0 f_{k-2}[\inl_{\Psi^0_k,X}] \s_0 \fun{f_{k-1}}
    \end{align*}
\end{definition}

Finally, we assemble the phases together to define the naturality of linear compositions.
For that, we will use that the variables of \(\Psi^0_k\) are linearly ordered by
the transitive closure of the relation \(x_j \prec f_j \prec x_{j+1}\). This is a
special case of the order of the variables in  pasting context described by
Street~\cite{street_petit_2000}, and Finster and Mimram~\cite{finster_typetheoretical_2017}.

\begin{definition}\label{def:nat-lincomp}
    Let \(k > 1\) and \(X = \{v_0\prec \dots\prec v_n\}\in\U(\Psi^0_k)\). The
    naturality of the linear composition \(\comp^0_k\) with respect to \(X\) is
    the unbiased composite
    \begin{align*}
        \comp^0_k \uparrow X &= h_{k,v_n}^X \s_0 \cdots \s_0 h_{k,v_0}^X \\
    \intertext{where the phases $h_{k,v}^X$ are defined by}
        h_{k,v}^X &=
        \begin{cases}
                \assoc_{k,j}[\psi_{k,j}^X] &\text{ if } v = x_j \\
                w_{k,j}^X &\text{ if } v = f_j
        \end{cases}
    \end{align*}
\end{definition}

\begin{figure}
  \[\begin{tikzcd}[row sep = large]
    {x_0^-} & {x_1^- } & {x_2} \\
    {x_0^+} & {x_1^+}
    \arrow["{f_0^-}", from=1-1, to=1-2]
    \arrow["{\fun{x_0}}"', dotted, from=1-1, to=2-1]
    \arrow["{\fun{f_0}}", Rightarrow, dotted, from=1-1, to=2-2]
    \arrow["{f_1^-}", dotted, from=1-2, to=1-3]
    \arrow["{\fun{x_1}}", from=1-2, to=2-2]
    \arrow["{f_0^+}"', dotted, from=2-1, to=2-2]
    \arrow[""{name=0, anchor=center, inner sep=0}, "{f_1^+}"', from=2-2, to=1-3]
    \arrow["{\fun{f_1}}", dotted, shorten >=3pt, Rightarrow, from=1-2, to=0]
  \end{tikzcd}\]
\caption{The substitution \(\psi^{\{x_0,x_1,f_0,f_1\}}_{2,1}\).}
\label{fig:scan-sub}
\end{figure}

\noindent This completes the definition of the naturality of linear compositions
of dimension \(1\). The definition for higher dimensions is obtained by suspension,
as explained in Section~\ref{sec:naturality}.

We now proceed to define the naturality of reduced
compositions. For that, suppose that we are given a reduced pasting context
\(\Gamma \vdashps\) of dimension \(n>0\), a full type \(A = \arr[B]{u}{v}\) of
dimension \((n-1)\), and a set \(X\in \U(\Gamma)\) such that
\(\depth_X(\Gamma) = \depth_X(\coh_{\Gamma,A}[\id]) = 1\). We will denote by
\(X^m\) the set of variables of \(X\) of dimension \(n\), by \(X^{lm}\) the set
of variables of \(X\) of depth \(0\), by \(X^l = X^{lm} \setminus X^m\) the set
of variables of depth \(0\) and dimension less than \(n\), and finally by
\(X^b = X \setminus X^{lm}\) the set of variables of \(X\) of depth \(1\) in
\(X\). Due to the depth and the fullness condition, we can infer
that \(\partial^-\Gamma \vdash u : B\) and \(\partial^+\Gamma \vdash v : B\),
which will be used implicitly in the construction.

Recall that the naturality of a reduced composition is given by the formula
\[
  \coh_{\Gamma,A}\uparrow X = j^-_{\Gamma,X,A}
  \s_n (\coh_{\Gamma,A}\uparrow X^{lm})[\theta_{\Gamma,X}]
  \s_n j^+_{\Gamma,X,A}
\]
for a substitution \(\theta_X\) and a pair of interchangers
\(j^\pm_{\Gamma,X,A}\), which we will define below.

We first define the substitution
\(\Gamma\uparrow X \vdash \theta_{\Gamma,X}:\Gamma\uparrow X^{lm}\). This
substitution sends the functorialised variables in \(\Gamma\uparrow X^{lm}\) to
their counterpart in \(\Gamma\uparrow X\), as shown in
Fig.~\ref{fig:reduced-comp}. The subtlety is that the type of said variables
in the two contexts may differ.

\begin{definition}\label{def:theta-subs}
    The substitution $\theta_{\Gamma,X}$ is defined by
    \begin{align*}
      \fun x[\theta_{\Gamma,X}] &= \fun x &\text{if } x \in X^{lm} \\
      y[\theta_{\Gamma,X}] &= y &\text{if } y \in \Var(\Gamma)\setminus X \\
      x^\pm[\theta_{\Gamma,X}] &= \partial^\pm(\fun x) &\text{if } x \in X^{lm} \\
      y[\theta_{\Gamma,X}] &= \partial^\varepsilon(\fun y) &\text{if } y \in X^b
    \end{align*}
    where \(\varepsilon =\pm \) when \(y \in \partial^\pm\Gamma\).
\end{definition}

\noindent We will now define the interchangers \(j^\pm_{\Gamma,X,A}\). In the
case that
\(\Var(v)\cap X = \Var(\partial^+\Gamma) \cap X = \emptyset\),
the interchanger \(j^-_{\Gamma,X,A}\) is required to have type
\[
    \coh_{\Gamma,A}[\inl_{\Gamma,X}] \to
    \coh_{\Gamma,A}[\inl_{\Gamma,X^{lm}} \circ\ \theta_{\Gamma,X}]
\]
in \(\Gamma\uparrow X\). The source and target of the interchanger are equal by
Lemma~\ref{lem:interchanger-trivial}. Therefore, we take it to be the coherence
given by
\[
  j_{\Gamma,X,A}^- = \coh_{\Gamma,\coh_{\Gamma,A}[\id] \to \coh_{\Gamma,A}[\id]}[\inl_{\Gamma,X}]
\]
The case where \(\Var(v)\cap X \neq \emptyset\) is more involved. In this case,
the interchanger \(j^-_{\Gamma,X,A}\) needs to have the following type:
\begin{align*}
  \coh_{\Gamma,A}[\inl_{\Gamma,X}] &\s_{n-1} (v\uparrow X) \to \\
  &(\coh_{\Gamma,A}[\inl_{\Gamma,X^{lm}}] \s_{n-1} (v\uparrow X^{lm})) [\theta_{\Gamma,X}]
\end{align*}
To define the context that it lives over, we use the composition of pasting
contexts \(\otimes_n\) introduced for trees by Batanin~\cite{batanin_monoidal_1998}
and described in detail by Weber~\cite{weber_generic_2004}. We let
\(Y = (\delta^+_\Gamma)^{-1}X = \Var(\partial^+\Gamma)\cap X\) and similarly
let \(Y^l\) and \(Y^b\) be the intersections of \(X^l\) and \(X^b\) with the
variables of \(\partial^+\Gamma\). This allows us to define the following
pasting contexts:
\begin{align*}
  \Phi &= \Gamma \otimes_{n-1} (\partial^+\Gamma \uparrow Y) \\
  \Psi &= \Gamma \otimes_{n-1} (\partial^+\Gamma \uparrow Y^b)
\end{align*}
using Lemma \ref{lem:psctx-func} to show pasting contexts are closed under
functorialisation. Those pasting contexts fit into pullback squares in the syntactic category of the type theory \catt{}, as depicted in Figure \ref{fig:intch-ctx}.
This is a consequence of the definition of composition of pasting contexts, and
of Lemma \ref{lem:psctx-func} which shows that the substitutions \(\inl_{\partial^+\Gamma,Y^b}\) and \(\inl_{\partial^+\Gamma,Y}\)
are source inclusions, since the variables in $Y$ are of maximal dimension. The triangle at the bottom
commutes by Lemma~\ref{lem:func-of-func}.
Moreover, by Lemma \ref{lem:pullback-stability}, the other square commutes.
This allows us to define the
substitutions \(\phi\) and \(\psi\) using the universal property of the pullbacks.
The interchanger \(j^-_{\Gamma,X,A}\) is then defined as the coherence
\begin{gather*}
  p = \coh_{\Gamma,A}[\proj_1]
    \s_{n-1} (v\uparrow Y)[\proj_2] \\
  q = \coh_{\Gamma,A}[\rho_\Psi\circ \psi]
    \s_{n-1} (v\uparrow Y^l)[\inr_{\partial^+\Gamma \uparrow Y^l,Y^b}\circ \proj_2]\\
  j^-_{\Gamma,X,A} = \coh_{\Phi,\arr p q}[\phi]
\end{gather*}
where \(\Psi \vdash \rho_\Psi : \Psi^r\) is the reduction substitution. For the
definition of \(q\) we further use that \(\Gamma = \Psi^r\) up to
\(\alpha\)\=/equivalence, which will be shown in Lemma~\ref{lem:psi-r-is-gamma}.

The interchanger \(j^+_{\Gamma,X,A}\) is defined dually. It is given when
\(\Var(u)\cap X = \emptyset\) by
\[
  j_{\Gamma,X,A}^+ =
  \coh_{\Gamma,\coh_{\Gamma,A}[\id] \to \coh_{\Gamma,A}[\id]}[\inr_{\Gamma,X}]
\]
and otherwise, it is defined as the following coherence:
\begin{gather*}
  Y = \Var(\partial^-\Gamma)\cap X \\
  \Phi = (\partial^-\Gamma \uparrow Y) \otimes_{n-1}\Gamma \\
  \Psi = (\partial^-\Gamma \uparrow Y^b) \otimes_{n-1}\Gamma \\
  p = (u\uparrow Y)[\proj_1] \s_{n-1} \coh_{\Gamma,A}[\proj_2] \\
  q = (u\uparrow Y^b)[\inr_{\partial^-\Gamma \uparrow Y^l,Y^b}\circ \proj_1]
    \s_{n-1}\coh_{\Gamma,A}[\rho_\Psi\circ \psi] \\
  j^+_{\Gamma,X,A} = \coh_{\Phi,\arr q p}[\phi]
\end{gather*}
where the morphism \(\phi\) and \(\psi\) are defined again by the universal
property of the pullbacks defining \(\Phi\) and \(\Psi\).

\begin{figure}
\centering
\[\begin{tikzcd}[column sep = 35]
        {\Gamma \uparrow X} \\
        &[-20] \Phi & \Psi & \Gamma \\
        & {\partial^+\Gamma \uparrow Y} & {\partial^+\Gamma \uparrow Y^b} & {\partial^+\Gamma}
        \arrow["\phi"{description}, dashed, from=1-1, to=2-2]
        \arrow["{\inl_{\Gamma,X}}", shift left=2, curve={height=-18pt}, from=1-1, to=2-4]
        \arrow["{\delta^+_\Gamma \uparrow X}"', shift right=2, curve={height=12pt}, from=1-1, to=3-2]
        \arrow["\psi"{description}, dashed, from=2-2, to=2-3]
        \arrow["{\proj_1}", shift left, curve={height=-12pt}, from=2-2, to=2-4]
        \arrow["{\proj_2}"', from=2-2, to=3-2]
        \arrow["\lrcorner"{anchor=center, pos=0.125}, draw=none, from=2-2, to=3-3]
        \arrow["{\overline{\proj_1}}"', from=2-3, to=2-4]
        \arrow["{\overline{\proj_2}}", from=2-3, to=3-3]
        \arrow["\lrcorner"{anchor=center, pos=0.125}, draw=none, from=2-3, to=3-4]
        \arrow["{\delta^+_\Gamma}", from=2-4, to=3-4]
        \arrow["{\inl_{\partial^+\Gamma \uparrow Y^b, Y^l}}", from=3-2, to=3-3]
        \arrow["{\inl_{\partial^+ \Gamma,Y}}"', curve={height=18pt}, from=3-2, to=3-4]
        \arrow["{\inl_{\partial^+\Gamma, Y^b}}", from=3-3, to=3-4]
\end{tikzcd}\]
        \caption{The construction of $\Phi$ and $\Psi$ in $j_{\Gamma,X,A}^-$.}
        \label{fig:intch-ctx}
\end{figure}

\section{Correctness of Naturality}\label{app:correctness}

In this appendix, we prove a sequence of lemmas that are used for the correctness
theorem~\ref{thm:correctness}. Those lemmas can be separated in two groups: ones
that are used to prove the correctness of the inductive scheme, and ones that
are used to prove the correctness for coherences. The way that those proofs fit
together is explained in the proof of Theorem~\ref{thm:correctness}.

\begin{proof}[Proof of Lemma~\ref{lem:empty-intersection}]
  The first two points are proven by mutual induction. The first one is
  true for variables by definition of the substitutions \(\inc_{\Gamma,X}^\pm\).
  It is true on coherence terms \(t = \coh_{\Delta,A}[\sigma]\) by the inductive
  hypothesis on \(\sigma\). The second one is definitionally true when
  \(\sigma = \sub{}\). When \(\sigma = \sub{\sigma',y\mapsto t}\), it follows
  from the inductive hypothesis on \(\sigma'\) and \(t\). Finally, the third
  point is definitionally true when \(A = \obj\), and it follows from the
  first one when \({A = \arr[B]{u}{v}}\).
\end{proof}



\begin{lemma}\label{lemma:vartmsubunion}
    Given a substitution $\Gamma \vdash \sigma:\Delta$ and a term $\Delta\vdash t:A$,
    we have that
        \[\Var(t[\sigma]) = \bigcup_{y \in \Var(t)} \Var(y[\gamma])\]
\end{lemma}
\begin{proof}
        We prove this by induction on the structure of $t$.
        If $t = x$ is a variable, $\Var(t) = \{x\}$ and the result is trivially true.
        Assume now $t = \coh_{\Phi,A}[\gamma]$, then by structural induction:
\begin{align*}
        \Var(t[\sigma]) &= \bigcup_{x \in \Var(\Phi)} \Var(x[\gamma][\sigma]) \\
        &\stackrel{i.h.}{=} \bigcup_{x \in \Var(\Phi)} \bigcup_{y \in \Var(x[\gamma])} \Var(y[\sigma]) \\
        &= \bigcup_{y \in \Var(t)} \Var(y[\sigma]) \qedhere
\end{align*}
\end{proof}

\begin{lemma}\label{lem:weakening}
  Consider a valid context \((\Gamma,x:A)\vdash\) and a set
  \(X\in\U(\Gamma,x:A)\) such that \(\depth_{X}(\Gamma,x:A)\leq 1\). Denote
  \(X' = X\cap\Var (\Gamma)\), the following hold
  \begin{itemize}
  \item For \(\Gamma\vdash t:B\),
    \[
      t[\inj^{\pm}_{(\Gamma,x:A),X}] = t[\inj^{\pm}_{\Gamma,X'}].
    \]
  \item For \(\Gamma\vdash \sigma:\Delta\),
    \[
      \sigma\circ\inj^{\pm}_{(\Gamma,x:A),X} = \sigma\circ\inj^{\pm}_{\Gamma,X}.
    \]
  \item For
    \(\Gamma\vdash B\),
    \[
      B[\inj^{\pm}_{(\Gamma,x:A),X}] = B[\inj^{\pm}_{\Gamma,X'}].
    \]
  \item For \(\Gamma\vdash t:B\) such that
    \(0\leq \depth_{X}(t) \leq 1\),
    \[
      t\uparrow X = t \uparrow X'.
    \]
  \item for
    \(\Gamma\vdash \sigma:\Delta\) such that
    \(0\leq \depth_{X}(\sigma) \leq 1\),
    \[
      \sigma\uparrow X = \sigma \uparrow X'.
    \]
  \item For \(\Gamma\vdash t:B\) such
    that \(0\leq \depth_{X}(t) \leq 1\),
    \[
      B\uparrow^{t} X = B \uparrow^{t} X'.
    \]
  \end{itemize}
\end{lemma}
\begin{proof}
  The first two points are proven by mutual induction. The first one holds
  when \(t\) is a variable by definition. When \(t = \coh_{\Delta,B}[\sigma]\)
  is a coherence term, it follows from the inductive hypothesis on \(\sigma\).
  Similarly, the second point is definitionally true when
  \(\sigma = \sub{}\), and when \(\sigma = \sub{\sigma',y\mapsto t}\), it
  follows by the inductive hypothesis on \(\sigma'\) and \(t\). The third
  point is true trivially when \(B = \obj\) and follows from the first
  one when \(B\) is an arrow type.

  The fourth and fifth points are similarly proven by mutual induction.
  The fourth one is definitionally true for variables, and it holds when
  \(t = \coh_{\Delta,B}[\sigma]\) by the inductive hypothesis on \(\sigma\),
  using that
  \[\sigma^{-1} X = \sigma^{-1}X'.\]
  This last equality follows from \(t\) being a term in context \(\Gamma\), and
  hence \(x \not\in\Var(\sigma)\). The fifth point follows vacuously
  when \(\sigma = \sub{}\). When \(\sigma = \sub{\sigma',y\mapsto t}\), it
  follows by the inductive hypothesis on \(\sigma'\) and \(t\), and by the
  first point. Finally, the last point follows from the previous
  ones easily, treating the cases \(B = \obj\) and \(B = \arr[C]{u}{v}\)
  separately.
\end{proof}

\begin{proof}[Proof of Lemma~\ref{lemma:nat-variable}]
  We have that
  \[t[\inc^\pm_{(\Gamma,x:A),X\cup\{x\}}] = x^\pm. \]
  This shows the equality of the two types when \(A = \obj\). In the case that
  \(A = \arr[B]{u}{v}\) then we further have that the naturality of \(u\) and
  \(v\) with respect to \(X\) and \(X\cup\{x\}\) coincide by
  Lemma~\ref{lem:weakening}. Hence, the two types are equal in this case as
  well.
\end{proof}

\begin{lemma}\label{lemma:varsetpreimcomp}
        For substitutions $\Delta \vdash \gamma: \Gamma$ and $\Theta \vdash \sigma: \Delta$
        and every $X \subseteq \Var(\Theta)$, we have:
        \[{\gamma}^{-1}({\sigma}^{-1}X) = {(\gamma  \circ \sigma)}^{-1}X\]
\end{lemma}
\begin{proof}
        Fix some $x \in \Var(\Theta)$ and let us assume that \hbox{$x \in {(\gamma \circ \sigma)}^{-1}X$}.
        By unpacking definitions, this means we must have some \hbox{$z \in \Var(x[\sigma][\gamma])$} such that $z \in X$.
        By Lemma~\ref{lemma:vartmsubunion}, we have the equality
\begin{align*}
        \Var(x[\sigma][\gamma]) = \bigcup_{y \in \Var(x[\sigma])} \Var(y[\gamma])
.\end{align*}
Hence $x \in {(\gamma \circ \sigma)}^{-1}X$ iff there exists some $y \in \Var(x[\sigma])$ and $z \in \Var(y[\gamma])$ such that $z \in X$, which is exactly the same as $x \in {\sigma}^{-1}({\gamma}^{-1}X)$.
\end{proof}

\begin{proof}[Proof of Lemma~\ref{lem:pullback-stability}]
  The first point holds by definition when \(\sigma = \sub{}\) and
  \(\Delta = \emptycontext\). Suppose that \(\Delta = (\Delta', y:A)\)
  and \(\sigma = \sub{\sigma',y\mapsto t}\). In the case that
  \(\Var(t)\cap X \not= \emptyset\), we compute by the inductive hypothesis on
  \(\sigma'\) that:
  \[\begin{split}
    \inc^{\pm}_{\Delta,\sigma^{-1}X}&\circ(\sigma \uparrow X) \\
      &= \sub{\inc^{\pm}_{\Delta',(\sigma')^{-1}X}
        \circ (\sigma' \uparrow X) , y \mapsto y^\pm [\sigma \uparrow X]} \\
      &= \sub{\sigma' \circ \inc_{\Gamma,X}^\pm ,
        y \mapsto t[\inc^\pm_{\Gamma,X}]} \\
      &= \sigma \circ \inc_{\Gamma,X}^\pm
  \end{split}\]
  The case \(\Var(t)\cap X = \emptyset\) follows similarly using the first part
  of Lemma~\ref{lem:empty-intersection}.

  For the remaining points, we observe that both sides of the equations are
  well-defined by Lemma~\ref{lem:preimage-bound}.
  The second and third points are proven by mutual induction. For the case
  where \(t = y\) is a variable, we may assume that \(\Delta = (\Delta',x:B)\)
  and \(\sigma = \sub{\sigma',x\mapsto t}\). If \(x \not=y\), then the equality
  holds by induction on the length of \(\Delta\). if \(x = y\), then the
  equality holds from the definition of \(\sigma\uparrow X\). For
  the case where \(t = \coh_{\Theta,B}[\tau]\) and \(\Delta,\sigma\) are
  arbitrary, we have by the inductive hypothesis that:
  \[\begin{split}
    t[\sigma]\uparrow X
    &= (\coh_{\Theta,B}\uparrow (\tau\circ\sigma)^{-1}X)
      [(\tau\circ \sigma)\uparrow X] \\
    &= (\coh_{\Theta,B}\uparrow \tau^{-1}(\sigma^{-1}X))
      [\tau\uparrow \sigma^{-1}X][\sigma\uparrow X] \\
    &= (t\uparrow \sigma^{-1}X)[\sigma\uparrow X]
  \end{split}\]
  The third point holds for \(\tau = \sub{}\) definitionally,
  so suppose that \(\Theta = (\Theta',x : A)\) and
  \(\tau = \sub{\tau',x\mapsto t}\). If \(x \in (\tau\circ\sigma)^{-1}X\), then
  by Lemma~\ref{lemma:varsetpreimcomp},
  \[\begin{split}
    (\tau&\uparrow \sigma^{-1}X)\circ (\sigma\uparrow X) \\
    &= \sub{\tau' \uparrow \sigma^{-1}X, x^\pm \mapsto t[\inc^\pm_{\Delta,\sigma^{-1}X}], \fun x \mapsto t\uparrow \sigma^{-1}X} \\
    &\qquad\circ (\sigma\uparrow X) \\
    &= \sub{(\tau'\circ \sigma) \uparrow X, x^\pm \mapsto t[\sigma][\inc^\pm_{\Gamma,X}], \fun x \mapsto t[\sigma] \uparrow X} \\
    &= (\tau\circ \sigma)\uparrow X
  \end{split}\]
  The case where \(x \notin (\tau\circ\sigma)^{-1}X\) follows similarly using
  Lemma~\ref{lem:empty-intersection}:
  \[\begin{split}
    (\tau&\uparrow \sigma^{-1}X)\circ (\sigma\uparrow X) \\
    &= \sub{\tau' \uparrow \sigma^{-1}X, x \mapsto t} \circ (\sigma\uparrow X)\\
    &= \sub{\tau' \uparrow \sigma^{-1}X, x \mapsto t[\inl_{\Delta,\sigma^{-1}X}]} \circ (\sigma\uparrow X)\\
    &= \sub{(\tau'\circ \sigma) \uparrow X, x \mapsto t[\sigma][\inc^\pm_{\Gamma,X}]} \\
    &= \sub{(\tau'\circ \sigma) \uparrow X, x \mapsto t[\sigma]} \\
    &= (\tau\circ \sigma)\uparrow X
  \end{split}\]
  Finally, the last statement follows easily from the previous ones by treating
  the cases \(A = \obj\) and \(A = \arr[B]{u}{v}\) separately.
\end{proof}

\begin{lemma}\label{lem:preimage-bound}
  For every \(\Gamma\vdash\sigma:\Delta\) and \(X\in \U(\Gamma)\):
  \begin{itemize}
      \item \(\sigma^{-1}X \subseteq \Var\Delta\) is up-closed.
      \item \(\depth_{\sigma^{-1}X}t \le \depth_X(t[\sigma])\)
          for every \(\Delta\vdash t:A\)
      \item \(\depth_{\sigma^{-1}X}\tau \le \depth_X(\tau\sigma)\) for every
          \(\Delta\vdash\tau:\Theta\)
      \item \(\depth_{\sigma^{-1}X} \Delta \le \depth_X \sigma\)
  \end{itemize}
\end{lemma}
\begin{proof}
  For the first part, let \(\Delta \vdash y : A\) and suppose that \(
  x\in \Var(A)\cap \sigma^{-1}X\) to show that \(y\in \sigma^{-1}X\). By
  definition and Lemma~\ref{lemma:vartmsubunion},
  \[\emptyset\neq \Var(x[\sigma])\cap X \subseteq \Var(A[\sigma])\cap X.\]
  If \(y[\sigma] = z\) is a variable, then it must have type \(A[\sigma]\), so
  by up-closure of \(X\), we must have that \(z\in X\) or equivalently
  \(y\in\sigma^{-1}X\). If \(y[\sigma] = \coh_{\Theta,B}[\tau]\) is a coherence,
  then \(B[\tau] = A[\sigma]\) and
  \[
    \Var(A[\sigma]) = \Var(B[\tau]) \subseteq \Var(\tau) = \Var(y[\sigma])
  \]
  so again \(y\in \sigma^{-1}X\).

  For the second part, let \(\Delta \vdash t : A\) a term. Then for every variable
  \(z\in \Var(t)\cap \sigma^{-1}X\), there exists some other variable
  \({w\in \Var(z[\sigma])\cap X\subseteq \Var(t[\sigma])\cap X}\)
  such that
  \[\begin{split}
      \dim t - \dim z
          &= \dim t[\sigma] - \dim z[\sigma] \\
          &\le \dim t[\sigma] - \dim w \\
          &\le \depth_X t[\sigma]
  \end{split}\]
  Taking the maximum over all \(z\), we get the second part. The third part
  follows immediately from the second one, while the last one is obtained by
  the third one by setting \(\tau = \id_\Delta\).
\end{proof}

This concludes the first family of lemmas needed for Theorem~\ref{thm:correctness}.
We now embark on the proof of correctness of the naturality of coherences. The
depth \(d = 0\) case is covered in Benjamin's thesis~\cite[Section~3.4]{benjamin_type_2020},
so we need to show the case \(d = 1\). This is done in four cases: the case of
naturality with respect to variables of depth 0 in the context, that of linear compositions,
that of reduced compositions, and that of arbitrary compositions. The first
case is covered by Lemma \ref{lem:corr-nat-coh-lmax}, which relies on the following results.

\begin{lemma}\label{lem:lmaxnonmax-bdry}
        For any pasting context $\Gamma \vdashps$ and variable $x \in \Var(\Gamma)$ of depth~\(0\) with $\dim(x) < \dim(\Gamma)$, we have ${x \in \Var(\partial^\pm \Gamma)}$.
\end{lemma}
\begin{proof}
        Since $x$ is of depth\=/0 in \(\Gamma\), it is a \emph{source and target boundary} position in the sense of Dean et al.~\cite[Definition 2.9]{dean_computads_2024},
  meaning that it is not the target or source of any variable respectively.
        Hence by their characterisation of $\Var(\delta^\pm_\Gamma)$~\cite[Proposition 2.10]{dean_computads_2024},
        since we have $\dim(x) < \dim(\Gamma)$, we deduce that ${x \in \Var(\delta^\pm_\Gamma) = \Var(\partial^\pm \Gamma)}$.
\end{proof}

\begin{lemma}\label{lem:psctx-func}
        Let $\Gamma \vdashps$ be a pasting context and $X \in \U(\Gamma)$ a set of depth\=/0 variables in $\Gamma$.
        Then $\Gamma \uparrow X$ is also a pasting context.
        Moreover, if $X$ is a non-empty set of maximal-dimensional variables of $\Gamma$, then $\partial(\Gamma \uparrow X) = \Gamma$ and the following substitutions coincide
\begin{align*}
        \Gamma \uparrow X \vdash \inc^\pm_{\Gamma,X} = \delta^\pm_{\Gamma \uparrow X} : \Gamma
\end{align*}
\end{lemma}
\begin{proof}
        This is proven by Benjamin~\cite[Lemmas~87,~88]{benjamin_type_2020}.
\end{proof}

\begin{lemma}\label{lem:func-of-func}
  Let \(\Gamma\) be an arbitrary context and \(X,Y \in \U(\Gamma)\) disjoint sets of
  variables of depth \(0\) in \(\Gamma\). Then for every \(\alpha,\beta \in \{+,-\}\),
  \begin{itemize}
    \item \((\Gamma \uparrow X) \uparrow Y = (\Gamma \uparrow Y) \uparrow X = \Gamma \uparrow (X\cup Y)\)
    \item \(\inc_{\Delta,X}^\alpha\circ \inc_{\Delta\uparrow X,Y}^\beta = \inc_{\Delta,Y}^\beta \circ \inc_{\Delta\uparrow Y,X}^\alpha\)
    \item \(\inc_{\Delta,X}^\alpha\circ\inc_{\Delta\uparrow X,Y}^\alpha = \inc_{\Delta,X\cup Y}^\alpha\)
  \end{itemize}
\end{lemma}
\begin{proof}
  The equality between the contexts has been proven in Benjamin's
  thesis~\cite[Lemma~83]{benjamin_invertible_2024}. The first equality of
  substitutions can be checked easily: both sides preserve variables that are
  not in \(X\cup Y\), they send variables of \(y\in Y\) to \(y^\beta\), and they
  send variables \(x\in X\) to \(x^\alpha\). The second equality can be checked
  similarly.
\end{proof}

\begin{lemma}\label{lem:var-func}
  For any term \(\Gamma \vdash t : A\) and any \(X\in \U(\Gamma)\) such that
  \(\depth_X(t) = \depth_X(\Gamma) = 0\),
  \[\Var(t\uparrow X) \supseteq \{ \fun x : x\in \Var(t)\cap X\}\]
  Similarly, for any substitution \(\Gamma \vdash \sigma : \Delta\) such that
  \(\depth_X(\sigma) = \depth_X(\Gamma) = 0\),
  \[\Var(\sigma\uparrow X) \supseteq \{ \fun x : x\in \Var(\sigma)\cap X\}\]
\end{lemma}
\begin{proof}
  We prove the two statements by mutual induction. The case where \(t = x \in X\)
  follows trivially. The case \(t = \coh_{\Delta,A}[\sigma]\) follows by the
  inductive hypothesis for \(\sigma\). Similarly the case where \(\sigma = \sub{}\)
  is vacuously true. The case where \(\sigma = \sub{\sigma', y\mapsto t}\) follows
  by the inductive hypothesis for \(\sigma'\) and for \(t\).
\end{proof}

\begin{lemma}\label{lem:fullterm-fullfunct}
  For any term \(\Gamma \vdash t : A\) and \({X\in \U(\Gamma)}\) such that
  \(\depth_X(t) = \depth_X(\Gamma) = 0\) and ${\Var(t) \cup \Var(A) = \Var(\Gamma)}$, we have:
  \[\Var(t\uparrow X) \cup \Var(A\uparrow^t X) = \Var(\Gamma \uparrow X)\]
\end{lemma}
\begin{proof}
        It is easy to show by induction that:
\begin{align*}
        \Var(\Gamma \uparrow X) = (\Var(\Gamma) \setminus X) \cup \{x^-, x^+, \fun{x} : x \in X\}
\end{align*}
        Since $\depth_X t = 0$, we know that $X \cap \Var(A) = \emptyset$, hence we must have $X \subseteq \Var(t)$.
        By Lemma \ref{lem:var-func}, for all $x \in X$, we have $\fun{x} \in \Var(t\uparrow X)$.
        Therefore it suffices to show
\begin{align*}
        \Var(A \uparrow^t X) \supseteq \Var(A) \cup (\Var(t) \setminus X) \cup \{x^\pm : x \in X\}
\end{align*}
        But, by construction, we have:
\begin{align*}
        \Var(A \uparrow^t X) = \Var(A) \cup \Var(t[\inl_X]) \cup \Var(t[\inr_X])
\end{align*}
        Hence the result follows by applying Lemma \ref{lemma:vartmsubunion} to $t[\inl_X]$ and $t[\inr_X]$.
\end{proof}

\begin{lemma}\label{lem:corr-nat-coh-lmax}
  Let \(\Gamma\vdashps\) be a pasting context, \(A\) be a full type and
  \(t = \coh_{\Gamma,A}[\id]\). For any set \(X\in \U(\Gamma)\) such that
  \(\depth_X {\Gamma} = 0\) and \(\depth_X t = 1\), the context
  \(\Gamma\uparrow X\) is a pasting context and \(A\uparrow^t X\) is a full
  type over it. In particular, \( \coh_{\Gamma,A}\uparrow X\) is a well-defined
  term of type \(A\uparrow^t X\).
\end{lemma}
\begin{proof}
        Since $\depth_X \Gamma = 0$, $\Gamma \uparrow X$ is a pasting context~\cite[Lemma 86]{benjamin_type_2020}.
        Hence it suffices to prove that $A \uparrow^t X$ is full.

        Let $n = \dim \Gamma > 0$.
        By the assumptions on depth, \({A = \arr[B]{u}{v}}\) and either \(\dim A = n\) or \(\dim A = n-1\), depending on which side condition it satisfies.    Let us assume $\dim A = n$. Then:
        \[
                \Var(\Gamma) = \Var(u) \cup \Var(A) = \Var(v) \cup \Var(A)
        \]
        Hence by Lemma~\ref{lem:fullterm-fullfunct} we have:
        \begin{align*}
                \Var(\Gamma \uparrow X) =&\Var(u \uparrow X) \cup \Var(B \uparrow^u X)  \\
                = &\Var(v \uparrow X) \cup \Var(B \uparrow^v X)
        \end{align*}
        This shows $A \uparrow^t X$ is full.
        Assume now $\dim A = n$ and thus:
        \begin{align*}
                \Var(\partial^- \Gamma) &= \Var(u) \cup \Var(A) \\
                \Var(\partial^+ \Gamma) &= \Var(v) \cup \Var(A)
        \end{align*}
        Let us write $X = X^m \sqcup X^l$ where $X^m$ is the set of variables of dimension $n$ in $X$, and recall that since $\depth_X t = 1$ we must have $X^l \neq \emptyset$.
        Moreover, by Lemma~\ref{lem:lmaxnonmax-bdry} $X^l \subseteq \Var(\partial^\pm \Gamma)$, and thus $X^l \subseteq \Var(u) \cap \Var(v)$.
        This shows that:
        \begin{align*}
                A \uparrow^t X = t[\inl_{\Gamma,X}] * (v \uparrow X) \to (u \uparrow X) * t[\inr_{\Gamma,X}]
        \end{align*}
        We now observe that by dimension constraints we also have
        $X^m \cap \Var(u) = \emptyset$, hence $u \uparrow X = u \uparrow X^l$ by Lemma~\ref{lem:weakening}.
        Moreover, by Lemma~\ref{lem:func-of-func}, $\inc^\pm_{\Gamma,X} = \inc^\pm_{\Gamma,X^l} \inc^\pm_{\Gamma \uparrow X^l,X^m}$.
        We note that by Lemma~\ref{lem:fullterm-fullfunct}:
        \begin{align*}
                &\Var(v \uparrow X) \cup \Var(B \uparrow^v X^l) \\
                =& \Var(\partial^+ \Gamma \uparrow X^l) \supseteq \{x^-, x^+, \fun{x} : x \in X^l\}
        \end{align*}
        Moreover, $\Var(t[\inl_{\Gamma,X^l}]) \subseteq \Var(t) \setminus X^l$ by Lemma~\ref{lemma:vartmsubunion}, hence by putting these results together and applying a similar argument for $u \uparrow X$, we have:
        \begin{align*}
                \Var(\Gamma \uparrow X^l) &= \Var(t[\inl_{\Gamma,X^l}] * (v \uparrow X^l)) \\
                &= \Var((u \uparrow X^l) * t[\inr_{\Gamma,X^l}])
        \end{align*}
        If $X^m = \emptyset$, then $\Gamma \uparrow X = \Gamma \uparrow X^l$, hence we have shown that $A \uparrow^t X$ is full.
        Suppose then that $X^m \neq \emptyset$. In this case, we use that by Lemma~\ref{lem:func-of-func}:
        \[
                \Gamma \uparrow X = (\Gamma \uparrow X^l) \uparrow X^m
        \]
  Since the variables in $X^m$ are of maximal dimension in $\Gamma \uparrow X^l$, by Lemma \ref{lem:psctx-func}, we deduce that:
  \[
    \partial(\Gamma \uparrow X) = \Gamma \uparrow X^l
  \]
  Moreover, the source and target inclusions are given by:
  \[
    \inc^\pm_{\Gamma \uparrow X^l,X^m} = \delta^\pm_{\Gamma \uparrow X}
  \]
  From that, we may then conclude that:
        \begin{align*}
                \Var(\partial^\pm (\Gamma \uparrow X)) = \Var(\inc^\pm_{\Gamma \uparrow X^l,X^m})
        \end{align*}
        Finally by Lemma~\ref{lemma:vartmsubunion} we get that:
        \begin{align*}
                (t[\inl_{\Gamma,X^l}] * (v \uparrow X^l))[\inl_{\Gamma \uparrow X^l,X^m}] &= (t[\inl_{\Gamma,X}] * (v \uparrow X)) \\
                ((u \uparrow X^l) * t[\inr_{\Gamma,X^l}])[\inr_{\Gamma \uparrow X^l,X^m}] &= ((u \uparrow X) * t[\inr_{\Gamma,X}])
        \end{align*}
  This shows that
  \begin{align*}
    \Var(\partial^- (\Gamma \uparrow X)) &= \Var(t[\inl_{\Gamma,X}] * (v \uparrow X)) \\
    \Var(\partial^+ (\Gamma \uparrow X)) &= \Var((u \uparrow X) * t[\inr_{\Gamma,X}])
  \end{align*}
  so the type \(A\uparrow^t X\) is again full.
\end{proof}

The second case is that of linear compositions. For that, we need to show first that
the associator and the whiskering phases defined in Appendix~\ref{app:nat-coh}
are well-defined and composable, as well as show that suspension commutes with
naturality of depth \(0\). Correctness of linear composites will be established
in Lemma~\ref{lemma:lincompnatsusp}.

\begin{lemma}\label{lem:assoc-lincomp-correct}
  For $k > 0$ and $0 \le j \le k$, the terms \(\assoc_{k,j}\) of
  Definition~\ref{def:assoc-lincomp} satisfy the judgements:
  \[\Psi^0_{k+1} \vdash \assoc_{k,j} : \bcomp_{k,j+1} \to \bcomp_{k,j}\]
\end{lemma}
\begin{proof}
  By application of \ruleref{rule:coh}, it suffices to prove that we have
  $\Var(\bcomp_{k,j}) = \Var(\Psi^0_{k+1})$ and the following judgement is
  derivable:
  \[
    \Psi^0_{k+1} \vdash \bcomp_{k,j} : x_0 \to_* x_{k+1}
  \]
  This follows from the definition of the context \(\Psi^{0}_{k+1}\), in which
  the variables \(f_{0},\ldots, f_{k}\) all satisfy:
  \[
    \Psi^{0}_{k+1} \vdash f_{i} : \arr{x_{i}}{x_{i+1}} \qedhere
  \]
\end{proof}

\begin{lemma}\label{lem:scan-subst-correct}
  For every $k\in \N$, $0 \le i,j \le k+1$, $X \in \U(\Psi^0_k)$ such that $x_j \in X$,
  the substitution \(\psi^{X}_{k,j,i}\) of Definition~\ref{def:scan-subst} satisfies the following judgement:
  \begin{align*}
        \Psi^0_k &\uparrow X \vdash \psi_{k,j,i}^X: \Psi^0_{i}
  \end{align*}
\end{lemma}
\begin{proof}
        We proceed by induction on $i$.
        For $i=0$, this is immediate from the definition of $\psi^X_{k,j,0}$.
        For the inductive case, we can assume $\Psi^0_k \uparrow X \vdash \psi_{k,j,i}^X: \Psi^0_{i}$
        and prove the result for $i+1$.
        If $i < j$, this follows by correctness of $\inl$, and
        if $i > j$, this follows by correctness of $\inr$.
        Finally, if $i=j$, since $x_j \in X$, the term $\fun{x_j}$ exists and has the correct type.
\end{proof}

\begin{lemma}\label{lem:whisk-lincomp-correct}
  For \(0\le j < k\) and $X \in \U(\Psi^0_k)$ such that \(f_j\in X\), the
  whiskering phase \(w^{X}_{k,j}\) of Definition~\ref{def:whisk-lincomp} is
  well-defined in context \(\Psi^{0}_{k}\uparrow X\). Denoting
  \(\inc^{\pm} = \inc^{\pm}_{\Psi^{0}_{k},X}\), the source of this term is given
  as follows:
  \begin{align*}
    f_{0}[\inl] \!\s_{0}\! \cdots\!\s_{0}\! f_{j}^{\sm} \!\s_{0}\! \cdots
    \s_{0} \! f_{k-1}[\inr] & \text{ if } x_{j+1} \notin X \\
    f_{0}[\inl] \!\s_{0}\! \cdots\!\s_{0}\! (f_{j}^{\sm}\!\s_{0}\!\fun{x_{j+1}}) \!\s_{0}\! \cdots
    \!\s_{0}\! f_{k-1}[\inr] & \text{ if } x_{j+1} \in X \\
  \intertext{Similarly, the target of this term is given as follows:}
    f_{0}[\inl] \!\s_{0}\! \cdots \!\s_{0}\! f_{j}^{\sp} \!\s_{0}\! \cdots
    \!\s_{0}\! f_{k-1}[\inr] & \text{ if } x_{j} \notin X \\
    f_{0}[\inl] \!\s_{0}\! \cdots\!\s_{0}\! (\fun{x_{j}}\!\s_{0}\! f_{j}^{\sp}) \!\s_{0}\! \cdots
    \!\s_{0}\! f_{k-1}[\inr] & \text{ if } x_{j} \in X
  \end{align*}
\end{lemma}
\begin{proof}
  This can be checked explicitly by the formulae given for \(w^{X}_{k,j}\) and
  the definition of naturality of contexts.
\end{proof}

\begin{lemma}
\label{lemma:lincompnat}
Let $k > 1$ and $X \in \U(\Psi^0_k)$ be of depth 1 in $\Psi^0_k$. Then the term
\(\comp^{0}_{k}\uparrow X\) of Definition~\ref{def:nat-lincomp} satisfies:
\[
  \Psi^0_k \uparrow X \vdash \comp^{0}_{k} \uparrow X : (x_0 \to x_k) \uparrow^{\comp^{0}_{k}} X
\]
\end{lemma}
\begin{proof}
  Denote \(\inc^{\pm} = \inc^{\pm}_{\Psi^{0}_{k},X}\), and first remark that if
  \(x_{0}\in X\), then the last phase is given by
  \(h^{X}_{k,x_{0}} = \assoc_{k,0}[\psi^{X}_{k,0}]\), which is derivable by
  Lemma~\ref{lem:scan-subst-correct}, and whose target computes by
  Definition~\ref{def:scan-subst} and Lemmas~\ref{lem:assoc-lincomp-correct} to:
  \[
    \fun{x_{0}} \s_{0}(f_{0}[\inr] \s_{0} \ldots \s_{0} f_{k-1}[\inr]) =
    \fun{x_{0}} \s_{0} (\comp^{0}_{k}[\inr])
  \]
  This is exactly the target of the type
  \((x_0 \to x_k) \uparrow^{\comp^{0}_{k}} X\) when \(x_{0}\in X\). If
  \(x_{0}\notin X\), then by up-closure, the last phase is necessarily of
  the form \(h^{X}_{k,f_{i}} = w^{X}_{k,j}\) for a variable \(f_{j}\) such that
  \(x_{j+1}\notin X\). The target of this term is then given by
  Lemma~\ref{lem:whisk-lincomp-correct} by:
  \[
    f_{0}[\inl] \s_{0} \ldots\s_{0} f^{+}_{j}\s_{0}\ldots \s_{0} f_{k-1}[\inr]
  \]
  Since \(j\) is the last phase, none of the \(f_{i}\) for \(i < j\) can be in
  \(X\), thus, for \(i < j\), we have \(f_{i}[\inl] = f_{i} = f_{i}[\inr]\).
  This shows that the target of \(w^{X}_{k,j}\) is \(\comp^{0}_{k}[\inr]\), which
  is again exactly the type of \((x_0 \to x_k) \uparrow^{\comp^{0}_{k}} X\) when
  \(x_{0}\notin X\).

  Similarly, if \(x_{k}\in X\) then the first phase is
  \(h^{X}_{k,x_{k}} = \assoc_{k,k}[\psi^{X}_{k,k}]\), whose source computes to:
  \[
    (f_{0}[\inl] \s_{0} \ldots \s_{0} f_{k-1}[\inl])\s_{0}\fun{x_{k}} =
    (\comp^{0}_{k}[\inr]) \s_{0} \fun{x_{k}}
  \]
  Furthermore, if \(x_{k}\notin X\), then the first phase is of the form
  \(h^{X}_{k,f_{j}} = w_{k,0}^{X}\) for some \(j\) such that \(x_{j} \notin X\),
  and in that case, the source computes to:
  \[
    f_{0}[\inl] \s_{0} \ldots \s_{0} f_j^{-} \s_{0}\ldots \s_{0} f_{k-1}[\inl] =
    \comp^{0}_{k}[\inl]
  \]
  In both cases, this is exactly the source of the type
  \({(x_0 \to x_k) \uparrow^{\comp^{0}_{k}} X}\). This shows that if the term
  is well defined, then it has the correct type, and it suffices to show that it
  is well-defined, that is that the phases compose.

  We observe that it suffices to prove that any two consecutive phases
  \(h^{X}_{k,y}\) and \(h^{X}_{k,z}\) compose. Because \(X\) is up-closed,
  this reduces to the following case split:
  \begin{enumerate}
  \item \(z = f_{j+1} \in X\) and \(y = f_{j}\)
  \item \(z = x_{j+1} \in X\) and \(y = f_{j} \in X\)
  \item \(z = f_{j}\in X\) and \(y = x_{j}\in X\).
  \end{enumerate}
  In the first case, \(x_{j+1}\notin X\), hence the phases are respectively of
  the form \(w^{X}_{k,j+1}\) and \(w^{X}_{k,j}\). By
  Lemma~\ref{lem:whisk-lincomp-correct}, the target of the former and the source
  of the latter both compute to:
  \[
    f_{0}[\inl] \s_{0} \ldots \s_{0} f_{j}^{-} \s_{0} f_{j+1}^{+} \s_{0} \ldots
    \s_{0} f[\inr]
  \]
  In the second case, the phases are respectively of the form
  \(\assoc_{k,j+1}[\psi^{X}_{k,j+1}]\) and \(w^{X}_{k,j}\). By
  Lemma~\ref{lem:scan-subst-correct}, the latter is derivable, and by
  Definition~\ref{def:scan-subst} and Lemmas~\ref{lem:assoc-lincomp-correct}
  and~\ref{lem:whisk-lincomp-correct}, the target of the former and the source
  of the latter then both compute to:
  \[
    f_{0}[\inl] \s_{0} \ldots \s_{0} (f_{j}^{-} \s_{0} \fun{x_{j+1}}) \s_{0}
    \ldots \s_{0} f[\inr]
  \]
  Similarly, in the third case, the phases are respectively of the form
  \(w^{X}_{k,j}\) and \(\assoc_{k,j}[\psi^{X}_{k,j}]\). By
  Lemma~\ref{lem:scan-subst-correct}, the latter is derivable, and by
  Definition~\ref{def:scan-subst} and Lemmas~\ref{lem:assoc-lincomp-correct},
  and, the target of the former and the source of the latter then both compute
  to:
  \[
    f_{0}[\inl] \s_{0} \ldots \s_{0} ( \fun{x_{j}}\s_{0}f_{j}^{-}) \s_{0} \ldots
    \s_{0} f[\inr] \qedhere
  \]
\end{proof}

\begin{lemma}\label{lemma:suspincludefunc}
  For every context \(\Gamma\vdash\) and every \(X\in\U(\Gamma)\) such that \({\depth_{X}(\Gamma)=0}\):
    \begin{align*}
      \Sigma(\Gamma \uparrow X) &= (\Sigma \Gamma) \uparrow X \\
      \Sigma \inj^\pm_{\Gamma, X} &= \inj^{\pm}_{\Sigma \Gamma, X}
    \end{align*}
\end{lemma}
\begin{proof}
  We proceed by structural induction on $\Gamma$. For the empty context
  $\emptycontext$, we have:
  \[
    \Sigma(\emptycontext \uparrow \emptyset) = (\Sigma \emptycontext)\uparrow \emptyset
    = (N : \obj, S : \obj)
  \]
  For the context $(\Gamma,x: A)$, denote \(X' = X\setminus\{x\}\). Then, if
  $x \in X$ we have:
  \begin{align*}
    \Sigma ((\Gamma,x:A)\uparrow X)
          & =\Sigma (\Gamma \uparrow X', x^\pm : A, \fun{x} : x^- \to_{A} x^+) \\
          & = (\Sigma(\Gamma\uparrow X'), x^\pm : \Sigma{A}, \fun{x} : x^-
            \to_{\Sigma A} x^+) \\
          & = ((\Sigma\Gamma)\uparrow X', x^\pm : \Sigma{A}, \fun{x} : x^-
            \to_{\Sigma A} x^+) \\
          & = \Sigma(\Gamma, x : A) \uparrow \Sigma X
  \end{align*}
  On the other hand, if $x \notin X$, we have:
  \begin{align*}
    \Sigma((\Gamma, x : A) \uparrow X)
    &= \Sigma(\Gamma \uparrow X', x : A) \\
    &= (\Sigma(\Gamma \uparrow X'), x : \Sigma A)\\
    &= ((\Sigma \Gamma) \uparrow X', x : \Sigma A)\\
    &= \Sigma(\Gamma, x : A) \uparrow \Sigma X
  \end{align*}
  For the second statement, consider a variable \(x\) of \(\Gamma\).
  If \(x\notin X\), then we have:
  \[
    x[\inc^{\pm}_{\Sigma\Gamma,X}] = x = \Sigma(x[\inc^{\pm}_{\Gamma,X}])
  \]
  If \(x\in X\), then:
   \[
    x[\inc^{\pm}_{\Sigma\Gamma,X}] = x^{\pm} = x[\Sigma\inc^{\pm}_{\Gamma,X}]
  \]
  Finally, since \(N,S \notin X\), we have:
  \begin{align*}
    N[\inc^{\pm}_{\Sigma\Gamma,X}] &= N = N[\Sigma\inc^{\pm}_{\Gamma,X}]\\
    S[\inc^{\pm}_{\Sigma\Gamma,X}] &= S = S[\Sigma\inc^{\pm}_{\Gamma,X}]
  \end{align*}
  The two substitutions thus coincide on all variables and therefore are equal.
\end{proof}

\begin{lemma}\label{lemma:funcsusp}
  For every context \(\Gamma\) and every \(X\in\U(\Gamma)\) such that \(\depth_{X}(\Gamma)=0\),
  the following hold:
  \begin{itemize}
  \item For any term \(\Gamma\vdash t:A\) such that \(\depth_{X}(t) = 0\), we
    have:
    \[
      \Sigma(t \uparrow X) = (\Sigma t) \uparrow X
    \]
  \item For any substitution $\Gamma \vdash \sigma : \Delta$ such that
    \(\depth_{X}(\sigma) = 0\), we have:
    \[
      \Sigma(\sigma \uparrow X) = (\Sigma \sigma) \uparrow X
    \]
  \item For any term \(\Gamma\vdash t:A\) such that \(\depth_{X}(t) = 1\), we
    have:
    \[
      \Sigma(A \uparrow^{t} X) = (\Sigma A) \uparrow^{\Sigma t} X
    \]
  \end{itemize}
\end{lemma}
\begin{proof}
  We prove the first two statements together by mutual induction. If $t = x$ is
  a variable in \(X\), then:
  \[
    \Sigma(x \uparrow X) = \fun{x} = (\Sigma x) \uparrow \Sigma X
  \]
  If \(t = \coh_{\Delta,B}[\sigma]\) and $\sigma^{-1} X = \emptycontext$,
  then:
 \[
   \Sigma(t \uparrow X) = \Sigma t = (\Sigma t) \uparrow \Sigma X
 \]
 If $\sigma^{-1}X \neq \emptycontext$, since \(N,S\notin X\), we have that
 \(\sigma^{-1} X = (\Sigma\sigma)^{-1} X\). Denoting \(Y = \sigma^{-1} X\) and
 \(u = \coh_{\Delta,B}[\id]\), we have by Lemma~\ref{lemma:suspincludefunc}:
 \begin{align*}
   \Sigma(t\uparrow X)
   &= \coh_{\Sigma (\Delta \uparrow Y), (\Sigma
     u)[\Sigma\inj_{\Delta,Y}^-] \to (\Sigma u)[\Sigma\inj_{\Delta,Y}^+]}
     [\Sigma(\sigma \uparrow X)] \\
   &= (\Sigma t) \uparrow \Sigma X
 \end{align*}
 For the second statement, for the empty substitution $\sub{}$, we have, since
 \({N,S\notin X}\):
 \[
   \Sigma(\sub{}\uparrow X) = \sub{N\mapsto N, S\mapsto S} = (\Sigma \sub{})
   \uparrow X.
 \]
 For the substitution $\sub{\sigma, x \mapsto t}$, if $x \notin X$, we have:
 \begin{align*}
   \Sigma(\sub{\sigma, x\mapsto t} \uparrow X)
   &=\sub{\Sigma(\sigma\uparrow X), x \mapsto \Sigma t}\\
   &= \sub{(\Sigma \sigma)\uparrow  X, x \mapsto \Sigma t}\\
   &= (\Sigma \sub{\sigma,x\mapsto t}) \uparrow  X
 \end{align*}
 On the other hand, if $x \in X$, then bythe inductive hypothesis, by
 Lemma~\ref{lemma:suspincludefunc}, and by the following equation~\cite[Lemma~71]{benjamin_type_2020}:
 \[\Sigma(t[\inj_{\Gamma,X}^{\pm}]) = (\Sigma
 t)[\Sigma\inj_{\Gamma,X}^{\pm}]\]
 We may therefore compute that:
 \begin{align*}
   \Sigma&(\sub{\sigma, x\mapsto t} \uparrow X)\\
   &= \sub{\Sigma(\sigma \uparrow X), x^\pm \mapsto
     \Sigma(t[\inj_{\Gamma,X}^{\pm}]), \fun{x}\mapsto \Sigma(t \uparrow X)}\\
   & = \sub{(\Sigma \sigma)\uparrow X, x^\pm \mapsto \Sigma(t[\inj_{\Gamma,X}^\pm]),
     \fun x \mapsto (\Sigma t)\uparrow X} \\
   &=\sub{(\Sigma \sigma)\uparrow \Sigma X, x^\pm \mapsto (\Sigma t)
     [\inj_{\Sigma \Gamma,X}^\pm], \fun{x} \mapsto (\Sigma t)\uparrow
     X}\\
   &= (\Sigma \sub{\sigma,x\mapsto t}) \uparrow X
 \end{align*}

 Finally, for the last statement, write \(A = \arr[]{u}{v}\) and \(n = \dim A\).
 If
 \(\Var(v)\cap X = \emptyset\) the the source of \(\Sigma (A\uparrow^{t} X)\) is
 given by \(\Sigma (t[\inl_{\Gamma,X}])\). On the other hand, the source of
 \((\Sigma A)\uparrow^{\Sigma t} X\) is \((\Sigma t)[\inl_{\Sigma\Gamma,X}]\).
 Again by Lemma~\ref{lemma:suspincludefunc} and the same equality as above, we
 may deduce that the two sources agree. If
 \(\Var{(v)}\cap X = \emptyset\), then the source of \(\Sigma (A\uparrow^{t} X)\)
 is \(\Sigma ((t[\inl_{\Gamma,X}]) \s_{n} (v\uparrow X))\), while the source of
 \((\Sigma A)\uparrow^{\Sigma t} X\) is
 \((\Sigma t)[\inl_{\Sigma\Gamma,X}] \s_{n+1} ((\Sigma v)\uparrow X)\). By the
 first part of the lemma and the same reasoning as in the previous case, we see
 that the two sources agree. A similar argument shows that the target are also
 equal, proving that the two types coincide.
\end{proof}

\begin{lemma}
  \label{lemma:lincompnatsusp}
  Let $n \in \N$, $k > 1$ and $X \in \U(\Psi^n_k)$ of depth 1 in $\Psi^n_k$.
  Then the following judgement holds:
  \[
    \Psi^n_k \uparrow X \vdash \comp^{n}_{k} \uparrow X : (x_0 \to x_k)
    \uparrow^{\comp^{n}_{k}} X
  \]
\end{lemma}
\begin{proof}
  We recall that the term \(\comp^{n+1}_{k} \uparrow X\) is defined as the
  suspension \(\Sigma (\comp^{n}_{k} \uparrow X)\). This result is thus
  immediate by induction, the base case being given by
  Lemma~\ref{lemma:lincompnat}, and the inductive case by
  the equation
  \[
    \comp^{n+1}_{k} = \Sigma \comp^{n}_{k}
  \]
  together with Lemma~\ref{lemma:funcsusp}.
\end{proof}

Now, we proceed to show correctness for the case of reduced compositions, culminating in Lemma \ref{lem:red-comp-correct}.
To this end, in what follows, let \(\Gamma \vdashps\) be a reduced pasting context of dimension
\(n > 0\) and let \(A = \arr[B]{u}{v}\) be a type of dimension \(n-1\). Suppose
also that a set \(X\in \U(\Gamma)\) is given such that
\(\depth_X(\Gamma) = \depth_X(\coh_{\Gamma,A}[\id]) = 1\) and define
\(X^{lm}\), \(X^m\) and \(X^b\) as in Appendix~\ref{app:nat-coh}. We will first
show correctness of the substitution \(\theta_{\Gamma,X}\) of
Definition~\ref{def:theta-subs}.

\begin{lemma}\label{lem:theta-subs-correctness}
  The following judgement is derivable:
  \[\Gamma\uparrow X \vdash \theta_{\Gamma,X} : \Gamma \uparrow X^{lm}\]
\end{lemma}
\begin{proof}
  Since \(\Gamma\) is reduced, every variable of dimension at most \(n-1\) belongs
  either in the source or target of \(\Gamma\), which justifies the last case of
  Definition~\ref{def:theta-subs}. Since \(X\) is up-closed, we can easily
  deduce that \(\theta_{\Gamma,X}\) preserves the type of every variable in
  \(\Var(\Gamma)\setminus X\). Preservation of types for variables of the form
  \(\fun x\) where \(x\in X^{lm}\) holds definitionally by the third case.
  For \(y\in X^b\) of type \(\obj\), preservation of typing is automatic. For
  \(y \in X^b\) of type \(\arr u v\), we have that the boundary of \(y\) in
  \(\Gamma\uparrow X^{lm}\) is again \(\arr{u}{v}\), while
  \[ \partial^\pm_{\Gamma\uparrow X}(y) = y^\pm \]
  which also has type \(\arr u v\). By the second case,
  we see that \(\theta_{\Gamma,X}\) preserves the type of \(y\).
  Finally, we need to check preservation of types for variables of the form
  \(x^\pm\) where \(x\in X^{lm}\). When \(x\) is of type \(\obj\), preservation
  of types is again automatic, so suppose again that \(x\) has type \(\arr u v\)
  in \(\Gamma\). Then the type of \(x^\pm\) in \(\Gamma\uparrow X^{lm}\) is
  again \(\arr u v\). On the other hand:
  \[
    \partial^-\partial^\pm(\fun x) = \begin{cases}
      u^- &\text{if } u \in X^{r} \\
      u &\text{if } u \not\in X^b
    \end{cases}
  \]
  Therefore the source of \(x^\pm\) is preserved by \(\theta_{\Gamma,X}\) by the
  last case of the definition  when \(u \in X^b\) and by the second case
  otherwise. The case for the target of \(x^\pm\) is symmetric.
\end{proof}

Combining Lemmas~\ref{lem:corr-nat-coh-lmax} and~\ref{lem:theta-subs-correctness}, we see that the middle term in the definition
of \(\coh_{\Gamma,A}\uparrow X\) is well-defined. We will now show correctness
for the first interchanger \(j^-_{\Gamma,X,A}\) appearing in the definition.
Correctness for the second interchanger \(j^+_{\Gamma,X,A}\) is completely
symmetric. For that, we will assume first that \(\Var(v) \cap X = \emptyset\).
The fullness condition on \(A\) implies that
\(\Var(u)\cup\Var(B) = \Var(\partial^+\Gamma)\). By up-closure of \(X\),
this implies that the case split can be equivalently be done on
\(\Var(\partial^+\Gamma)\cap X = \emptyset\). In this case, the desired source
of \(\coh_{\Gamma,A}\uparrow X\) is \(\coh_{\Gamma,A}[\inl_{\Gamma,X}]\), while
the actual source of the middle term is
\(\coh_{\Gamma,A}[\inl_{\Gamma,X^{lm}} \circ\ \theta_{\Gamma,X}]\). The
following lemma ensures that the interchanger \(j^-_{\Gamma,X,A}\) defined
in Appendix~\ref{app:nat-coh} is a well-defined term from the claimed source of
the naturality to the source of the middle term.

\begin{lemma}\label{lem:interchanger-trivial}
  When \(\Var(\partial^\pm\Gamma)\cap X = \emptyset\) then:
  \[
    \inc^\mp_{\Gamma,X} = \inc^\mp_{\Gamma,X^{lm}} \circ\ \theta_{\Gamma,X}
  \]
\end{lemma}

\begin{proof}
  It suffices to show that the two substitutions agree on every variable of
  \(\Gamma\). Variables \(y\in \Var(\Gamma)\setminus X\) are
  preserved by both substitutions. For variables \(y\in X^b\), we compute that
  \[
    y[\inc^\mp_{\Gamma,X}\theta_{\Gamma,X}] = y[\theta_{\Gamma,X}] = \partial^\mp(\fun y) = y^\mp = y[\inc^\mp_{\Gamma,X}]
  \]
  where the second to last equality follows by the assumption on depth. Finally,
  for variables \(y\in X^{lm}\), we compute similarly that
  \[
    y[\inc^\mp_{\Gamma,X}\theta_{\Gamma,X}] = y^-[\theta_{\Gamma,X}] = \partial^\mp(\fun y) = y^\mp = y[\inc^\mp_{\Gamma,X}]
  \]
  where the second to last equality follows by \(X\) not intersecting the
  variables of the boundary of \(\Gamma\).
\end{proof}

To complete the proof of correctness for the first interchanger, it remains to
check the case when \(\Var(\partial^+\Gamma)\cap X \not= \emptyset\). In this
case, the claimed source of \(\coh_{\Gamma,A}[\inl_{\Gamma,X}]\) is given by
\(\coh_{\Gamma,A}\uparrow X \s_{n-1} (v\uparrow X)\), while the source of the
middle term is given by \(\coh_{\Gamma,A}[\inl_{\Gamma,X^{lm}}\circ\
\theta_{\Gamma,X}] \s_{n-1} (v\uparrow X^{lm})[\theta_{\Gamma,X}]\). The
following lemmas will be used to show that \(j^-_{\Gamma,X,A}\) is again
well-defined in this case, and goes from the claimed source of the naturality to
the source of the middle term. This is established in Lemma~\ref{lem:interchanger-correct}.
Building towards this lemma, we will show that the terms \(p\), \(q\) appearing in the
definition of the interchanger are valid, parallel and define a full type. We will then
show that \(j^-_{\Gamma,X,A}\) as the correct type.

\begin{lemma}\label{lem:p-correct}
  The following judgement is derivable:
  \[\Phi \vdash p : \arr{u[\proj_1]}{v[\inr_{\partial^+\Gamma,Y}\proj_2]}\]
\end{lemma}
\begin{proof}
  By correctness of the naturality with respect to depth\=/\(0\) variables, we
  have that:
  \begin{align*}
    \Phi &\vdash \coh_{\Gamma,A}[\proj_1] : \arr{u[\proj_1]}{v[\proj_1]} \\
    \Phi &\vdash (v \uparrow Y)[\proj_2] : \arr
    {v[\inl_{\partial^+\Gamma,Y}\proj_2]}{v[\inr_{\partial^+\Gamma,Y}\proj_2]}
  \end{align*}
  Moreover, since \(v\) is well-typed over \(\partial^+\Gamma\) and \(\delta_\Gamma^+\)
  is a weakening substitution,
  \[v[\proj_1] = v[\delta_\Gamma^+\proj_1] = v[\inl_{\partial^+\Gamma,Y}\proj_2]\]
  where the last equality follows by the definition of the projections as given in Figure \ref{fig:intch-ctx}. It follows
  that the two terms are composable, so \(p\) satisfies the judgement above.
\end{proof}

\begin{lemma}\label{lem:psi-r-is-gamma}
  We have that \(\Psi^r = \Gamma\) up to \(\alpha\)\-equivalence.
\end{lemma}
\begin{proof}
  Let \(Z_\Psi = \{\partial^-x : \dim x = \dim \Psi\}\cap \partial^-\Psi\)
  and define \(Z_\Gamma\) analogously. Then, by the definition of the reduction of a pasting context~\cite[Section~4.4]{benjamin_invertible_2024}, we can deduce that
  \begin{align*}
    \Psi^r &= (\partial^-\Psi) \uparrow Z_\Psi &
    \Gamma^r &= (\partial^-\Gamma) \uparrow Z_\Gamma
  \end{align*}
up to \(\alpha\)\-equivalence. Moreover, since composition of Batanin trees satisfies the rules of strict \(\omega\)\=/categories~\cite{batanin_monoidal_1998}, we get that
\[
  \partial^-\Psi \cong \partial^-\Gamma
\]
with the isomorphism given by restricting \(\overline{\proj}_1\) on the source of both pasting contexts. Since pasting contexts have no non-trivial isomorphisms~\cite[Lemma~1.7]{berger_cellular_2002}, this isomorphism must be an equality up to \(\alpha\)\-equivalence. Under this equality, the sets \(Z_\Psi\) and \(Z_\Gamma\) coincide: indeed, the maximal dimensional variables of \(\Psi\) are the maximal dimensional variables of \(\Gamma\), included via \(\proj_1\), or those of the
form \(\fun y\) for \(y \in Y^b\) included via \(\proj_2\). The source of the latter can never intersect \(\partial^-\Psi\) though. Finally, since \(\Gamma\) was assumed to be reduced, we have that \(\Gamma^r = \Gamma\), completing the proof.
\end{proof}

\begin{lemma}\label{lem:q-correct}
  The following judgement is derivable:
  \[\Phi \vdash q : \arr{u[\rho_{\Psi} \circ \psi]}{v[\inr_{\partial^+\Gamma,Y}\proj_2]}\]
\end{lemma}

\begin{proof}
  By correctness of the naturality with respect to depth\=/0 variables, we
  have that:
  \begin{align*}
    \Phi \vdash &\coh_{\Gamma,A}[\rho_\Psi \psi] : \arr{u[\rho_\Psi \psi]}{v[\rho_\Psi \psi]} \\
    \Phi \vdash &(v \uparrow Y^l)[\inr_{\partial^+\Gamma \uparrow Y^l, Y^b}\proj_2] :
    {v[\inl_{\partial^+\Gamma,Y^l}\inr_{\partial^+\Gamma \uparrow Y^l, Y^b}\proj_2]} \\
    &\to {v[\inr_{\partial^+\Gamma,Y^l}\inr_{\partial^+\Gamma \uparrow Y^l, Y^b}\proj_2]}
  \end{align*}
  The target of the second term is the claimed target of \(q\) by Lemma~\ref{lem:func-of-func}. We want to show the two terms are composable and thus \(q\) is well-defined. Since \(v\) lives over \(\partial^+\Gamma\), it suffices to show that:
  \[ \delta^+_\Gamma \rho_\Psi \psi = \inl_{\partial^+\Gamma,Y^l}\inr_{\partial^+\Gamma \uparrow Y^l, Y^b}\proj_2\]
  As explained in Section~\ref{sec:naturality}, the reduction substitution \(\rho_\Psi\) is the identity on the boundary, so
  combined with Lemma~\ref{lem:psi-r-is-gamma}, we see that:
  \[
    \delta^+_\Gamma \rho_\Psi = \delta^+_{\Psi^r} \rho_\Psi = \delta^+_{\Psi}
  \]
  Using again that composition of trees satisfies the rules of strict \(\omega\)\=/categories, we get that
  \(\overline{\proj}_2\) is an isomorphism between the target of \(\Psi\) and that of \(\partial^+\Gamma \uparrow Y^b\). The latter can be identified with \(\partial^+\Gamma\) via \(\inr_{\partial^+\Gamma,Y^r}\) by Lemma~\ref{lem:psctx-func}.

  In other words, we have that:
  \[
    \delta_\Psi^+ = \delta^+_{\partial^+\Gamma \uparrow Y^l} \overline{\proj}_2 = \inr_{\partial^+\Gamma, Y^l} \overline{\proj}_2
  \]
  Combining those equations, we compute that:
  \begin{align*}
    \delta^+_\Gamma \rho_\Psi \psi
      &= \delta^+_{\Psi} \psi
      = \inr_{\partial^+\Gamma, Y^l} \overline{\proj}_2 \psi \\
      &= \inr_{\partial^+\Gamma, Y^l}\inr_{\partial^+\Gamma \uparrow Y^l, Y^b}\proj_2
  \end{align*}
  Here the last equality follows by the definition of \(\psi\).
\end{proof}

\begin{lemma}\label{lem:pq-correct}
  The following judgement is derivable:
  \[\Phi \vdash \arr p q\]
  Moreover, this type is full.
\end{lemma}
\begin{proof}
  To show that this type is derivable, we need that the source and target of \(p\) and \(q\) coincide.
  Lemmas~\ref{lem:p-correct} and~\ref{lem:q-correct} show that the targets agree, so it remains to show that:
  \[ u[\proj_1] = u[\rho_\Psi \psi] \]
  Since \(u\) lives over the source of \(\Gamma\), this amounts to:
  \[\delta^-_\Gamma \proj_1 = \delta_\Gamma^-\rho_\Psi \psi\]
  As in the proof of Lemma~\ref{lem:q-correct}, we can compute that
  \begin{align*}
    \delta_\Gamma^-\rho_\Psi \psi
      = \delta^-_{\Psi} \psi
      = \delta^-_\Gamma \overline{\proj}_1 \psi
      = \delta_\Gamma^-\proj_1
  \end{align*}
  where the last equality follows by definition of \(\psi\) in Figure~\ref{fig:intch-ctx}. This shows that the judgement is derivable.

  It remains to show that the type \(\arr p q\) is full. For that, we will use
  that variables of \(\Phi\) are either of the form \(x[\proj_1]\) for some \(x\in \Var(\Gamma)\),
  or of the form \(y[\proj_2]\) for some \({y \in \Var(\partial^+\Gamma \uparrow Y)}\). The former are
  in \(\Var(p)\) due to its first component, while the latter are there by the second
  component in conjunction with Lemma~\ref{lem:fullterm-fullfunct}. A similar argument
  can be used to show that \(\Var(q)\) contains every variable of \(\Phi\).
  Alternatively, this is guaranteed by \(q\) being parallel to a term containing
  all variables of the context~\cite[Corollary~103]{benjamin_type_2020}.
\end{proof}

\begin{lemma}\label{lem:interchanger-correct}
  The following judgement is derivable:
  \begin{align*}
    \Phi \vdash j^-_{\Gamma,X,A} : &\coh_{\Gamma,A}[\inl_{\Gamma,X}] \s_{n-1} (v\uparrow X) \to \\
    &\coh_{\Gamma,A}[\inl_{\Gamma,X^{lm}}\theta_{\Gamma,X}] \s_{n-1} (v\uparrow X^{lm})[\theta_{\Gamma,X}]
  \end{align*}
\end{lemma}
\begin{proof}
  Lemma~\ref{lem:pq-correct} shows that \(j^-_{\Gamma,X,A}\) is a
  well-defined term over \(\Phi\) of type \(p[\phi]\to q[\phi]\). To show the
  claimed judgement, we need to show therefore the following for the two types to coincide:
  \begin{gather*}
    \proj_1\phi = \inl_{\Gamma,X} \\
    (v\uparrow Y)[\proj_2\phi] = (v\uparrow X) \\
    \rho_\Psi\psi\phi = \inl_{\Gamma,X^{lm}}\theta_{\Gamma,X} \\
    (v\uparrow Y^l)[\inr_{\partial^+\Gamma \uparrow Y^l, Y^b}\proj_2\phi] = (v\uparrow X^{lm})[\theta_{\Gamma,X}]
  \end{gather*}
  The first equation holds by definition of \(\phi\).
  The second one holds again by definition of \(\phi\) in Figure \ref{fig:intch-ctx} and Lemma~\ref{lem:pullback-stability}.

  We prove the last two equations by testing the substitutions on depth\=/0 variables ${y \in \Var(\Gamma)}$.
  If $y \notin X$, then we can immediately check that:
  \begin{align*}
          y[\rho_\Psi\psi\phi] &= y[\overline{\proj_1} \psi\phi] = y[\inl_{\Gamma,X}] = y = y[\inl_{\Gamma,X^{lm}}\theta_{\Gamma,X}]
  \end{align*}
  Assume now $y \in X$ and recall that over $\Gamma \uparrow X$:
  \begin{align*}
        \partial^-(\fun{y}) =
        \begin{cases}
                y^- &\text{if } \partial^+(y) \not\in X \\
                y^- \s_{n} \fun{\partial^+(y)} &\text{if } \partial^+(y) \in X
        \end{cases}
  \end{align*}
  where $n = \dim \Gamma-1$.
  Hence in the case that $\partial^+(y) \not\in X$:
  \begin{align*}
          y[\rho_\Psi\psi\phi] &= y[\overline{\proj_1}\psi\phi] = y[\inl_{\Gamma,X}] = y^- = y[\inl_{\Gamma,X^{lm}}\theta_{\Gamma,X}]
  \end{align*}
  Finally, let us assume $z = \partial^+(y) \in X$ and verify:
  \begin{align*}
          y[\rho_\Psi\psi\phi] &= (y[\overline{\proj_1}] *_n \fun{z}[\overline{\proj_2}])[\psi\phi] \\
          &= y[\inl_{\Gamma,X}] *_n \fun{z}[\inl_{\partial^+\Gamma\uparrow Y^b,Y^l} (\delta^+_{\Gamma} \uparrow X)] \\
          &= y^- *_n \fun{z} = y[\inl_{\Gamma,X^{lm}}\theta_{\Gamma,X}]
  \end{align*}
  This concludes the proof of the third equation.
  For the fourth equation, we use Lemma~\ref{lem:pullback-stability} to show:
  \begin{align*}
          v \uparrow X^{lm} = v[\delta^+_\Gamma] \uparrow X^{lm} = (v \uparrow Y^l)[\delta^+_\Gamma \uparrow X^{lm}]
  \end{align*}
  Hence by using $\proj_2\phi = \delta^+_\Gamma \uparrow X$, it suffices to show the following substitutions from $\Gamma \uparrow X$ to $\partial^+ \Gamma \uparrow Y^l$ coincide:
  \begin{align*}
        \inr_{\partial^+\Gamma \uparrow Y^l, Y^b} (\delta^+_\Gamma \uparrow X) = (\delta^+_\Gamma \uparrow X^{lm}) \theta_{\Gamma,X}
  \end{align*}
  Let $y \in \Var(\delta^+ \Gamma)$ be a depth\=/0 variable in $\partial^+ \Gamma$.
  If $y \notin Y$, then it is immediate to see:
  \begin{align*}
        y[\inr_{\partial^+\Gamma \uparrow Y^l, Y^b}][\delta^+_\Gamma \uparrow X] = y = y[\delta^+_\Gamma \uparrow X^{lm}][\theta_{\Gamma,X}]
  \end{align*}
  If $y \in Y^b$, then we have:
  \begin{align*}
        y[\inr_{\partial^+\Gamma \uparrow Y^l, Y^b}][\delta^+_\Gamma \uparrow X] &= y^+[\delta^+_\Gamma \uparrow X] \\
        &= y^+ = \partial^+(\fun{y}) \\
        &= y[\theta_{\Gamma,X}] = y[\delta^+_\Gamma \uparrow X^{lm}][\theta_{\Gamma,X}]
  \end{align*}
  Finally, if $y \in Y^l$, we have:
  \begin{align*}
          \fun{y}[\inr_{\partial^+\Gamma \uparrow Y^l, Y^b}][\delta^+_\Gamma \uparrow X] = \fun{y} = \fun{y}[\delta^+_\Gamma \uparrow X^{lm}][\theta_{\Gamma,X}]
  \end{align*}
  Since all cases are exhausted, this concludes the proof.
\end{proof}

\begin{lemma}\label{lem:red-comp-correct}
	Let $\Gamma \vdashps$ be a reduced pasting context, ${\Gamma \vdash A}$ a full type and $X \in \U(\Gamma)$.
	Writing $t = \coh_{\Gamma,A}[\id]$, if $\depth_X \Gamma = \depth_X t = 1$,
	then the following judgement is derivable:
\[
	\Gamma \uparrow X \vdash \coh_{\Gamma,A} \uparrow X : A \uparrow^{t} X
\]
\end{lemma}
\begin{proof}
	Let us write $A = \arr[B]{u}{v}$, and recall that we have:
\[
  \coh_{\Gamma,A}\uparrow X = j^-_{\Gamma,X,A}
  \s_n (\coh_{\Gamma,A}\uparrow X^{lm})[\theta_{\Gamma,X}]
  \s_n j^+_{\Gamma,X,A}
\]
Validity of the middle term is given by Lemmas~\ref{lem:corr-nat-coh-lmax} and~\ref{lem:theta-subs-correctness}.
Validity of $j^-_{\Gamma,X,A}$ is immediate if $X \cap \Var(v) = \emptyset$ and proven by Lemma \ref{lem:pq-correct} otherwise.
Moreover, by Lemma \ref{lem:interchanger-trivial} and Lemma \ref{lem:interchanger-correct} respectively, we can show $j^-_{\Gamma,X,A}$ has the expected source and target in each case.
Since $j^+_{\Gamma,X,A}$ is completely symmetric, similar arguments show that it is well-defined and has the correct type.
\end{proof}
This concludes the proof of correctness for reduced composites.
Since the general case is described in the proof of Theorem \ref{thm:correctness}, the proof the main Theorem is now complete.


%% file: appendix/cones.tex

\section{Cylinders, Composition and Stacking}\label{app:cylinders}

This appendix is a complement to Sec.~\ref{sec:cylinders}, presenting detailed
proofs of the results in this section. We start with
Fig.~\ref{fig:cyl-ctx-computation} giving an explicit and full descriptions of
the cylinder contexts \(\cylctx^{2}\) and \(\cylctx^{3}\) presented in
Fig.~\ref{fig:cylinder-ctx}.

\begin{figure*}
  \begin{gather*}
    \cylctx^{2} =\left(
    \makecell[l]{
    \bot_{1}^{-}:\obj, \bot_{1}^{+}:\obj, \bot_{2} : \arr{\bot_{1}^{-}}{\bot_{1}^{+}},\\
    \top_{1}^{-}:\obj, \top_{1}^{+}:\obj, \top_{2} :
    \arr{\top_{1}^{-}}{\top_{1}^{+}}, \\
    \filler_{1}^{-}:\arr{\bot_{1}^{-}}{\top_{1}^{-}},
    \filler_{1}^{+}:\arr{\bot_{1}^{+}}{\top_{1}^{+}},\\
    \filler_{2}:\arr{\bot_{2}\s\filler_{1}^{+}}{\filler_{1}^{-}\s\top_{2}}
    }
    \right)
    \\
    \cylctx^{3} =
      \left(
      \makecell[l]{
      \bot_{1}^{-}:\obj, \bot_{1}^{+}:\obj,
      \bot_{2}^{-} : \arr{\bot_{1}^{-}}{\bot_{1}^{+}},
      \bot_{2}^{+} : \arr{\bot_{1}^{-}}{\bot_{1}^{+}},
      \bot_{3} : \arr{\bot_{2}^{-}}{\bot_{2}^{+}},\\
    \top_{1}^{-}:\obj, \top_{1}^{+}:\obj,
    \top_{2}^{-} : \arr{\top_{1}^{-}}{\top_{1}^{+}},
    \top_{2}^{+} : \arr{\top_{1}^{-}}{\top_{1}^{+}},
    \top_{3} : \arr{\top_{2}^{-}}{\top_{2}^{+}}, \\
    \filler_{1}^{-}:\arr{\bot_{1}^{-}}{\top_{1}^{-}},
    \filler_{1}^{+}:\arr{\bot_{1}^{+}}{\top_{1}^{+}},\\
    \filler_{2}^{-}:\arr{\bot_{2}^{-}\s\filler_{1}^{+}}{\filler_{1}^{-}\s\top_{2}^{-}},
    \filler_{2}^{+}:\arr{\bot_{2}^{+}\s\filler_{1}^{+}}{\filler_{1}^{-}\s\top_{2}^{+}},\\
    \filler_{3} : \arr{(\bot_{3}\s_0\filler_{1}^{+})\s_1\filler_{2}^{+}}{\filler_2^{-}\s_1(\filler_1^-\s_0\top_{3})}
    }
    \right)
\end{gather*}

\caption{The cylinder contexts of dimension \(2\) and \(3\) of
  Fig.~\ref{fig:cylinder-ctx}.}
 \label{fig:cyl-ctx-computation}
\end{figure*}



Our aim is now to give a proof of Proposition~\ref{prop:cyl-type}. We
first prove that the equations postulated in this lemma hold in the contexts
\(\partial\cylctx^{n}\).
\begin{lemma}\label{lemma:equations-cyl}
  The following judgements are derivable:
  \begin{mathpar}
    \hspace{-1cm}(i)\, \partial\cylctx^{1} \!\vdash\! \top_{1}\!:\!\obj\!\!\!\!\!\! \and \!\!\!\!\!\!(ii)\,\partial\cylctx^{1} \!\vdash\!
    \bot_{1}\!:\!\obj\!\!\!\!\!\! \and \!\!\!\!\!\!
    \!\!\!\!(iii)\,\cylctx^1 \!\vdash\! \filler_1 \!:\! \!{\top_1}\!\to\!{\bot_1}\hspace{-1cm}
    \\
    \mbox{$\hspace{0cm}(iv)\,\partial\cylctx^{2} \vdash \filler_{1}^- : \Cyl^1(\top_1^-,\bot_1^-)\hspace{0pt}$}
\and \hspace{0pt}
    (v)\,\partial\cylctx^{2} \vdash \filler_{1}^+ : \Cyl^1(\top_n^+,\bot_n^+)\hspace{0cm} \\
    (vi)\,\partial\cylctx^{n+1}\vdash \top_{n+1}:\arr{\top_n^-}{\top_{n}^+} \and
    (vii)\,\partial\cylctx^{n+1}\vdash \bot_{n+1}:\arr{\bot_n^-}{\bot_{n}^+}\\
    (viii)\,\cylctx^{n+1} \vdash \filler_{n+1}:
    \Cyl^{n+1}(\top_{n+1},\bot_{n+1},\filler_n^-,\filler_n^+)\\
    (ix)\,\partial\cylctx^{n+2} \vdash \filler_{n+1}^-:
    \Cyl^{n+1}(\top_{n+1}^-,\bot_{n+1}^+,\filler_n^-,\filler_n^+)\and
    (x)\,\partial\cylctx^{n+2} \vdash \filler_{n+1}^+:
    \Cyl^{n+1}(\top_{n+1}^+,\bot_{n+1}^+,\filler_n^-,\filler_n^+)
    \end{mathpar}
\end{lemma}
\begin{proof}
  The derivation of the judgements \textit{(i)-(iii)} is immediate. Judgements \textit{(iv)} and \textit{(v)} follow from Theorem~\ref{thm:correctness} and the judgement \textit{(iii)}. Judgements \textit{(vi)} and \textit{(vii)} follow from
  Theorem~\ref{thm:correctness}. Judgement \textit{(viii)} is the definition of the
  type \(\Cyl^{n}\), and judgements \textit{(ix)} and \textit{(x)} follow from
  Theorem~\ref{thm:correctness} and judgement \textit{(viii)}.
\end{proof}

Then, we prove that cylinder types are preserved under the action of
substitutions.
\begin{lemma}\label{lemma:img-cylinder}
  Let \(\Gamma\vdash t:\Cyl^{n}(a,b,c,d)\) be a cylinder and
  \(\Delta\vdash\gamma:\Gamma\) a substitution, then \(t[\gamma]\) is a cylinder
  satisfying:
  \[
    \Delta\vdash t[\gamma]:\Cyl^{n}(a[\gamma],b[\gamma],c[\gamma],d[\gamma])
  \]
\end{lemma}
\begin{proof}
  A cylinder \(\Gamma\vdash t : \Cyl^{n}(a,b,c,d)\) is defined by the existence
  of a substitution \(\Gamma\vdash \gamma_{0} : \partial\cylctx^{n}\) such that
  \(\Gamma\vdash t:\Cyl^{n}[\gamma_{0}]\). Given a substitution
  \(\Delta\vdash \gamma:\Gamma\), we get the judgement:
  \[
    \Delta\vdash t[\gamma]:\Cyl^{n}[\gamma_{0}\circ\gamma]
  \]
  Hence \(t[\gamma]\) is a cylinder satisfying the desired equality.
\end{proof}

We also need the following result, characterising the substitutions whose target
is a context obtained by functorialisation. We expect a generalisation of this
result to hold for naturality as well, but to be significantly harder to prove.
\begin{lemma}\label{lemma:univ-prop-functorialisation}
  Consider a context \(\Gamma\vdash\) and a set \(X\in \U(\Gamma)\) of variables
  of depth 0. Given two substitutions
  \({\Delta\vdash \gamma^{-}:\Gamma}\) and \(\Delta\vdash \gamma^{+}:\Gamma\) that
  coincide on every variable not in \(X\), together with, for every \(x\in X\),
  a family of terms
  \[
    \Delta\vdash t_{x} : \arr[]{x[\gamma^{-}]}{x[\gamma^{+}]}
  \]
  there exists a unique substitution \(\Delta\vdash \gamma:\Gamma \uparrow X\)
  such that:
  \begin{align*}
    \gamma\circ \inj^{-}_{\Gamma,X} &= \gamma^{-}
    & \gamma\circ \inj^{+}_{\Gamma,X} &= \gamma^{+}
                                        \intertext{Furthermore, for every \(x\in X\):}
                                        x[\gamma] & =t_{x}
  \end{align*}
\end{lemma}
\begin{proof}
  We prove this result by induction on the context \(\Gamma\). For the empty
  context \(\emptycontext\), necessarily \(X = \emptyset\), and there is a
  unique substitution \(\Delta\vdash \sub{}:\emptycontext\), so the result
  trivialises. For the context \((\Gamma,x:A)\) when \(x \notin X\), the two
  substitutions decompose as follow: \(\sub{\gamma^{-},x\mapsto t}\) and
  \(\sub{\gamma^{+},x\mapsto t}\). Then by induction, there is a unique
  substitution \(\Delta\vdash \gamma:\Gamma \uparrow X\) satisfying the
  appropriate conditions, and thus a unique substitution satisfying
  the desired conditions:
  \[
    \Delta\vdash \sub{\gamma,x\mapsto t} : (\Gamma,x:A) \uparrow X
  \]
  For the context \((\Gamma,x:A)\) with \(x\in X\), the
  two substitution decompose as \(\sub{\gamma^{-},x\mapsto t_{0}}\) and
  \(\sub{\gamma^{+},x\mapsto t_{1}}\). By induction, there is a unique
  substitution \(\Delta\vdash \gamma:\Gamma\uparrow X\setminus\{x\}\) satisfying
  the adequate conditions. This gives a unique substitution satisfying the
  desired conditions:
  \[
    \Delta\vdash \sub{\gamma,x^{-}\mapsto t_{0},x^{+}\mapsto
      t_{1},\fun{x}\mapsto t_{x}}
      \qedhere
  \]
\end{proof}

\begin{proof}[Proof of Proposition~\ref{prop:cyl-type}]
  We consider a context \(\Gamma\), and we prove this result by case disjunction
  on the dimension \(n\) of the cylinder type in \(\Gamma\) that we consider.

  The type \(\Cyl^{1}(a,b)\) is well defined if we have a substitution
  \(\Gamma\vdash\gamma:\partial\cylctx^{1}\), and \(a = \top_{1}[\gamma]\),
  \(b =\bot_{1}[\gamma]\). This is equivalent to:
  \begin{align*}
    \Gamma\vdash a:\obj && \Gamma\vdash b:\obj
  \end{align*}

  The type \(\Cyl^{n+1}(a,b,c,d)\) is well defined if and only if we have a
  substitution \(\Gamma\vdash\gamma:\partial\cylctx^{n+1}\), such that:
    \begin{align*}
      \top_{n}[\gamma] &= a & \bot_{n}[\gamma]&= b & \filler_{n-1}^{-}[\gamma]
      &= c & \filler_{n-1}^{+}[\gamma] &= d
    \end{align*}
    By Lemma~\ref{lemma:equations-cyl}, in the context
    \(\partial\cylctx^{n+1}\), the terms \(\filler_{n-1}^{-}\) and
    \(\filler_{n-1}^{+}\) are cylinders satisfying the equations:
    \begin{align*}
      \topface(\filler_{n-1}^{-}) &= \partial^{-}\top_{n}
      & \topface(\filler_{n-1}^{+}) &= \partial^{+}\top_{n} \\
      \botface(\filler_{n-1}^{-}) &= \partial^{-}\bot_{n}
      & \botface(\filler_{n-1}^{+}) &= \partial^{+}\bot_{n} \\
      \back(\filler_{n-1}^{-}) &= \back(\filler_{n-1}^{+})
      & \front(\filler_{n-1}^{-}) &= \front(\filler_{n-1}^{+})
    \end{align*}
    And the desired equations are the images of these by \(\gamma\), which hold
    by Lemma~\ref{lemma:img-cylinder}. Conversely, by definition, we have the
    equality:
    \[
      \partial\cylctx^{n+1} =
      \left(
        \begin{array}{l}
          \partial\cylctx^{n}\uparrow\{\top_{n},\bot_{n}\}, \\
          \filler_{n}^{-}:
          \Cyl^{n}(\top_{n}^{-},\bot_{n}^{-},\filler_{n-1}^{-},\filler_{n-1}^{+}),\\
          \filler_{n}^{+}
          : \Cyl^{n}(\top_{n}^{+},\bot_{n}^{+},\filler_{n-1}^{-},\filler_{n-1}^{+})
        \end{array}
      \right)
    \]
    Then, consider two cylinders \(\Gamma \vdash c:\Cyl^{n}[\gamma_{c}]\) and
    \({\Gamma \vdash d:\Cyl^{n}[\gamma_{d}]}\) such that:
    \begin{align*}
      \back_{\Gamma}(c) &= \back_{\Gamma}(d) & \front_{\Gamma}(c) &= \front_{\Gamma}(d)
    \end{align*}
    This reformulates as \(\gamma_{c}\) and \(\gamma_{d}\) agreeing on every
    variable that is not in \(X\). Together with terms
    \begin{align*}
      \Gamma & \vdash a : \arr[]{\topface_{\Gamma}(c)}{\topface_{\Gamma}(d)}\\
      \Gamma & \vdash b : \arr[]{\botface_{\Gamma}(c)}{\botface_{\Gamma}(d)}
    \end{align*}
    these define by Lemma~\ref{lemma:univ-prop-functorialisation} a unique
    substitution
	\({\Gamma\vdash \gamma : \partial\cylctx^{n}\uparrow\{\top_{n},\bot_{n}\}}\)
    such that:
    \begin{align*}
      \top_{n+1}[\gamma] &= a & \bot_{n}[\gamma] &= b \\
      \gamma\circ\inj^{-}_{\partial\cylctx^{n},\{\top_{n},\bot_{n}\}}
                         &=\gamma_{c}
      & \gamma\circ\inj^{+}_{\partial\cylctx^{n},\{\top_{n},\bot_{n}\}}
      &=\gamma_{d}
    \end{align*}
    Thus, we can complete this uniquely into a substitution
    \[
      \Gamma\vdash \sub{\gamma,\filler_{n}^{-} \mapsto c,\filler_{n}^{+}\mapsto
        d} : \partial\cylctx^{n+1}. \qedhere
    \]
\end{proof}

We conclude this appendix with an explicit construction of the interchangers
that appear in the proof of Theorem~\ref{thm:cyl-comp}. We recall that our goal
is to produce the cylindrical composition
\(\prescript{k+1}{}\cyls^{k+1}_{k}\), for cylinders in a context \(\Gamma\),
assuming that the composition \(\prescript{k}{}{\cyls}^{k}_{k-1}\) has
already been defined. Our key remark is that the cylinder type
\(\Cyl^{k+1}(a,b,c,d)\) can be obtained as a suspension of a cylinder type of
the previous dimension, in which we substitute a composition, as follows
\[
  \Cyl^{k+1}(a,b,c,d) = (\Sigma\Cyl^{k})(a\s_{0}\front^{1}(c),\back^{1}(d)\s_{0}b,c,d)
\]
This can be observed directly on the formula for cylinder types presented in
Fig.~\ref{fig:cyl-type}. Given two cylinders
\begin{calign}
  \nonumber
  {\Gamma\vdash m:\Cyl^{k+1}(a_{0},b_{0},c,d)}
  &
  \nonumber
  \Gamma\vdash n:\Cyl^{k+1}(a_{1},b_{1},d,e)
  \end{calign}
we consider the following term:
\[
  u = \Sigma(\prescript{k}{}\cyls^{k}_{k-1}) (m,n)
\]
By induction, this term satisfies the following judgement, which is not exactly the desired type:
\[
  \Gamma \vdash u : \Cyl^{k+1} \left(
  \begin{array}{l}
    (a_{0}\s_{0}\front^{1}(c)) \s_{k} (a_{1}\s_{0}\front^{1}(d)), \\
    (\back^{1}(d)\s_{0}b_{0}) \s_{k}(\back^{1}(e)\s_{0}b_{1}), \\
    c,e
  \end{array}\right)
\]
The equalities from
Proposition~\ref{prop:cyl-type} imply that we have:
\begin{align*}
  \front^{1}(d) &= \front^{1}(c) & \back^{1}(d) &= \back^{1}(e)
\end{align*}
To account for these we introduce interchangers. We introduce the interchangers
\(j^\pm\) as coherences of the following types:
\begin{align*}
  j^- &: (a\s_{k} b) \s_{0} c \to (a \s_{0} c) \s_{k} (b \s_{0} c) \\
  j^+ &:  (a \s_{0} b) \s_{k} (a \s_{0} c) \to a \s_{0} (b\s_{k} c)
\end{align*}
We then define \(s\) and \(t\) to be respectively the source and target of the
type \(\Sigma\Cyl^{k}(a,b,c,d)\), and we define:
\begin{align*}
  j_{\square,k}^- & = (s \uparrow\{a\})[j^-(a_{0},a_{1},\front^{1}(c))] \\
  j_{\square,k}^+ &= (t\uparrow \{b\}) [j^+(\back^{1}(e),b_{0},b_{1})]\\
  m \prescript{k+1}{}{\cyls}^{k+1}_{k} n & = j_{\square,k}^- \s u \s j_{\square,k}^+
\end{align*}
The composition defined this way has type:
\begin{align*}
  & \Sigma(\Cyl^{k})((a_{0}\s_{k}a_{1})\s_{0}\front^{1}(c), \back^{1}(d)\s_{0}(b_{0}\s_{k}b_{1}),c,d) \\
  & =\Cyl^{k+1}(a_{0}\s_{k}a_{1},b_{0}\s_{k}b_{1},c,d)
\end{align*}

\section{Cones and Conical Compositions}
This appendix is a complement to Sec.~\ref{sec:cones}, and follows the same
structure than Appendix~\ref{app:cylinders}. We start with an explicit
computation of the cone contexts of dimension \(2\) and \(3\) presented in
Fig.~\ref{fig:cone-ctx-computation}.
\begin{figure*}
  \begin{gather*}
    \conectx^{2} = \left(\makecell[l]{ \apex:\obj, \base_{1}^{-}:\obj,
    \base_{1}^{+}:\obj,\\
    \base_{2}:\arr{\base_{1}^{-}}{\base_{1}^{+}},\\
    \filler_{1}^{-}:\arr{\base_{1}^{-}}{\apex},
    \filler_{1}^{+}:\arr{\base_{1}^{+}}{\apex},\\
    \filler_{2}:\arr{\filler_{1}^{-}}{\base_{2}\s\filler_{1}^{+}}} \right)
    \\
    \conectx^{2} = \left( \makecell[l]{
    \apex:\obj, \base_{1}^{-}:\obj,
    \base_{1}^{+}:\obj,
    \base^{-}_{2}:\arr{\base_{1}^{-}}{\base_{1}^{+}},
    \base^{+}_{2}:\arr{\base_{1}^{-}}{\base_{1}^{+}},\\
    \base_{3}:\arr{\base_{2}^{-}}{\base_{2}^{+}},\\
    \filler_{1}^{-}:\arr{\base_{1}^{-}}{\apex},
    \filler_{1}^{+}:\arr{\base_{1}^{+}}{\apex},\\
    \filler_{2}^{-}:\arr{\filler_{1}^{-}}{\base_{2}^{-}\s_{0}\filler_{1}^{+}},
    \filler_{2}^{+}:\arr{\filler_{1}^{-}}{\base_{2}^{+}\s_{0}\filler_{1}^{+}},\\
    \filler_{3} : \arr{\filler_2^-\s_{1}(\base_3 \s_0\filler_1^+)}{\filler_2^+}}
    \right)
  \end{gather*}
  \caption{The Cone contexts of dimension $2$ and $3$ of
    Fig.~\ref{fig:cone-ctx}.}
  \label{fig:cone-ctx-computation}
\end{figure*}

\begin{lemma}\label{lemma:equations-cone}
  The following judgements are derivable:
  \begin{mathpar}
    \partial\conectx^1\vdash \top : \obj \and \partial\conectx^1 \vdash
    \base_1:\obj \conectx^1\vdash \filler_1 : \Cone^1(\base_1,\top)
    \\
    \partial\conectx^2 \vdash \filler_1^- :\Cone^1(\base_1^-,\top) \and
    \partial\conectx^2 \vdash \filler_1^+ :\Cone^1(\base_1^+,\top) \\
    \partial\conectx^{n+1} \vdash \base_{n+1} : \arr[]{\base_n^-}{\base_n^+} \\
    \conectx^{n+1} \vdash \filler_{n+1} :
    \Cone^{n+1}(\base_{n+1},\filler_n^-,\filler_n^+) \\
    \partial\conectx^{n+2} \vdash \filler_{n+1}^- :
    \Cone^{n+1}(\base_{n+1}^-,\filler_n^-,\filler_n^+) \and
    \partial\conectx^{n+2} \vdash \filler_{n+1}^+ :
    \Cone^{n+1}(\base_{n+1}^+,\filler_n^-,\filler_n^+)
  \end{mathpar}
\end{lemma}
\begin{proof}
  The derivation of the first \(3\) judgements is immediate. The fourth and
  fifth judgements follow from Theorem~\ref{thm:correctness} and the third
  judgement. The sixth and seventh judgements follow from
  Theorem~\ref{thm:correctness}. The eighth judgement is the definition of the
  type \(\Cyl^{n}\) and the last two judgements follow from
  Theorem~\ref{thm:correctness} and the eighth one.
\end{proof}

\begin{lemma}\label{lemma:img-cone}
  Let \(\Gamma\vdash t:\Cone^{n}(a,b,c)\) be a cone and
  \({\Delta\vdash\gamma:\Gamma}\) a substitution, then \(t[\gamma]\) is a cone
  satisfying:
  \[
    \Delta\vdash t[\gamma]:\Cone^{n}(a[\gamma],b[\gamma],c[\gamma])
  \]
\end{lemma}
\begin{proof}
  The proof is the same than that of Lemma~\ref{lemma:img-cylinder}.
\end{proof}

\begin{proof}[Proof of Proposition~\ref{prop:cone-type}]
  We proceed the same way as for the proof of Proposition~\ref{prop:cyl-type},
  considering a context \(\Gamma\), and proceeding by case disjunction on the
  dimension of the cone type.

  The type \(\Cone^{1}(a,b)\) is well defined if we have a substitution
  \(\Gamma\vdash\gamma:\partial\conectx^{1}\), and \(a = \base_{1}[\gamma]\), \(b =\top[\gamma]\). This is equivalent to:
  \begin{align*}
    \Gamma\vdash a:\obj && \Gamma\vdash b:\obj
  \end{align*}

  The type \(\Cone^{n+1}(a,b,c)\) is well defined if and only if we have a
  substitution \(\Gamma\vdash\gamma:\partial\conectx^{n+1}\), such that:
  \begin{align*}
    \base_{n}[\gamma] &= a & \filler_{n-1}^{-}[\gamma]&= b & \filler_{n-1}^{+}[\gamma]
    &= c
  \end{align*}
  By Lemma~\ref{lemma:equations-cone}, in the context
  \(\partial\conectx^{n+1}\), the terms \(\filler_{n-1}^{-}\) and
  \(\filler_{n-1}^{+}\) are cones satisfying the equations:
  \begin{align*}
    \basecone(\filler_{n-1}^{-}) &= \partial^{-}\base_{n}
    & \basecone(\filler_{n-1}^{+}) &= \partial^{+}\base_{n} \\
    \backcone(\filler_{n-1}^{-}) &= \backcone(\filler_{n-1}^{+})
    & \frontcone(\filler_{n-1}^{-}) &= \frontcone(\filler_{n-1}^{+})
  \end{align*}
  And the desired equations are the images of these by \(\gamma\), which hold by
  Lemma~\ref{lemma:img-cone}. Conversely, two cylinders
  \(\Gamma \vdash b:\Cone^{n}[\gamma_{b}]\) and
  \(\Gamma \vdash c:\Cone^{n}[\gamma_{c}]\) such that:
  \begin{align*}
    \backcone_{\Gamma}(b) &= \backcone_{\Gamma}(c)
    & \frontcone_{\Gamma}(b) &= \frontcone_{\Gamma}(c)
  \end{align*}
  Thus, \(\gamma_{b}\) and \(\gamma_{c}\) agree on every variable that is not in
  \(X\). Together with a \(\Gamma\vdash a : \arr[]{\basecone(b)}{\basecone(c)}\)
  these define by Lemma~\ref{lemma:univ-prop-functorialisation} a unique
  substitution
  \[
    \Gamma\vdash \gamma : \partial\conectx^{n} \uparrow\{\base_{n}\}
  \]
  such that
the following hold:  \begin{align*}
    \base_{n+1}[\gamma] &= a \\
    \gamma\circ\inj^{-}_{\partial\conectx^{n},\{\base_{n}\}}
                        &=\gamma_{b}
                        & \gamma\circ\inj^{+}_{\partial\conectx^{n},\{\base_{n}\}}
                        &=\gamma_{c}
  \end{align*}
  Thus, we can complete this uniquely into a substitution:
  \[
    \Gamma\vdash \sub{\gamma,\filler_{n}^{-} \mapsto b,\filler_{n}^{+}\mapsto c}
    : \partial\conectx^{n+1} \qedhere
  \]
\end{proof}

Like in the previous section, we conclude with a description of the
interchangers that appear in the proof of Theorem~\ref{thm:cone-comp}. Recall
that we are trying to produce, in a context \(\Gamma\), the conical composition
\(\prescript{k+1}{}\cs^{k+1}_{k}\), assuming that the composition
\(\prescript{k}{}{\cs}^{k}_{k-1}\) has already been defined. For this, we use
the following equation that holds on cone types, and which can be observed on
the formula describing them in Fig.~\ref{fig:cone-type}, where \(\bar{k} = \{1,\ldots,k\}\):
\[
  \Cone^{k+1}(a,b,c) =
  (\Sigma(\op_{\bar{k}}\Cone^{k}))((a\s_{0}\frontcone^{1}(b)),c,b),
\]
 This leads us to consider, given
\begin{align*}
  \Gamma & \vdash m:\Cone^{k+1}(a_{0},b,c)\\
  \Gamma & \vdash n:\Cone^{k+1}(a_{1},c,d)
\end{align*}
 the term:
\[
  u = \Sigma(\op_{\bar{k}}\prescript{k}{}\cyls^{k}_{k-1}) (m,n)
\]
By induction, this term satisfies the following judgement:
\[
  \Gamma \vdash u : \Sigma(\op_{\bar k} \Cone^{k}) \left(
    (a_{0}\s_{0}\frontcone_{1}(b)) \s_{k} (a_{1}\s_{0}\frontcone_{1}(c)),d,b
  \right),
\]
The equalities from Proposition~\ref{prop:cone-type} imply that we have:
\begin{align*}
  \frontcone_{1}(c) &= \frontcone_{1}(d)
\end{align*}
Suppose that \(k\) is even, we introduce the interchanger \(j^-\) as the
coherence with type:
\[
  j^-(a,b,c) : (a\s_{k} b)\s_{0} c \to (a \s_{0} c) \s_{k} (b \s_{0} c)
\]
Define \(s\) to be the source of of the type
\(\Sigma(\op_{\bar{k}}\Cyl^{k})(a,b,c)\), so that we can define
\begin{align*}
  j_{\triangle,k}^- & = (s \uparrow\{a\})[j^-(a_{0},a_{1},\frontcone^{1}(b))] \\
  m \prescript{k+1}{}{\cs}^{k+1}_{k} n &= j^-_{\triangle,k} \s_{k+1} u
\end{align*}
which has the following type:
\begin{align*}
  & \Sigma(\op_{\bar k} \Cone^{k}) \left( (a_{0}\s_{k} a_{1}) \s_{0}
    \frontcone_{1}(b),d,b \right) \\
  & = \Cone^{k+1}(a_{0}\s_{k}a_{1},b,d)
\end{align*}
If \(k\) is odd, we introduce the interchanger \(j^+\) as the coherence with
type:
\[
  j^+(a,b,c) : (a \s_{0} c) \s_{k} (b \s_{0} c) \to (a\s_{k} b)\s_{0} c
\]
We define \(t\) to be the target of the type
\(\Sigma(\op_{\bar{k}}\Cyl^{k})(a,b,c)\), so that we can define
\begin{align*}
  j_{\triangle,k}^+ & = (t\uparrow
  \{a\})[j^-(a_{0},a_{1},\frontcone^{1}(b))]\\
  m \prescript{k+1}{}{\cs}^{k+1}_{k} n &= u \s_{k+1} j_{\triangle,k}^+
\end{align*}
which has type the following type:
\begin{align*}
  & \Sigma(\op_{\bar k} \Cone^{k}) \left( (a_{0}\s_{k} a_{1}) \s_{0}
    \frontcone_{1}(b),b,d \right) \\
  & = \Cone^{k+1}(a_{0}\s_{k}a_{1},b,c)
\end{align*}
